\documentclass[a4paper,10pt]{article}

\usepackage{mathtools}
\mathtoolsset{showonlyrefs=true}

\usepackage{booktabs}

\usepackage[colorlinks=true,bookmarks=true,citecolor=blue,
bookmarksnumbered=true,bookmarkstype=toc,linktocpage=true]{hyperref}

\usepackage[
margin=2.5cm]{geometry}
\usepackage[T1]{fontenc}
\usepackage{amsmath,amssymb,amsfonts,amsthm,mathrsfs,dsfont,verbatim}
\usepackage[english]{babel}

\usepackage{xcolor}

\newcommand {\wt}[1]{ \widetilde{#1}}
\newcommand {\wh}[1]{ \widehat{#1}}

\newcommand{\rone}{\mathbb{R}}
\renewcommand{\Re}{\rone}

\newcommand{\E}{\mathds{E}}
\renewcommand{\P}{\mathds{P}}

\newcommand{\prt}{\partial}
\renewcommand {\epsilon}{\varepsilon}

\newcommand{\RR}{\mathbb{R}}

\newcommand{\Ff}{\mathcal{F}}

\newcommand{\al}{\alpha}
\newcommand{\eps}{\varepsilon}

\newcommand{\ds}{\displaystyle}

 \newtheorem{thm}{Theorem}[section]
 
 \newtheorem{prp}{Proposition}[section]
 \newtheorem{cor}{Corollary}[section]
 
 \newtheorem{lem}{Lemma}[section]
 \newtheorem{asm}{Assumption}[section]
 \newtheorem{rem}{Remark}[section]
 \newtheorem{exmpl}{Example}[section]

\DeclareMathSymbol{\ophi}{\mathalpha}{letters}{"1E}

\newcommand{\be}{\begin{equation}}
\newcommand{\ee}{\end{equation}}
\newcommand{\ben}{\begin{equation*}}
\newcommand{\een}{\end{equation*}}

\newcommand{\ba}{\begin{aligned}}
\newcommand{\ea}{\end{aligned}}

\DeclareMathOperator{\Var}{Var}

\DeclareMathOperator{\sgn}{sgn}

\newfont{\cyrfnt}{wncyr10}
\def\J3{\cyrfnt{\rm \u{\cyrfnt I}}}
\def\j3{\cyrfnt{\rm \u{\cyrfnt i}}}

\usepackage[]{color}
\definecolor{DarkGreen}{rgb}{0.1,0.7,0.3}   

\definecolor{DarkGreen}{rgb}{0.1,0.7,0.3}   

\numberwithin{equation}{section}

\allowdisplaybreaks

\usepackage{xcolor}
\def\ds#1{\displaystyle{#1}}
\def\nn{\nonumber}

\def\rev#1{\textcolor{black}{#1}}

\def\nn{\nonumber}
\def\lp{L\'{e}vy process}
\def\lm{L\'{e}vy measure}

\newcommand{\tes}{\hat{\theta}_{n}}

\newcommand{\aes}{\hat{\alpha}_n}

\newcommand{\tz}{\theta_{0}}

\newcommand{\cip}{\xrightarrow{\P}}

\newcommand{\sumk}{\sum_{k=1}^{n}}

\def\lad{\mathds{L}}
\newcommand{\mbbr}{\mathbb{R}}
\newcommand{\mbbrp}{\mathbb{R}_{+}}

\newcommand{\mcf}{\mathcal{F}}
\newcommand{\wc}{\Rightarrow}

\newcommand{\Del}{\Delta}
\newcommand{\sig}{\sigma}

\begin{document}
\setlength{\baselineskip}{4.5mm}

\title{LAD estimation of locally stable SDE}

\author{%
	Oleksii M. Kulyk%
    \thanks{Wroclaw University of Science and Technology,
Wybrze\'ze Wyspia\'nskiego Str. 27, 50-370 Wroclaw, Poland
    \texttt{kulik.alex.m@gmail.com}}
    \and
    Hiroki Masuda%
    \thanks{Graduate School of Mathematical Sciences, University of Tokyo, 3-8-1 Komaba Meguro-ku Tokyo 153-8914, Japan.
    \texttt{hmasuda@ms.u-tokyo.ac.jp}}
   }

\date{\today}

\maketitle

\begin{abstract}
We prove the asymptotic mixed normality of the least absolute deviation (LAD) estimator for a locally $\al$-stable stochastic differential equation (SDE) observed at high frequency, where $\al\in(0,2)$. We investigate both ergodic and non-ergodic cases, where the terminal sampling time diverges or is fixed, respectively, under different sets of assumptions. The objective function for the LAD estimator is expressed in a fully explicit form without necessitating numerical integration, offering a significant computational advantage over the existing non-Gaussian stable quasi-likelihood approach.
\rev{We also give illustrative numerical results that support our theoretical finding.}
\end{abstract}

\section{Introduction}

The objective of this paper is the drift estimation of the stochastic differential equation (SDE)
\be\label{sde}
dX_t=a(\theta;X_t)\, dt+\sigma(X_{t-})dZ_t
\ee
on the basis of a discrete-time sample $(X_{t_{k,n}})_{k=0}^{n}$, where $t_{k,n}=kh_{n}$ with the sampling step size $h_{n}\to 0$ as $n\to\infty$, that is, we consider high-frequency sampling. We assume that the parametric drift coefficient $a(\theta;x)$ is known up to a finite-dimensional parameter $\theta\in\Theta\subset\RR^{m}$ while the scale coefficient $\sig(x)$ may be unknown, and that the driving process $Z$ is a \emph{pure jump} \emph{locally $\al$-stable} {\lp} for $\al<2$, to be specified in Section \ref{hm:sec_setup}.

To estimate the true value $\tz$ of $\theta$ from $(X_{t_{k,n}})_{k=0}^{n}$, we consider the Least Absolute Deviation (LAD) type estimator defined as any element $\tes$ minimizing the random function
\begin{equation}\label{lad.cf}
\lad_{n}(\theta):=\sumk V(X_{t_{k-1,n}}) \left|X_{t_{k,n}}-F_{h_{n}}(\theta; X_{t_{k-1,n}} )\right|,
\end{equation}
for suitable regressor function $F_h(\theta;x)$ and non-negative weight function $V(x)$. Apart from the weighting through $V(x)$, the choice of the loss function \eqref{lad.cf} corresponds to adopting the Laplace (double exponential) quasi-likelihood with constant scale.
The goal is to prove that the scaled LAD estimator $\sqrt{n}\,h_n^{1-1/\alpha}(\hat\theta_n-\tz)$ is asymptotically (mixed-)normally distributed; we have $\sqrt{n}\,h_n^{1-1/\alpha} \to \infty$ under the condition \eqref{sampling.design} given below. 

\rev{
The LAD estimator is preferable from both theoretical and practical points of view: once the weight function $V(x)$ is given, the random function \eqref{lad.cf} is fully explicit and does not involve any specific characteristic of the driving noise $Z$; we emphasize that the locally $\al$-stable property of $Z$ is implicit in the form of \eqref{lad.cf}, so that we can apply the LAD estimator without knowing the value of $\al$ beforehand. 
Besides its relatively simple form and computational advantages, the LAD estimation strategy has the following desirable features:
}

\begin{itemize}
\item Unlike the conventional Gaussian quasi-likelihood estimation (see \cite{Mas13as}), and as in the non-Gaussian stable quasi-likelihood estimation (see \cite{Mas19spa}, \cite{CleGlo20}, and the references therein), the LAD estimation allows us to deal with cases of fixed terminal sampling time (i.e. the terminal sampling time $t_{n,n}\equiv T$ is fixed);
\item Further, we can obtain a rate-optimal, namely $\sqrt{n}\,h_n^{1-1/\al}$-consistent estimator without specific information of the index $\al$ (see \cite{BroMas18}, \cite{IvaKulMas15}, \cite{Mas09jjss}, and also the references in \cite{CleGlo20} for the rate-optimality).
\end{itemize}
\rev{
These features are novel, and neither the Gaussian quasi-likelihood estimator (including the least-squares estimator) nor the non-Gaussian stable quasi-likelihood estimators in the literature can simultaneously achieve them.
Specifically, the Gaussian quasi-maximum-likelihood estimator \cite{Mas13as} is consistent and asymptotically normal only if $T_n\to\infty$ and $Z$ has sufficiently light tails, specifically, finite variance.
Furthermore, as in diffusion cases, maximum-likelihood-type estimators in the present setting are rarely available, except in very limited cases. Existing studies \cite{Mas19spa} and \cite{CleGlo20} (and references therein) have introduced non-Gaussian $\al$-stable quasi-likelihoods. Although these methods provide non-Gaussian counterparts to the conventional Gaussian quasi-likelihood, they inevitably require knowing or jointly estimating the true value of the local stability index $\al$, making numerical computation less stable and much more time-consuming.
}

The rate of convergence $\sqrt{n}\,h_n^{1-1/\al}$ reflects the non-Gaussian nature of the driving noise process through the index $\al\in(0,2)$.
This reveals that the $L^1$-loss case can more appropriately handle the small-time character of the driving noise than the $L^2$-loss (Gaussian-loss) case.

The theory and asymptotics of the LAD-type estimator and related Laplace quasi-likelihood-based inference have a long history: among others, we refer to \cite{BarBouDja17}, \cite{CYZZ08}, \cite{Obe82}, \cite{Pol91}, \cite{Zwa97}, \cite{ZhuLin11}, \cite{ZhuLin15} and the references therein.
However, to the best of our knowledge, our results are the first to establish rigorous asymptotics for the self-weighted LAD-type estimator in nonlinear locally stable L\'{e}vy-driven SDEs, whose weak solutions are observed at high frequency.
To our knowledge, \cite{Mas10ejs} is the only paper that studied a LAD-type estimation based on high-frequency sampling. It considered a self-weighted LAD estimation of a class of L\'{e}vy-driven ergodic Ornstein-Uhlenbeck processes under the ergodicity (i.e. $a(\theta;x)=\theta_{1}-\theta_{2}x$ with $\theta_{1}\in\mbbr$ and $\theta_{2}>0$, with $\sig(x)$ being a constant) and the large-time asymptotics where $t_{n,n}\to\infty$ as $n\to\infty$. The domain of applicability of the result given in this paper is much broader, making it possible to deal with the process $X$ with the non-linearity of $a(\theta;x)$, the skewness of the noise distribution $\mathcal{L}(Z_1)$, and the non-ergodicity in a unified manner.

\medskip

The rest of this paper is organized as follows.
Section~\ref{hm:sec_setup} describes the model setup and the basic assumptions.
Section~\ref{sec:main.results} presents the main results, namely the asymptotic (mixed) normality of the LAD estimator; their proofs are given in Sections~\ref{sec:proof.I} and~\ref{sec_proof.II}.
In Section~\ref{hm:sec_ceAV}, we construct a consistent estimator of the (possibly random) asymptotic covariance matrix of the LAD estimator.
Appendices~\ref{sA1}, \ref{sB}, and~\ref{hm:sec_moments} collect some technical results used in the proofs of the main results.
\rev{Finally, Section~\ref{hm:sec_simulations} reports numerical experiments.}

\medskip

\paragraph{Conventions.} 
We denote by $C, c$ generic positive constants whose values can vary at each appearance. When these constants depend on additional parameters (e.g. the truncation level $R$), we may write $c_R, C_R$.
For any sequences $(\xi_{n})$ and $(\zeta_{n})$ of nonnegative random variables, we write $\xi_{n}\lesssim \zeta_{n}$ if $\xi_{n}/\zeta_{n}\le C$ a.s. for every $n$ large enough.
We denote by $\frac{\partial}{\partial x}$ the partial derivative operator with respect to $x$, and by $\nabla_\theta$ the gradient operator with respect to $\theta$. For a matrix $A$, we write $A^{\otimes 2}=AA^\top$ with $A^\top$ denoting the transpose of $A$.

\section{Setup and assumptions}
\label{hm:sec_setup}


\rev{We consider the partly parametrized family of SDEs \eqref{sde} assuming  that:}
\begin{itemize}
\item The parametrized drift coefficient $a(\theta;x)$ is known up to a finite-dimensional parameter $\theta\in\Theta$, where $\Theta$ is a bounded 
domain in $\RR^{m}$;
\item The scale coefficient $\sig(x)$ may be unknown;
\item The driving process $Z$ is assumed to be a \emph{pure jump locally $\al$-stable} {\lp} with $\al\in(0,2)$, which we will specify later.
\end{itemize}
\rev{ The sampling step size $h_n$ satisfies, that for $T_n:=t_{n,n}=nh_n$,
\begin{equation}\label{sampling.design}
\liminf_{n\to\infty}T_{n} >0
\end{equation}
 (the terminal sampling time $T_n$ may or may not diverge).}

 \rev{ The particular statistical model is specified as follows. Our assumptions on the noise and the coefficients of the SDE \eqref{sde}  will guarantee that the weak solution to \eqref{sde} is uniquely defined for every $\theta\in \Theta$. We consider the statistical experiment $(\Omega,\mathcal{F},(\mathcal{F}_{t})_{t\in\mbbrp},\{\mathds{P}_{\theta}, \theta\in \Theta\}),$ where the 
 filtered space is canonical ($\Omega=\mathbb{D}([0,\infty))$,  $(\mathcal{F}_{t})_{t\in[0,\infty)}$ is the corresponding natural filtration, and  $X_t(\omega)=\omega(t), \omega\in \Omega, t\geq 0$) and $\P_\theta$ is the distribution in $\mathbb{D}([0,\infty))$ of the solution to \eqref{sde}. We note that, in such a canonical setting, for any $\theta\in \Theta$ there exists a L\'evy process $Z$ on the probability space $(\Omega,\mathcal{F},(\mathcal{F}_{t})_{t\in\mbbrp},\mathds{P}_{\theta})$ such that \eqref{sde} holds true with this process and the canonical process $X$; see Appendix \ref{sB2} for details. In what follows, we assume the true value $\theta_0$ to be fixed and the filtered probability space $(\Omega,\mathcal{F},(\mathcal{F}_{t})_{t\in\mbbrp},\mathds{P})$ and the processes $X,Z$ to be given in the canonical way specified above with $\theta=\theta_0.$ Whenever this does not cause ambiguity in the notation, we omit the index $\theta_0$ and write $\P, \E, \dots$ instead of $\P_{\theta_0}, \E_{\theta_0}, \dots$. }  Although we treat here only a single value $\theta_{0}\in\Theta$ of $\theta$ for brevity, with a trivial modification of the regularity conditions, many of our forthcoming results hold uniformly for \rev{$\theta_0$ taking values in a compact subset of} $\Theta$. 

Recall that, by the L\'evy-Khinchin representation, the characteristic function of a L\'evy process without a diffusion component has the form
$$
\E[e^{i\xi Z_t}]=e^{-t\psi(\xi)},
$$
where the \emph{L\'evy exponent} $\psi(\xi)$ has representation
\be\label{psi}
\psi(\xi)=-ib\xi+\int_{\Re}(1-e^{i\xi u}+i\xi u1_{|u|\le                          1})\mu(du);
\ee
here and below we denote by $\mu(du)$ the L\'evy measure of the process. 
We refer to \cite{Sat99} for a systematic account of {\lp es}. 
In this paper, we call the L\'evy process \emph{pure jump} if
\be\label{psi_pj}
\psi(\xi)=\mathrm{P.V.}\int_{\Re}(1-e^{i\xi u})\mu(du)=\lim_{\eps\to 0+}\int_{|u|>\eps}(1-e^{i\xi u})\mu(du),
\ee
that is, the `external drift' part $ -ib\xi$ in the representation \eqref{psi} cancels, in the Principal Value sense,  the `internal drift' which comes from the compensator part in the integral. If \eqref{psi_pj} holds true, the L\'evy process $Z$ can be obtained as a weak limit for $\varepsilon\to 0+$ of the compound Poisson processes with the L\'evy exponents
$$
\psi_{(\eps)}(\xi)=\int_{|u|>\eps}(1-e^{i\xi u})\mu(du),
$$
and thus can be intuitively understood to be free both from the drift and compensator terms; see also Remark \ref{r23} below. It is essential for our estimation procedure that the noise is `free from a drift' in the above sense: \rev{otherwise, the (unknown) stochastic term $\sigma(X_{t-})dZ_t$ in \eqref{sde} would produce an (also unknown) extra drift term, which would} interfere with the parameterized drift coefficient $a(\theta;X_t)$ \rev{thus making it difficult, if even possible, to estimate  $\theta$}.
A sufficient but not necessary condition for this property to hold is that the {\lp} $Z$ is symmetric.

Furthermore, we assume the driving process to be \emph{locally $\alpha$-stable} for $\al\in(0,2)$, meaning that its L\'evy measure has the form
\be\label{LM}
\mu(du)=\mu_\alpha(du)+\nu(du),
\ee
where the `principal' part 
$$
\mu_\al(du)=\frac{c_\al}{|u|^{\alpha+1}}\, du, \quad c_\al=\left\{
  \begin{array}{ll}\al\left(2\Gamma(1-\alpha)\cos\frac{\pi\alpha}{2}\right)^{-1}, &\alpha\not=1,\\[2mm]
\frac{1}{\pi}, &\alpha=1
\end{array}\right.
$$
is the L\'evy measure of the standard symmetric $\al$-stable process with the L\'evy exponent
$$
\psi_\al(\xi)=|\xi|^\alpha.
$$
The `nuisance' part $\nu(du)$ is allowed to be a signed measure and is assumed to have a small-jump activity strictly lower than the `principal' one. The latter means that its total variation should have the \emph{Blumenthal-Getoor index} $\beta<\alpha$:
\be\label{BG}
|\nu|(\{u:|u|\geq \eps\})\leq C\eps^{-\beta}, \quad \eps\in (0,1].
\ee

The principal assumptions on the model \eqref{sde} are given below.

\begin{asm}[Properties of coefficients and noise]\label{Ass1}\indent
\begin{enumerate}
  \item Drift coefficient $a(\theta;x)$ satisfies the finite range H\"older condition in $x$ with exponent $\eta\in(0,1]$:
  \be\label{ass1}
  |a(\theta; x)-a(\theta; y)|\leq C|x-y|^\eta, \quad |x-y|\leq 1, \quad \theta\in \Theta.
  \ee
  Moreover, $a(\theta;x)$ is differentiable in $\theta$ for each $x$.
  \item Jump coefficient $\sigma(x)$ is bounded, separated from $0$, and satisfies the finite range H\"older condition in $x$ with exponent $\zeta\in(0,1]$:
  \be\label{ass2}
  |\sigma(x)-\sigma(y)|\leq C|x-y|^\zeta, \quad |x-y|\leq 1.
  \ee
  \item The L\'evy exponent of the L\'evy process $Z$ is given by  \eqref{psi_pj}, where  the L\'evy measure $\mu(du)$ satisfies \eqref{LM} and \eqref{BG}.
   \item The \emph{balance condition} holds:
   \be\label{balance}
   \alpha+\eta>1.
   \ee
\end{enumerate}
\end{asm}

In particular,  Assumption \ref{Ass1} ensures the existence \rev{and uniqueness} of a \rev{weak} solution to \eqref{sde} \rev{ according to  \cite[Theorem 3.3]{AK}; see Appendix \ref{sA1} for related details. We emphasize that the balance condition \eqref{balance} is known to be close to a necessary condition for \eqref{sde} to have a unique weak solution, and thus is crucial for the process $X$ to be well-defined by the equation \eqref{sde}. This observation dates back
to \cite{TTW74}, where a natural example of an SDE driven by a symmetric additive $\alpha$-stable process and having two different weak solutions was constructed.}

We will also require a technical balance condition between the regularity of the `nuisance' part $\nu(du)$ of the L\'evy measure and the stability index $\alpha$. Denote $B_r(x)=\{y:|y-x|\leq r\}$, $x\in \mathbb{R}$, $r\geq 0$.
%

\begin{asm}\label{AssReg}\indent
There exist constants $\kappa\geq 0, \beta'>0, C>0$ such that
$$
1-\kappa<\beta'<\alpha
$$
and
\be\label{dimension}
|\nu|(B_r(z))\leq Cr^{\kappa}|z|^{-\beta'-\kappa}, \quad r\leq \frac12|z|, \quad z\in \mathbb{R}.
\ee
\end{asm}
\begin{rem} Condition \eqref{dimension} has the same spirit as the notion of a `$\kappa$-measure', which requires a measure of a ball of radius $r$ to be bounded by $Cr^\kappa$, e.g. \cite{KalSzt15}. The additional factor $|z|^{-\beta'-\kappa}$ appears here because we extend this notion to L\'evy measures, which may `explode' near the origin. It is easy to verify that, for a measure satisfying  \eqref{dimension},  
the index condition \eqref{BG} holds with $\beta=\beta'$. Therefore, such an extension is well-adjusted with the basic setup.
\end{rem}

To illustrate Assumption \ref{AssReg}, let us consider two natural examples.

\begin{exmpl} 
\label{ex_A2.4}
Let $\kappa=0$, which means that there is no actual regularity limitation on the measure $\nu(du)$. The bound \eqref{dimension} for $\kappa=0$ is equivalent to the following: \eqref{BG}  holds for $\beta=\beta'$ and \emph{for all} $\eps>0$. With a minor rearrangement, we can formulate the following  equivalent condition for Assumption \ref{AssReg} to hold with $\kappa=0$: $\alpha>1,$ the ``small-jump intensity'' condition \eqref{BG} holds with some $\beta<\alpha$, and the following ``tail condition'' holds with some $\beta''>1$:
\be\label{tail}
|\nu|(\{u:|u|\geq r\})\leq Cr^{-\beta''}, \quad r\geq 1.
\ee
Indeed, we can assume without loss of generality that $\beta''<\alpha$, and then combining \eqref{BG} and \eqref{tail} we see that Assumption \ref{AssReg} holds true for $\kappa=0$ with $ \beta'=\max\{\beta, \beta''\}$. Note that  $\beta'$, unlike $\beta$, is not involved in any other assumption, such as the statistic stability Assumption \ref{AssStab} below, hence we are free to choose it close to $\alpha$.
\end{exmpl}

\begin{exmpl} 
\label{ex_A2.4-2}
Let $\kappa=1$ so that there is no limitation on $\alpha>0$. Considerations similar to those in Example \ref{ex_A2.4} lead to  the following  equivalent condition for Assumption \ref{AssReg} to hold true with $\kappa=1$: $\alpha\in (0,2)$ is arbitrary and $\nu(du)$ admits a density $\mathfrak{n}(u)$ with respect to the Lebesgue measure such that
\begin{equation}\label{ex_A2.4-2-eq1}
|\mathfrak{n}(u)|\leq C|u|^{-\beta-1}1_{|u|\leq 1}+C|u|^{-\beta''-1}1_{|u|>1}, \quad \beta<\alpha, \quad \beta''>0.
\end{equation}
This is a very mild and easy-to-verify condition when the {\lm} $\mu(du)$ admits an explicit density; recall also the condition $\beta<\al/2$ mentioned in Remarks 
\ref{hm:rem_A2.3}.
See also Example \ref{hm:rate-example}.
\end{exmpl}

\rev{Next, we impose an assumption on the discretization step $h_n.$} 

\begin{asm}[Statistic stability assumption]\label{AssStab}\indent
There exists positive constant $\delta$ such that
$$
\delta<\delta_{drift}:=\frac{\alpha+\eta-1}{\alpha}, \qquad \delta<\delta_\sigma:=\frac{\zeta}{\alpha}, \qquad  \delta<\delta_\nu:=\frac{\alpha-\beta}{\alpha}
$$
and that
\be\label{h_step}
nh_n^{2\delta}\to 0.
\ee
\end{asm}

\begin{rem}\label{r23}
\label{hm:rem_A2.3}
Under \eqref{sampling.design}, the condition \eqref{h_step} yields that
$$
\delta>\frac12.
$$
This actually puts stronger assumptions on $\eta, \zeta$ and $\beta$ than those imposed primarily:
$$
 \frac{\alpha}2+\eta>1, \quad \zeta>\frac{\alpha}2, \quad \beta<\frac{\alpha}2.
$$
The first two inequalities hold trivially in the regular case  $\eta=\zeta=1$. The third one yields  
\begin{equation}
    \beta<\frac{\al}{2}<1,
\end{equation}
which in particular requires that the nuisance part $|\nu|(du)$ (recall \eqref{BG}) is finite-variation. Hence the assumption \eqref{psi_pj} for the L\'evy process to be `pure jump' is equivalent to the following: in \eqref{psi}, we have
\be\label{pur_jump}
b=\int_{|u|\le 1}u\,\nu(du).
\ee
\end{rem}

Next, let us proceed with the assumptions on the \rev{regression function} involved in the LAD-type estimator $\hat\theta_n$; recall \eqref{lad.cf}. Consider the Cauchy problem
\be\label{ODE}
df_t=a(\theta;f_t)dt, \quad f_0=x.
\ee
Assumption \ref{Ass1} yields that this problem has a solution ${f}_t(\theta; x)$. This solution may fail to be unique; however, it is easy to show that any two solutions 
$f^{1}_t(\theta; x)$ and $f^{2}_t(\theta; x)$ satisfy
\be\label{ODE_sol}
|f^{1}_t(\theta; x)-f^{2}_t(\theta; x)|\leq C t^{\frac{1}{1-\eta}}, \quad t\in [0,1].
\ee
We call any locally bounded function $W:\Re\to [0, \infty)$ a \emph{weight function}; typical example here is 
$$
W(x)=C(1+|x|^p), \quad p\geq0,
$$
for some $C>0$ whose specific value does not matter.

\begin{asm}[Properties of the regressor]\label{Ass2}
The following conditions hold for $x\in \mathbb{R}, \in \Theta, t\in (0,1]$ with a certain constant $\gamma>0$ and weight function $W(x)$.
\begin{enumerate}
\item Regressor function $F_t(\theta;x)$ satisfies
\be\label{F3}
|F_t(\theta;x)-x-t a(\theta;x)|\leq t^{1+\gamma} W(x).
\ee

\item
There exists a constant
\be\label{delta_regr}
\delta_{regr}>\delta,
\ee
where $\delta>0$ is given in Assumption \ref{AssStab}, such that 
for some solution ${f}_t(\theta; x)$ to \eqref{ODE},
\be\label{F0}
|F_t(\theta; x)-f_t(\theta; x)|\leq t^{1/\alpha+\delta_{regr}}W(x).
\ee

\item Regressor function $F_t(\theta;x)$ is differentiable in $\theta$ and the following bounds hold:  \be\label{F1}
|\nabla_{\theta} F_t(\theta;x)|\leq t W(x),
\ee
\be\label{F1a}
|\nabla_{\theta} F_t(\theta;x)-t\nabla_\theta a(\theta;x)|\leq t^{1+\gamma} W(x),
\ee
\be\label{F5}\ba
|F_t(\theta;x)&-F_t(\theta_0;x)-t\nabla_\theta a(\theta_0;x)\cdot(\theta-\theta_0)|
\leq t|\theta-\theta_0|\Big(|\theta-\theta_0|^\gamma+ t^\gamma\Big) W(x).
\ea
\ee

\end{enumerate}
\end{asm}

\begin{rem}\label{remAss2} By the balance condition \eqref{balance}, we have 
$$
\frac{1}{\alpha}+\delta_{drift}=\frac{\alpha+\eta}{\alpha}=\frac{\alpha+\eta(1-\alpha-\eta)}{\alpha(1-\eta)}<\frac{1}{1-\eta}.
$$
    Without loss of generality, we can and will assume that $\delta_{regr}\leq \delta_{drift}$. Then, by \eqref{ODE_sol} the assumption \eqref{F0} holds for \textit{any} solution to \eqref{ODE}. Moreover, it still holds with a solution $f_t(\theta;x)$ replaced by the approximate solution  $\mathfrak{f}_t(\theta;x)$ involved in the semi-explicit representation of the transition density $p_t(\theta;x,y)$ of the process; see Appendix \ref{sA1} and the discussions therein. 
\end{rem}

We illustrate the above assumption with several examples of \rev{practically useful} regressors. 

\begin{exmpl}[Euler scheme]
\label{exF1}
One natural choice of the regressor is the one from  the classical Euler scheme:
$$
F_t(\theta;x)=x+ta(\theta;x).
$$
For this choice, the conditions \eqref{F3}, \eqref{F1a} hold true trivially and the conditions \eqref{F1}, \eqref{F5} hold true whenever
\be\label{ass_Eu}
|\nabla_\theta a(\theta,x)|\leq W(x), \quad |\nabla_\theta a(\theta,x)-\nabla_\theta a(\theta_0,x)|\leq |\theta-\theta_0|^\gamma W(x).
\ee
To verify \eqref{F0}, we assume that Assumption \ref{Ass1} holds true. By the first condition in this assumption, $a(\theta;x)$ has at most linear growth uniformly in $x$ and thus 
$$
|f_t(\theta;x)-x|\leq \int_0^t|a(\theta;f_s(\theta;x))|\,ds\leq Ct(1+|x|), 
$$
$$
|F_t(\theta;x)-f_t(\theta;x)|\leq \int_0^t|a(\theta;x)-a(\theta;f_s(\theta;x))|\,ds\leq Ct^{1+\eta}(1+|x|).
$$
Therefore \eqref{F0} holds true with 
$$
\delta_{regr}=1+\eta-\frac{1}{\alpha}
$$
and $W(x)=C(1+|x|^p)$ whenever $p\geq2$ is such that \eqref{ass_Eu} holds true.
\end{exmpl}

\begin{rem} \rev{By Assumption \ref{AssStab} and Remark \ref{r23}, we need} 
$$
\delta_{regr}>\frac{1}{2}.
$$
This gives the following lower bound for $\alpha$ for the Euler scheme-based regressor to be applicable in our setting:
\be\label{cond1}
\alpha>\frac{2}{1+2\eta},
\ee
which becomes $\al>2/3$ in the regular case where $\eta=1$.
\end{rem}

\begin{exmpl}[Improved Euler schemes]
\label{exF2}
Continued from Example \ref{exF1}, we can show that additional regularity of $a(\theta;x)$ allows one to weaken the bound \eqref{cond1} for $\alpha$ by using improved Euler schemes based on the Taylor expansion. For instance, if $\frac{\partial}{\partial x}a(\theta;\cdot)$ is 
bounded and H\"older continuous with the index $\eta$ uniformly in $\theta$, then we can put 
$$
F_t(\theta;x)=x+ta(\theta;x)+\frac{1}{2}t^2 a(\theta;x)\frac{\partial}{\partial x}a(\theta;x).
$$
Indeed, then direct computation gives for  $t\le 1$
$$
|F_t(\theta;x)-f_t(\theta;x)|\leq C (1+|x|^{1+\eta})t^{2+\eta},
$$
and therefore \eqref{F0} holds true with 
$$
\delta_{regr}=2+\eta-\frac{1}{\alpha}
$$
and the same $W(x)$ as in Example \ref{exF1}.
This gives the following lower bound for $\alpha$:
\be\label{cond2}
\alpha>\frac{2}{3+2\eta}.
\ee
In addition, \eqref{F3} holds true trivially and conditions \eqref{F1}, \eqref{F1a}, \eqref{F5} hold true under \eqref{ass_Eu} combined with 
\be\label{ass_Eu_2}
\left|\nabla_\theta \frac{\partial}{\partial x}a(\theta,x)\right|\leq W(x).
\ee
\end{exmpl}

\rev{These two examples show that one natural way to construct a regressor for the LAD type estimator is to use an approximate solution to the Cauchy problem \eqref{Cauchy} for the drift coefficient. The actual accuracy of this approximation then will impose a lower bound for possible values of the stability index $\alpha$, and the choice of the approximate solution should be adjusted with the actual range $(\alpha_0, 2)$ where this index is supposed to vary. It is visible that, for any given $\alpha_0\in (0,2)$,  one can construct the LAD estimator suitable for $\alpha\in (\alpha_0, 2)$ by choosing a Taylor expansion-based improved Euler scheme of a sufficiently high order. Note that such an improvement in the accuracy of the construction of the regressor would require correspondingly higher regularity assumptions on the drift coefficient, and also greater computational demands in numerical optimization.}

\rev{This complication can be avoided in some particular important cases}, where the solution to \eqref{ODE} can be given explicitly and used as the regressor. 

\begin{exmpl}[Linear ODE]
\label{ex:linear.drift}
Let $a(\theta;x)=\theta_1+\theta_2x$ for $\theta=(\theta_1, \theta_2)$. Then, 
$$
f_t(\theta;x)=e^{\theta_2 t}x+\theta_1 t\psi(\theta_2 t)
$$
for the smooth function $\psi(x)=(e^x-1)/x$. 
Taking $F_t(\theta;x)=f_t(\theta;x)$, we get \eqref{F0} with arbitrarily large $\delta_{regr}$. It is also easy to show that conditions \eqref{F3}, \eqref{F1}, \eqref{F1a}, \eqref{F5} hold true with $W(x)=C(1+|x|)$.
\end{exmpl}

\begin{exmpl}[Bernoulli's ODE] 
\label{ex:Bernoulli.ODE.drift}
Let $a(\theta;x)=\theta_1 x^{\langle\kappa\rangle}+\theta_2 x$ for $\theta=(\theta_1, \theta_2)$, where $x^{\langle\kappa\rangle}:=|x|^\kappa\sgn(x)$. Then the function   
$$
f_t(\theta;x)=\left(e^{(1-\kappa)\theta_2 t}|x|^{1-\kappa}+(1-\kappa)\theta_1 t\,\psi((1-\kappa)\theta_2 \,t)\right)_+^{\frac{1}{1-\kappa}}\sgn(x), 
$$
solves \eqref{ODE}, note that for $\kappa<1$ and $\theta_1>0$ this solution is not unique. 
Taking $F_t(\theta;x)=f_t(\theta;x)$, we get \eqref{F0} with arbitrarily large $\delta_{regr}$. Assuming $\kappa<1$, which is required for Assumption \ref{Ass1} to be satisfied, it is easy to show that conditions \eqref{F3}, \eqref{F1}, \eqref{F1a}, \eqref{F5} hold with $W(x)=C(1+|x|)$.  
\end{exmpl}

\section{Main results} 
\label{sec:main.results}

We study the asymptotic distribution of the LAD estimator $\tes$, defined by any minimizer of the random function \eqref{lad.cf} for a weight function $V(x)$:
\begin{equation}\nn
\lad_{n}(\theta):=\sumk V(X_{t_{k-1,n}})
\left|X_{t_{k,n}}-F_{h_{n}}(\theta; X_{t_{k-1,n}} )\right|.
\end{equation}
We will formulate our main results, separating two principal cases: finite observation horizon $T_n=nh_n\to T\in (0, \infty)$ and infinite observation horizon $T_n\to \infty$.

\subsection{Finite observation horizon}

In this case, the observed trajectory of $X$ is a.s. bounded, namely $\sup_{t\le T}|X_t(\omega)|\le  C(\omega)$ for an a.s. finite random variable $C(\omega)$,
hence the weight $V(x)$ is not important for the entire consideration. Thus, for simplicity, we take  $V(x)\equiv 1$ in this case. Denote
$$
\Gamma_0=\Gamma_0(\theta_0)=
\frac1T \int_0^T\Big(\nabla_\theta a(\theta_0; X_{t})\Big)^{\otimes 2}\, dt,
$$
$$
\Sigma_0=\Sigma_0(\theta_0)=
\frac1T \int_0^T\frac{1}{\sigma(X_{t})}\Big(\nabla_\theta a(\theta_0; X_{t})\Big)^{\otimes 2}\, dt,
$$
\rev{here and below, an integral of a matrix-valued function is understood component-wise.}
For an $\Ff$-measurable non-negative definite $m\times m$-random matrix $A$, we will use the symbol $MN(0,A(\omega))$ to denote the distribution of a random vector of the form $A^{1/2}\xi$ with a standard $m$-dimensional Gaussian random vector $\xi$ defined on an extended probability space and independent of $\Ff$. Further, we denote by
\begin{equation}\label{hm:def_phi.al}
\phi_{\al}(x)=\frac1{2\pi}\int_{-\infty}^{\infty}e^{-ix\xi-|\xi|^\al}\, d\xi    
\end{equation}
the $\al$-stable density with the characteristic function $\xi\mapsto e^{-|\xi|^\al}$, and by $\wc$ the weak convergence.

\begin{thm}\label{t1} Let Assumptions \ref{Ass1} to \ref{AssStab} hold. Assume also the following identifiability conditions:
\begin{itemize}
  \item (Global identifiability): for any $ \ theta\neq\theta_0$,
    $$
\Lambda_0(\theta)=\int_0^T\Big(a(\theta; X_{t})-a(\theta_0; X_{t})\Big)^2\, dt>0\quad \hbox{a.s.}
$$
  \item (Local identifiability):
  $$
  \P(\Gamma_0\hbox{ is positive definite})=1.
  $$
  \end{itemize}
  Then, the LAD estimator satisfies
  \begin{equation}
\sqrt{n} \, h_n^{1-1/\alpha}(\tes-\tz) \rev{ \wc_{\P_{\theta_0}}} MN\left(0,\,\left(2\phi_{\al}(0)\right)^{-2}\Sigma_{0}^{-1}\Gamma_{0}\Sigma_{0}^{-1}\right).
\nonumber
\end{equation}
\end{thm}

\rev{
\begin{rem}
    Note that 
    \be\label{phi_al0}
    \phi_\al(0) = \frac{1}{\pi}\int_0^\infty \exp(-\xi^\al)d\xi = \frac{1}{\pi}\Gamma\left(1+\al^{-1}\right).
\ee
\end{rem}
}

\rev{\begin{rem} In the finite observation horizon case we have $h_n\sim n^{-1},$ hence the normalization 
$$
\sqrt{n} \, h_n^{1-1/\alpha}\sim n^{1/2-1/\alpha}\to \infty, \quad n\to \infty.
$$
\end{rem}}

\subsection{Infinite observation horizon}

When $T_n\to \infty$, we further assume the following conditions.

\begin{asm}[Drift dissipation and tail moments]\label{A_diss}\ There exists a constant $\kappa>-1$ such that
$$
\limsup_{|x|\to \infty} \frac{a(\theta_0,x)\sgn x}{|x|^\kappa}<0.
$$
In addition,
$$
\int_{|u|>1}|u|^{q}\mu(\rev{du})<\infty
$$
for some $q>0$ with
$$
\kappa+q>1.
$$
\end{asm}

This assumption ensures that the process $X$ is ergodic; 
see Section \ref{sB2.e}. Denote by $\pi(\theta_0, dx)$ the corresponding unique invariant probability measure, and put
$$
\Gamma_0=\Gamma_0(\theta_0)=\int_{\mathbb{R}}V(x)^2\Big(\nabla_\theta a(\theta_0; x)\Big)^{\otimes 2}\, \pi(\theta_0, dx),
$$
$$
\Sigma_0=\Sigma_0(\theta_0)=\int_{\mathbb{R}}\frac{V(x)}{\sigma(x)}\Big(\nabla_\theta a(\theta_0; x)\Big)^{\otimes 2}\, \pi(\theta_0, dx).
$$
Unlike the case of a finite observation horizon, we must now address the integrability. \rev{Recall that $\E$ denotes the expectation with respect to $\P=\P_{\theta_0},$ the law of the solution to \eqref{sde} with the true value $\theta=\theta_0$.} 

\begin{asm}[Integrability condition for weights]\label{A_moment} There exists $p>1$ such that
$$
\sup_{t\geq 0}\E\left[ V(X_t)^p W(X_t)^{2p} \right] <\infty,
$$
where $W(x)$ is the weight function given in Assumption \ref{Ass2}.
\end{asm}

\begin{rem}[On the choice of $V(x)$]\label{rem31} Assumption \ref{A_moment} exhibits the role of the function $V$: while the weight $W$ from Assumption \ref{Ass2} on the drift coefficient and regressor may be growing, the function $V$ should be decaying at $\infty$ sufficiently fast to ensure that $V(X_t)W(X_t)^{2}$ belongs to $L_{1+\eps}$ uniformly in $t>0$. Using Assumption \ref{A_diss}, one can explicitly choose $V$ in the natural case $W(x)=(1+|x|)^{p_W}$, $p_{W}\geq 0$. Indeed, by \cite{KulPav21}, Assumption \ref{A_diss} provides that, for any $p_X<q+\kappa-1$,
$$
\sup_{t\geq 0}\E[|X_t|^{p_X}] \leq C+|X_0|^{p_X}.
$$
Thus $V(x)=(1+|x|)^{-p_V}$ satisfies Assumption \ref{A_moment} whenever 
$$
2p_W-p_V<q+\kappa-1 \ \Longleftrightarrow \ p_V>2p_W+1-q-\kappa.
$$
Since $V$ is meant to dump the growth of the weight function $W$, we will assume furthermore that $V$ is bounded.       
We could set $V(x)\equiv 1$ from the beginning if $Z$ can be supposed to be light-tailed, such as $\int_{|u|> 1}|u|^K \mu(du)<\infty$ for any $K>0$.
\end{rem}

\begin{thm}\label{t2} Let Assumptions \ref{Ass1} to \ref{AssStab}, \ref{A_diss}, and \ref{A_moment} hold. Assume also the following identifiability conditions:
\begin{itemize}
  \item (Global identifiability): for any $\theta\not=\theta_0$,
    $$
\Lambda_0(\theta)=\int_{\mathbb{R}}V(x)^2\Big(a(\theta; x)-a(\theta_0; x)\Big)^{2}\, \pi(\theta_0, dx)>0;
$$
  \item (Local identifiability): $\Gamma_0$  is positive definite.
   \end{itemize}
  Then, the LAD estimator satisfies
  \begin{equation}
\sqrt{n}\,h_n^{1-1/\alpha}(\tes-\tz) \rev{ \wc_{\P_{\theta_0}}} N\left(0,\,\left(2\phi_{\al}(0)\right)^{-2}\Sigma_{0}^{-1}\Gamma_{0}\Sigma_{0}^{-1}\right).
\nonumber
\end{equation}
\end{thm}

\begin{rem}
The convergence rate $\sqrt{n}h_n^{1-1/\alpha}$ is known to be optimal in some situations. Indeed, the local asymptotic normality (LAN) has been proved for several locally stable L\'{e}vy processes. Concerning the symmetric stable L\'{e}vy process, we refer to \cite{Mas09jjss} for the degenerate LAN with diagonal norming and \cite{BroMas18} for the non-degenerate LAN with asymmetric non-diagonal norming. For a general locally stable L\'{e}vy process with known index, \cite{IvaKulMas15} derived the LAN property
\footnote{There was a minor mistake about the asymptotic covariance matrix in \cite[Theorem 2.1]{IvaKulMas15}: the off-diagonal element should be non-null in case of asymmetric jumps.}. Some case studies, including other specific locally stable L\'{e}vy processes, can be found in \cite{Mas15LM}. As for related (locally) stable SDE models, \cite{CleGlo15} and \cite{CleGloNgu19} studied the local asymptotic mixed normality (LAMN) when the index is known.
\end{rem}

\rev{\begin{rem}  In the important special case $h_n \sim n^{-\gamma}$ ($\gamma>0$), the normalization 
$$
\sqrt{n} \, h_n^{1-1/\alpha}\sim n^{1/2-\gamma+\gamma/\alpha}.
$$
By \eqref{sampling.design}, for the index $\gamma$ we have an upper bound  $\gamma\leq 1,$ which becomes $\gamma<1$ in the infinite observation horizon setting. Assumptions \ref{AssStab} and \ref{Ass2}
impose a lower bound
$$
\gamma>\gamma_*:=\frac{1}{2}\max\left\{\frac{\alpha}{\alpha+\eta-1}, \frac{\alpha}{\zeta}, \frac{\alpha}{\alpha-\beta}, \frac{1}{\delta_{regr}}\right\}.
$$
Note that 
$$
\rho(\gamma):=\frac{1}{2}-\gamma+\frac{\gamma}{\alpha}
$$
is an increasing (resp. decreasing) function of $\gamma$ for $\alpha\in(0,1)$ (resp. $\alpha\in (1,2)$). Hence, to increase the convergence rate $n^{\rho(\gamma)}$ for the LAD estimator one may
choose $\gamma$ close to $1$ in the first case and to $\gamma_*$ in the second case, thus obtaining the convergence rate index $\rho(\gamma)$ close to 
$$
\rho_*:=
\begin{cases} \frac{1}{\alpha}-\frac{1}{2}, &\alpha\in (0,1),\\[2mm]
\min\left\{\frac{\eta}{2(\alpha+\eta-1)},\frac{1+\zeta-\alpha}{2\zeta}, \frac{1-\beta}{2(\alpha-\beta)},\frac{1+\alpha \delta_{regr}-\alpha}{2\alpha\delta_{regr}}\right\},
&\alpha\in (0,2).
\end{cases}
$$
In the boundary (Cauchy) case $\alpha=1,$ the convergence rate index does not depend on $\gamma$ and equals $\frac{1}{2}.$ 
\end{rem}
}

\section{Proofs, I: Consistency at sub-optimal rate}
\label{sec:proof.I}

In what follows, we denote
$$
r_n= r_n(\al) = \sqrt{n}\,h_n^{1-1/\alpha}.
$$
In this preparatory section, we prove that the LAD-type estimator is consistent with the rate $r_n^{-\varsigma}$ for any $\varsigma\in (0,1)$:
\be\label{cons_rate}
\hat \theta_n-\theta_0=O_\P(r_n^{-\varsigma}),
\ee
\rev{here and below we write $\xi_n=O_\P(q_n)$ or $\xi_n=o_\P(q_n)$ if the sequence $\{q_n^{-1}\xi_n\}$ is bounded (resp. converges to 0) in probability $\P=\P_{\theta_0}$.}
This statement is weaker than the main statement, as the latter one yields \eqref{cons_rate} with $\varsigma=1$. However, we will need this weaker statement in order to achieve the optimal rate $r_n^{-1}$, and its proof is already quite technically involved. Therefore, for the benefit of the reader, we discuss this proof separately, hoping that this will make the entire argument easier to follow.

We separate the proof of \eqref{cons_rate} into several steps, gradually improving the rate of consistency. In each of the steps, we will use essentially the same set of tools, which we will introduce now.

\subsection{Contrast function: definition and representation}

We define \emph{the contrast function} as
$$
H_n(\theta)=
\frac{1}{r_n\sqrt{n} h_n}
\Big(\lad_{n}(\theta)-\lad_{n}(\theta_0)\Big)
=\frac{1}{r_n^2}h_n^{-1/\al}\Big(\lad_{n}(\theta)-\lad_{n}(\theta_0)\Big),
$$
which is a.s. continuous in $\theta\in\overline{\Theta}$. 
Then $H_n(\hat\theta_n)\leq H_n(\theta_0)=0$ a.s., and in order to prove the consistency with some rate $\rho_n>0$ in the sense that
\be\label{cons_rate_generic}
\P(|\hat \theta_n-\theta_0|< \rho_n)\to 1,
\ee
it is sufficient to show that
\be\label{contrast_generic}
\P\left(\inf_{\theta:|\theta-\theta_0|\geq \rho_n} H_n(\theta)>0\right)\to 1.
\ee
To derive \eqref{contrast_generic} with various rates $\rho_n$, we will systematically use the following representation for the contrast function.
Denote
\begin{align}
\zeta_{k,n} &=h_n^{-1/\alpha}\left(X_{t_{k,n}}-F_{h_n}(\theta_0; X_{t_{k-1,n}})\right)V(X_{t_{k-1,n}}),
\nn\\
\kappa_{k,n}(\theta) &= h_n^{-1/\alpha}\left(F_{h_n}(\theta; X_{t_{k-1,n}})-F_{h_n}(\theta_0; X_{t_{k-1,n}})\right)V(X_{t_{k-1,n}}),
\nn
\end{align}
then we have
$$
\ba
H_n(\theta)
=\frac{1}{r_n^2}
\sum_{k=1}^n\left(\left|\zeta_{k,n}-\kappa_{k,n}(\theta)\right|-\left|\zeta_{k,n}\right|\right).
\ea
$$
Define a non-negative function $q(x,v)$ by
\be\label{q}
  q(x,v)=\left\{
           \begin{array}{ll}
             (2v-2x)1_{x\in [0,v)}, & v\geq0; \\
             (2x-2v)1_{x\in (v,0]}, & v<0.
           \end{array}
         \right.
  \ee
A direct computation shows that $|q(x,v)-q(x,v')|\leq 2|v-v'|$
and  
\be\label{identity}
|x-v|-|x| = - v\sgn x + q(x,v),\quad x\ne0,
\ee
which leads to the decomposition
\be\label{decomp_H}
 H_n(\theta)=
 \frac{1}{r_n^2}
 \sumk u_{k,n}(\theta)+\frac{1}{r_n^2}\sumk y_{k,n}(\theta)=:U_n(\theta)+Y_n(\theta)
\ee
with
$$
u_{k,n}(\theta):=-\kappa_{k,n}(\theta)\sgn (\zeta_{k,n}),\quad y_{k,n}(\theta):=q(\zeta_{k,n},\kappa_{k,n}(\theta)).
$$
The heuristics behind this decomposition can be explained as follows: formally, we have
$$
v\sgn x=(|x|)'v, \quad v^{-2}q(x,v)\to \delta_0(x)=\frac12(|x|)'', \quad v\to 0,
$$
which means that $U_n(\theta)$ and $Y_n(\theta)$ essentially correspond to the linear and the quadratic terms in the Taylor expansion of the random function $H_n(\theta)$.  This seemingly vague explanation is actually quite informative: we will see later that, on a proper restriction of the set $\Theta$ where the value of the LAD estimator $\tes$ is contained, functions  $U_n(\theta)$ and $Y_n(\theta)$ are linear and quadratic, respectively, up to negligible summands. With this perspective in mind, we will call   $U_n(\theta)$ and $Y_n(\theta)$ the linear term and the quadratic term, respectively.

We further decompose the linear and quadratic terms into their martingale and predictable parts as follows:
\be\label{decomp_U}
U_n(\theta)=
\frac{1}{r_n^2}\sumk u_{k,n}^M(\theta)+\frac{1}{r_n^2}\sumk u_{k,n}^P(\theta)=:U_n^M(\theta)+U_n^P(\theta),
\ee
\be\label{decomp_Y}
Y_n(\theta)=\frac{1}{r_n^2}\sumk y_{k,n}^M(\theta)+\frac{1}{r_n^2}\sumk y_{k,n}^P(\theta)=:Y_n^M(\theta)+Y_n^P(\theta),
\ee
with
$$
u_{k,n}^P(\theta):=\E[u_{k,n}(\theta)|\Ff_{k-1,n}], \quad u_{k,n}^M(\theta):=u_{k,n}(\theta)-u_{k,n}^P(\theta),
$$
$$
y_{k,n}^P(\theta):=\E[y_{k,n}(\theta)|\Ff_{k-1,n}],  \quad y_{k,n}^M(\theta):=y_{k,n}(\theta)-y_{k,n}^P(\theta).
$$
Later, we will show that:
\begin{itemize}
  \item for the linear term, its  martingale part $U_n^M(\theta)$ is the principal one, and the predictable part is negligible, in a sense;
  \item for the quadratic term, its predictable part $Y_n^P(\theta)$ is the principal one, and the martingale part is negligible, in a sense.
\end{itemize}

\subsection{Estimates for martingale and predictable parts}

We will repeatedly use, in slightly different settings, essentially the same argument in order to estimate the martingale and the `negligible predictable' parts in the decompositions \eqref{decomp_U}, \eqref{decomp_Y} and their analogues.
Therefore, for convenience of the reader, we give this argument separately.

\begin{lem}\label{l41} We have
\be\label{Delta_neglig}
\sup_{\theta\not=\theta_0}{|U_n^P(\theta)|\over |\theta-\theta_0|}=o_\P(r_n^{-1}), \quad n\to \infty.
\ee
\end{lem}
\begin{proof} Since $\kappa_{k,n}(\theta)$ is $\Ff_{k-1,n}$-measurable, we have
$$
\E[u_{k,n}(\theta)|\Ff_{k-1,n}]=-\kappa_{k,n}(\theta)\E[\sgn (\zeta_{k,n})|\Ff_{k-1,n}].
$$
Let us estimate the terms of this product separately. For the first term, we have simply by \eqref{F1},
\be\label{kappa0}\ba
|\kappa_{k,n}(\theta)|&=h_n^{-1/\alpha}|F_{h_n}(\theta; X_{t_{k-1,n}})-F_{h_n}(\theta_0; X_{t_{k-1,n}})|V(X_{t_{k-1,n}})
\\
&\leq C h_n^{-1/\alpha+1}|\theta-\theta_0|V(X_{t_{k-1,n}})W(X_{t_{k-1,n}}).
\ea
\ee
To estimate the second term,  we use the decomposition \eqref{decomp} for the transition density of the process $X$:
\begin{align}\label{decomp2}
\E[\sgn(\zeta_{k,n})|\Ff_{k-1,n}]
&=\int_{\Re}\sgn\left(y-F_{h_n}(\theta_0;X_{t_{k-1,n}})\over h^{1/\alpha}_n\right)\wt p_{h_n}(X_{t_{k-1,n}},y)\, dy \nn\\
&\qquad{}+\int_{\Re}\sgn\left(y-F_{h_n}(\theta_0;X_{t_{k-1,n}})\over h^{1/\alpha}_n\right)r_{h_n}(X_{t_{k-1,n}},y)\, dy,
\end{align}
and since $|\sgn(\cdot)|\leq 1,$  \eqref{Residue_1} guarantees that the second summand is bounded by $Ch_n^\delta$; here and below we assume $\delta$ to be fixed and satisfy all the relations in Assumption \ref{AssStab}. {For the first summand, we have 
$$\ba
\phantom{a}&\left|\int_{\Re}\sgn\left(y-F_{h_n}(\theta_0;X_{t_{k-1,n}})\over h^{1/\alpha}_n\right){1\over \sigma_{h_n}(x) h_n^{1/\alpha}}\phi_{\al}\left({y-\mathfrak{f}_{h_n}(\theta_0;X_{t_{k-1,n}})\over \sigma_{h_n}(X_{t_{k-1,n}})h_n^{1/\alpha}}\right)\, dy\right|
\\&= \left|\int_{\Re}\sgn\left(z-\frac{F_{h_n}(\theta_0;X_{t_{k-1,n}})-\mathfrak{f}_{h_n}(\theta_0;X_{t_{k-1,n}})}{\sigma_{h_n}(X_{t_{k-1,n}})
h_n^{1/\alpha}}\right)\phi_{\al}(z)\, dz\right|
\\&\leq C\frac{|F_{h_n}(\theta_0;X_{t_{k-1,n}})-\mathfrak{f}_{h_n}(\theta_0;X_{t_{k-1,n}})|}{\sigma_{h_n}(X_{t_{k-1,n}})
h_n^{1/\alpha}},
\ea
$$
}
where, in the last inequality, we have used that $\phi_\al$ is symmetric and bounded. Recall that \eqref{F0} holds true with $f_t$ replaced by $\mathfrak{f}_t$; see Remark \ref{remAss2} and the discussion in Appendix \ref{sA1}. Since $\sigma_t(x)$ is bounded away from 0, the first summand in \eqref{decomp2} is bounded by
$$
Ch_n^{\delta_{regr}}W(X_{t_{k-1,n}})\leq Ch_n^{\delta}W(X_{t_{k-1,n}}).
$$
Combining the above bounds, we obtain
$$\ba
\sup_{\theta\not=\theta_0}{|U_n^P(\theta)|\over |\theta-\theta_0|}&\leq {C\over r_n^2}h_n^{\delta}h_n^{-1/\alpha+1}\sumk V(X_{t_{k-1,n}})W(X_{t_{k-1,n}})^2
\\&\leq {Cn^{1/2} h_n^{\delta} \over r_n}\left({1\over n}\sumk V(X_{t_{k-1,n}})W(X_{t_{k-1,n}})^2\right).
\ea
$$
Observe that
$$
{1\over n}\sumk V(X_{t_{k-1,n}})W(X_{t_{k-1,n}})^2=O_\P(1), \quad n\to \infty.
$$
In the finite observation horizon case, this is obvious because the weight functions are locally bounded and $X$ has a.s. bounded trajectories; in the infinite observation horizon case, this follows from Assumption \ref{A_moment}. Thus,
$$
\sup_{\theta\not=\theta_0}{|U^P_n(\theta)|\over |\theta-\theta_0|}=O_\P(n^{1/2} h_n^{\delta} r_n^{-1}),\quad n\to \infty
$$
and applying the statistic stability Assumption \ref{AssStab}, we complete the proof.
\end{proof}

\begin{lem}\label{l42} 
\begin{itemize}
              \item[1.] For any $\theta\not=\theta_0$, we have
              \be\label{M_neglig_1}
{|U_n^M(\theta)|+|Y_n^M(\theta)|\over |\theta-\theta_0|}=O_\P(r_n^{-1}), \quad n\to \infty.
\ee
              \item[2.] For any $\upsilon\in (0,1)$, we have
              \be\label{M_neglig_2}
\sup_{\theta\not=\theta_0}{|U_n^M(\theta)|+|Y_n^M(\theta)|\over |\theta-\theta_0|^\upsilon}=O_\P(r_n^{-1}), \quad n\to \infty.
\ee
            \end{itemize}
\end{lem}

\begin{rem} The denominator $|\theta-\theta_0|^\upsilon$ in the \emph{uniform}  bound is worse than the term $|\theta-\theta_0|$ which we have in the
\emph{individual} one and which is the same as the one from the previous lemma.  This is one of the main sources of technical complications we will experience later, and it seems to be inevitable because of the method of the proof based on the Kolmogorov-Chentsov theorem.
See also the first paragraph in Section \ref{sec_proof.II}.
\end{rem}

\begin{proof}[Proof of Lemma \ref{l42}] 
1.  We have 
$$\rev{
|u_{k,n}^M(\theta)|\leq |\kappa_{k,n}(\theta)|, \quad |y_{k,n}^M(\theta)|\leq 2|\kappa_{k,n}(\theta)|}
$$
because
$$
 \rev{|v\sgn(x)|\leq |v|, \quad |q(x,v)|\leq 2|v|,}
$$
and $\kappa_{k,n}(\theta)$ is $\Ff_{k-1,n}$-measurable. From now on, we will prove the required bounds for $U_n^M(\theta)$ only; the proof for  $Y_n^M(\theta)$ will be literally the same.

In the finite observation horizon case, fix $R>0$ and denote
$$
\tau_R=\inf\{t\leq T: |X_{t_{k-1,n}}|>R\},
$$
with the usual convention $\inf\emptyset=T$. Then,  \rev{by \eqref{kappa0} and because $V(x)\equiv 1$ and $W(x)$ is locally bonded},  
$$\ba \E[U_{n}^M(\theta)^21_{\tau_R=T}]
&={1\over r_n^4}\E\sum_{k: t_{n, k-1}< \tau_R} \E[u_{k,n}^M(\theta)^2|\Ff_{k-1,n}]
\\&\leq {4nh_n^{2-2/\alpha}\over r_n^4} |\theta-\tz|^2 \sup_{|x|\leq R}W^2(x) \\
&\leq C_R{n h_n^{2-2/\alpha}\over r_n^4}|\theta-\tz|^2
={C_R\over r_n^2}|\theta-\tz|^2.
\ea
$$
This gives an analogue of \eqref{M_neglig_1} for $U_n^M(\theta)$ on the set $\{\tau_R=T\}$, and because $\{\tau_R=T\}\to \Omega$ as $R\to \infty$, actually proves \eqref{M_neglig_1} for $U_n^M(\theta)$.

In the infinite observation horizon case, in the above estimate, we should combine the localization argument with Assumption \ref{A_moment} on the weights; we omit the details, referring to the second part of the proof where the same issue is treated in a more complicated setting.

\medskip

2. In order to get the uniform estimate, we will apply the Kolmogorov-Chentsov criterion for the existence of a H\"older continuous modification. For that purpose, we need to estimate the higher-order moments of the (properly localized) differences of $U_n^M(\theta)$. 
Namely for any $\theta, \theta'$, we consider the martingale
$$
U_{r,n}^M(\theta, \theta') := \sum_{k=1}^r \Big(u_{k,n}^M(\theta)-u_{k,n}^M(\theta')\Big), \quad r\geq 0.
$$
Recall that $u_{k,n}^M(\theta_0)=0$ so that
$$
U_{n}^M(\theta) = r_n^{-2} U_{n,n}^M(\theta, \theta_0).
$$
We have
\begin{align}
\left(u_{k,n}^M(\theta)-u_{k,n}^M(\theta')\right)^2
&\leq 4(\kappa_{k,n}(\theta)-\kappa_{k,n}(\theta'))^2 
\nn\\
&\leq Ch_n^{2-2/\alpha} |\theta-\theta'|^2 V(X_{t_{k-1,n}})^2 W(X_{t_{k-1,n}})^2.
\nn
\end{align}
Let $\tau$ be a stopping time, then by the Burkholder-Davis-Gundy inequality applied to the martingale $\rev{U^M_{\cdot,n}}(\theta, \theta')$ stopped at $\tau$, we have for any $p>1$
$$
\E\left[\left|U^M_{\tau\wedge n, n}(\theta, \theta')\right|^p\right]
\leq C_ph_n^{p-p/\alpha}|\theta-\theta'|^p\,
\E\left[ \left|\sum_{k=1}^\tau V(X_{t_{k-1,n}})^2 W(X_{t_{k-1,n}})^2\right|^{p/2} \right].
$$
Fix $R>0$ and take in the above inequality
$$
\tau=\tau_{n,R}=\min\left\{l: \sum_{k=0}^l V(X_{t_{k,n}})^2W(X_{t_{k,n}})^2> R^2 n \right\}.
$$
Then,
\be\label{Moment_for_Kolmogorov}\ba
\E \left[ |U^M_n(\theta)-U^M_n(\theta')|^p 1_{\{\tau_{n, R}>n\}} \right] &\leq r_n^{-2p}\, \E\left[ \left|U^M_{\tau\wedge n, n}(\theta, \theta')\right|^p\right] 
\\&\leq {C_p h_n^{p-p/\alpha}\over r_n^{2p}}|\theta- \theta'|^p \E\left[\left|\sum_{k=1}^{\tau_{n,R}} V(X_{t_{k-1,n}})^2W(X_{t_{k-1,n}})^2\right|^{p/2}\right]
\\&\leq C_p R^p {h_n^{p-p/\alpha}\over r_n^{2p}} n^{p/2} |\theta- \theta'|^p=C_p R^p r_n^{-p}|\theta- \theta'|^p.
\ea
\ee
Recall that $m$ denotes the dimension of $\theta$.
\rev{For any fixed $v \in (0,1)$ we can choose $p$ large enough for the inequality $p - m > vp$ to hold,} and, and for such $p$ the Kolmogorov-Chentsov theorem yields
\be\label{Kolmogorov_Hol}
\E\left[ \left(\sup_{\theta\not=\theta'}{|U^M_n(\theta)-U^M_n(\theta')|\over |\theta-\theta'|^\upsilon}\right)^p 1_{\{\tau_{n, R}>n\}} \right] 
\leq R^p C_{p,\upsilon,m} r_n^{-p}.
\ee
We complete the proof by observing that
$$
\liminf_{n\to \infty}\P(\tau_{n, R}>n)\to 1, \quad R\to \infty,
$$
because
$$
{1\over n}\sumk V(X_{t_{k-1,n}})^2W(X_{t_{k-1,n}})^2=O_\P(1), \quad n\to \infty
$$
by Assumption \ref{A_moment} and our convention that $V$ is bounded; see Remark \ref{rem31}.
\end{proof}

\subsection{Basic consistency}

\begin{prp}\label{p41} We have
\begin{equation}
\hat \theta_n-\theta_0=o_\P(1), \quad n\to \infty.
\label{cons}
\end{equation}
\end{prp}

\begin{proof} Let us estimate from below the predictable part  $Y_n^P(\theta)$. For that, we will perform the following `double truncation' procedure.   First, note that $q(x,v)\geq 0$ and thus for arbitrary $R>0$,
$$
Y_n^P(\theta)\geq Y_{n,R}^P(\theta)={1\over r_n^2}\sumk y_{n,k}^P(\theta)\,\chi(R^{-1}X_{t_{k-1,n}}),
$$
where $\chi:\Re\to [0,1]$ is a continuous function such that
$$\chi(r)=\left\{
    \begin{array}{ll}
      1, & |r|\leq 1; \\
      0, & |r|\geq 2.
    \end{array}
  \right.
$$
Second, we perform the truncation of the domain of integration in the representation
\begin{align}\label{decomp3}
y_{n,k}^P(\theta)=\E[q(\zeta_{k,n}, \kappa_{k,n}(\theta))|\Ff_{k-1,n}]
&=\int_{\Re}q\left({y-F_{h_n}(\theta_0;X_{t_{k-1,n}})\over h^{1/\alpha}_n},\kappa_{k,n}(\theta)\right)p_{h_n}(X_{t_{k-1,n}},y)\, dy.
\end{align}
\rev{Namely, we change the domain of integration to the smaller one} 
\begin{equation}
    \left\{y:\,|y-F_{h_n}(\theta_0;X_{t_{k-1,n}})|\leq h_n^{1/\alpha}\right\}.
\end{equation}
We use decomposition \eqref{decomp} for the transition density $p_{h_n}(X_{t_{k-1,n}},y)$, and observe that on this domain we have the following:
\begin{itemize}
  \item  Since $\sigma_t(x)$ is bounded and separated from zero, there exists $c>0$ such that
$$
\wt p_{h_n}(X_{t_{k-1,n}},y)\geq 2c h^{-1/\alpha}_n;
$$
  \item By Theorem \ref{tA2}, the bound \eqref{Residue_2} holds true, and thus 
$$
|r_{h_n}(X_{t_{k-1,n}},y)|\leq C  h^{-1/\alpha+\delta'}_n.
$$
\end{itemize}
Then, for $n$ large enough such that $C h^{\delta'}_n<c$, we have on this domain
$$
p_{h_n}(X_{t_{k-1,n}},y)\geq  \wt  p_{h_n}(X_{t_{k-1,n}},y)-|r_{h_n}(X_{t_{k-1,n}},y)|\geq ch_n^{-1/\al},
$$
which gives the bound
$$\ba
y_{n,k}^P(\theta)&\geq c h^{-1/\alpha}_n\int_{|y-F_{h_n}(\theta_0;X_{t_{k-1,n}})|\leq h_n^{1/\alpha}}q\left({y-F_{h_n}(\theta_0;X_{t_{k-1,n}})\over h^{1/\alpha}_n},\kappa_{k,n}(\theta)\right)\, dy \nn\\
&=c \int_{[-1,1]}q\left(z,\kappa_{k,n}(\theta)\right)\, dz.
\ea
$$
By a direct calculation, we can verify that
\be\label{linear_quadratic}
\int_{[-1,1]}q(z, v)\, dz\geq c\left( |v|\wedge v^2\right), \qquad v\in\mbbr.
\ee
Summarizing all the above, we obtain that
\be\label{lower_1}
Y_n^P(\theta)\geq {c\over r_n^2}\sumk  \Big(|\kappa_{k,n}(\theta)|\wedge |\kappa_{k,n}(\theta)|^2\Big)\chi(R^{-1}X_{t_{k-1,n}}),\quad R\geq 0
\ee
with the constant $c$ not depending on $R$.

Using \eqref{F3} and the fact that $W(\cdot)$ is locally bounded, we get that, on the set $\{|X_{t_{k-1,n}}|\leq 2R\}$,
\be\label{kappa_low1}
|\kappa_{k,n}(\theta)|\geq h_n^{-1/\alpha+1}|a(\theta;X_{t_{k-1,n}})-a(\theta_0;X_{t_{k-1,n}})|V(X_{t_{k-1,n}})-C_Rh_n^{-1/\alpha+1+\gamma}.
\ee
In addition, $a(\theta;x)-a(\theta_0;x)$ is bounded on $\Theta\times[-R,R]$ and $V(x)$ is bounded on $[-R,R]$, hence
$$
|a(\theta;x)-a(\theta_0;x)|V(X_{t_{k-1,n}})\geq c_R(a(\theta;x)-a(\theta_0;x))^2V(X_{t_{k-1,n}})^2.
$$
\rev{ This gives 
$$
|\kappa_{k,n}(\theta)|\geq c_Rh_n^{-1/\alpha+1}(a(\theta;x)-a(\theta_0;x))^2V(X_{t_{k-1,n}})^2-C_Rh_n^{-1/\alpha+1+\gamma}
$$
and $$
|\kappa_{k,n}(\theta)|\wedge |\kappa_{k,n}(\theta)|^2\leq \left(c_R\wedge\frac12\right)(h_n^{-1/\alpha+1}\wedge h_n^{-2/\alpha+2}) \bigg[ \Big(a(\theta;X_{t_{k-1,n}})-a(\theta_0;X_{t_{k-1,n}})\Big)^2V(X_{t_{k-1,n}})^2 -C_Rh_n^{\gamma}\bigg],
$$
in the last inequality we have used the Cauchy inequality $(a-b)^2\geq a^2/2-b^2$ and assumption that $\gamma>0$. 
Denote 
$$
\Lambda_{n,R}(\theta)=\frac{1}{n} \sumk \Big(a(\theta;X_{t_{k-1,n}})-a(\theta_0;X_{t_{k-1,n}})\Big)^2V(X_{t_{k-1,n}})^2\chi(R^{-1}X_{t_{k-1,n}}), 
$$
then 
$$\ba
Y_{n}^P(\theta)&\geq c\left(c_R\wedge\frac12\right)\,{h_n^{-1/\alpha+1}\wedge h_n^{-2/\alpha+2}\over r_n^2}
\\
&{}\qquad \times \sumk \bigg[ \Big(a(\theta;X_{t_{k-1,n}})-a(\theta_0;X_{t_{k-1,n}})\Big)^2V(X_{t_{k-1,n}})^2 -C_Rh_n^{\gamma}\bigg] \chi(R^{-1}X_{t_{k-1,n}})
\\&\ge c\left(c_R\wedge\frac12\right)\,{h_n^{-1/\alpha+1}\wedge h_n^{-2/\alpha+2}\over r_n^2} \Big(n \Lambda_{n,R}(\theta)-nC_Rh_n^{\gamma}\Big).
\ea
$$
}
Note that
$$
n\,{h_n^{-1/\alpha+1}\wedge h_n^{-2/\alpha+2}\over r_n^2} =1\wedge h_n^{1/\alpha-1},
$$
hence we have proved the following family of lower bounds: for any $R>0$ there exist $c_R, C_R>0$ such that, for $n$ large enough,
\be\label{lower_3}
Y_{n}^P(\theta)\geq c_R\left(1\wedge h_n^{1/\alpha-1}\right)\left(\Lambda_{n,R}(\theta)-C_Rh_n^{\gamma}\right), \quad \theta\in \Theta.
\ee

We are now ready to prove \eqref{cons}. We have
$$
 \Big(1\wedge h_n^{1/\alpha-1}\Big)  r_n
 =\left\{
\begin{array}{ll}
\sqrt{n}, & \alpha\in (0,1] \\
r_n, & \alpha\in [1,2)
\end{array}
\right.\to \infty, \quad n\to \infty.
$$
Since $\Theta$ is bounded, this yields, together with  \eqref{Delta_neglig}, \eqref{M_neglig_2}, that
\be\label{Y^P_is_main}
\sup_{\theta\in \Theta}|H_n(\theta)-Y^P_n(\theta)|=o_\P(1\wedge h_n^{1/\alpha-1}), \quad n\to \infty.
\ee
On the other hand, we have the family of lower bounds \eqref{lower_3}.
Observe the following.
\begin{itemize}
    \item In the finite observation horizon case, denote
    $$
    \Lambda_{0,R}(\theta)=
    \frac1T
    \int_0^T\Big(a(\theta; X_{t})-a(\theta_0; X_{t})\Big)^2\chi(R^{-1}X_{t})\, dt
    $$
    and observe that, with probability 1,
    $$
    \Lambda_{n,R}(\theta)\to \Lambda_{0,R}(\theta), \quad n\to \infty
    $$
    uniformly in $\theta\in \Theta$; recall that $V(x)\equiv 1$ in this case. 
    \item In the infinite observation horizon case, we have uniform convergence in probability
    \be\label{Lambda}
    \sup_{\theta}|\Lambda_{n,R}(\theta)-\Lambda_{0,R}(\theta)|\to 0, \quad n\to \infty,
    \ee
    with (non-random)
    $$
    \Lambda_{0,R}(\theta)=\int_{\mathbb{R}}V(x)^2\Big(a(\theta; x)-a(\theta_0; x)\Big)^{\otimes 2}\chi(R^{-1}x)\, \pi(\theta_0, dx),
    $$
    see Appendix \ref{sB2.e}. 
\end{itemize}
We now complete the proof using the following standard argument. We have $H_n(\wh \theta_n)\leq 0$, hence
for every $\eps>0$, $\nu>0$, and $R>0$, we have by \eqref{lower_3} and \eqref{Y^P_is_main}
\begin{align}
\liminf_{n}&\,\P(|\tes-\tz|<\eps)
\ge \liminf_{n}\P\left(\inf_{\theta: |\theta-\tz|\ge\eps}H_{n}(\theta)>0\right)\nn
\\& \ge \liminf_{n}\P\left(\inf_{\theta: |\theta-\tz|\ge\eps}H_{n}(\theta)>\big(1\wedge h_n^{1/\alpha-1}\big)c_R\nu/2\right) \nn\\
&\ge\liminf_{n} \P\left(\inf_{\theta: |\theta-\tz|\ge\eps}\Lambda_{n,R}(\theta)>\nu\right) \ge \P\left(\inf_{\theta: |\theta-\tz|\ge\eps}\Lambda_{0,R}(\theta)>\nu\right),
\nonumber
\end{align}
where the last inequality follows from the portmanteau theorem. We have with probability 1,
$$
\sup_{\theta}|\Lambda_{0,R}(\theta)-\Lambda_{0}(\theta)|\to 0, \quad R\to \infty
$$
uniformly in $\theta$, hence  taking $R\to \infty$ and using the global identifiability condition, we get
$$
\liminf_{n}\P(|\tes-\tz|<\eps)
\geq \P\left(\inf_{\theta: |\theta-\tz|\ge\eps}\Lambda_{0}(\theta)>\frac{\nu}{2}\right).
$$
The right-hand side converges to $1$ as $\nu\to 0+$ so that we can conclude
$$
\liminf_{n}\P(|\tes-\tz|<\eps) = 1, \quad \eps>0,
$$
completing the proof of \eqref{cons}.
\end{proof}

\subsection{Improved consistency for $\alpha<1$}

\begin{prp}\label{p42}  In the case $\alpha<1$, the improved following consistency rate holds:
\begin{equation}
\hat \theta_n-\theta_0=o_\P(h_n^{1/\alpha-1}), \quad n\to \infty.
\label{cons1}
\end{equation}
\end{prp}

\begin{proof} We will mainly repeat the arguments from the proof of Proposition \ref{p41}, but now to obtain the lower bound for $Y_n^P(\theta)$, we will use a `local' bound instead of the `global' bound \eqref{kappa_low1} for $|\kappa_{k,n}(\theta)|$.

By the assumption \eqref{F5}, we have the following bound for $|X_{t_{k-1,n}}|\leq 2R$:
\be\label{kappa_low2}
|\kappa_{k,n}(\theta)|\geq h_n^{-1/\alpha+1}
\bigg( |\nabla_\theta a(\theta_0;X_{t_{k-1,n}})\cdot(\theta-\theta_0)|V(X_{t_{k-1,n}})^2 - C_R |\theta-\theta_0|\Big(|\theta-\theta_0|^\gamma+h_n^{\gamma}\Big)
\bigg).
\ee
  Since $|\nabla_\theta a(\theta_0;x)|$ is bounded for $|x|\leq R$, we have for   $|X_{t_{k-1,n}}|\leq 2R$,
$$
|\nabla_\theta a(\theta_0;X_{t_{k-1,n}})\cdot(\theta-\theta_0)|\geq c_R|\theta-\theta_0|\left(\nabla_\theta a(\theta_0;X_{t_{k-1,n}})\cdot{\theta-\theta_0\over |\theta-\theta_0|}\right)^2,
$$
which gives
$$
|\kappa_{k,n}(\theta)|\geq h_n^{-1/\alpha+1}|\theta-\theta_0|\left(c_R\left(\nabla_\theta a(\theta_0;X_{t_{k-1,n}})\cdot{\theta-\theta_0\over |\theta-\theta_0|}\right)^2 V(X_{t_{k-1,n}})^2 - C_R(|\theta-\theta_0|^\gamma+h_n^{\gamma})\right).
$$
Again by applying the Cauchy inequality $(a-b)^2\geq a^2/2-b^2$, we have for $|X_{t_{k-1,n}}|\leq 2R$,
\begin{align}
|\kappa_{k,n}(\theta)|^2&\geq {h_n^{-2/\alpha+2}\over 2}|\nabla_\theta a(\theta_0;X_{t_{k-1,n}})\cdot(\theta-\theta_0)|^2 V(X_{t_{k-1,n}})^2
\nn\\
&{}\qquad -C_Rh_n^{-2/\alpha+2}|\theta-\theta_0|^2\left(|\theta-\theta_0|^\gamma+h_n^{\gamma}\right)^2
\\&
\ge h_n^{-2/\alpha+2}|\theta-\theta_0|^2\bigg\{ {1\over 2}\left(\nabla_\theta a(\theta_0;X_{t_{k-1,n}})\cdot{\theta-\theta_0\over |\theta-\theta_0|}\right)^2 V(X_{t_{k-1,n}})^2
\nn\\
&{}\qquad -C_R(|\theta-\theta_0|^\gamma+h_n^{\gamma})^2\bigg\}.    
\label{Cauchy}
\end{align}
Since $|\kappa_{k,n}(\theta)|^2\geq 0,$ we can deduce from the latter bound that, for $|\theta-\theta_0|\geq \eps h_n^{1/\alpha-1}$,
\begin{align}
    |\kappa_{k,n}(\theta)|^2 &\geq \eps h_n^{-1/\alpha+1}|\theta-\theta_0|\bigg\{ {1\over 2}\left(\nabla_\theta a(\theta_0;X_{t_{k-1,n}})\cdot{\theta-\theta_0\over |\theta-\theta_0|}\right)^2 V(X_{t_{k-1,n}})^2
    \nn\\
    &{}\qquad -C_R(|\theta-\theta_0|^\gamma+h_n^{\gamma})^2\bigg\}.
\end{align}
Then, by \eqref{lower_1}, we get for arbitrary $\eps>0, \wt \eps>0$ and $\eps h_n^{1/\alpha-1}\leq |\theta-\theta_0|\leq \wt \eps$,
$$\ba
& h_n^{-1/\alpha+1}{Y_n^P(\theta) \over |\theta-\theta_0|}
\\&\geq {h_n^{-2/\alpha+2}\over r_n^2}\sumk \bigg\{\left({\eps\over 2}\wedge c_R\right)\left(\nabla_\theta a(\theta_0;X_{t_{k-1,n}})\cdot{\theta-\theta_0\over |\theta-\theta_0|}\right)^2 V(X_{t_{k-1,n}})^2\chi(R^{-1}X_{t_{k-1,n}})
\nn\\
&{}\qquad -C_R(\wt\eps^\gamma+h_n^{\gamma})\bigg\}.
\ea
$$
Recall that
$$
{h_n^{-2/\alpha+2}\over r_n^2}=\frac{1}{n}
$$
and denote
$$
\Gamma_{n,R}=\frac{1}{n}\sumk\left(\nabla_\theta a(\theta_0;X_{t_{k-1,n}})\right)^{\otimes 2}\chi(R^{-1}X_{t_{k-1,n}}) 
V(X_{t_{k-1,n}})^2, 
$$
$$
\gamma_{n,R}=\inf_{|\ell|=1}\Big(\Gamma_{n,R}\ell\cdot \ell\Big).
$$
Summarizing the above arguments, we get the following lower bound:
\be\label{lower_6}
h_n^{-1/\alpha+1}\inf_{\eps h_n^{1/\alpha-1}\leq |\theta-\theta_0|\leq \wt \eps}
{Y_n^P(\theta)\over |\theta-\theta_0|}
\geq
\left({\eps\over 2}\wedge c_R\right)\gamma_{n,R}-C_R(\wt\eps^\gamma+h_n^{\gamma}).
\ee

Next, by Lemma \ref{l41} we have
$$
h_n^{-1/\alpha+1}\sup_{\theta}{|U^P_n(\theta)|\over |\theta-\theta_0|}=o_\P(h_n^{-1/\alpha+1}r_n^{-1})=o_\P(n^{-1/2}).
$$
Similarly, by Lemma \ref{l42} we have for any $\upsilon\in(0,1)$,
\begin{align}
& h_n^{-1/\alpha+1}\sup_{|\theta-\theta_0|\geq \eps h_n^{1/\alpha-1}}{|U_n^M(\theta)+Y_n^M(\theta)|\over |\theta-\theta_0|}
\nn\\
&\quad \leq 
h_n^{-1/\alpha+1}h_n^{(-1/\alpha+1)(1-\upsilon)}\sup_{|\theta-\theta_0|\geq \eps h_n^{1/\alpha-1}}
\frac{|U_n^M(\theta)|+|Y_n^M(\theta)|}{|\theta-\theta_0|^\upsilon}
\nn\\
&\quad =o_\P(h_n^{(-1/\alpha+1)(1-\upsilon)}n^{-1/2}).
\nn
\end{align}
We are assuming that $\liminf_{n\to \infty}T_n>0$, hence taking $\upsilon>1$ sufficiently large so that
$$
(-1/\alpha+1)(1-\upsilon) + \frac12 \ge 0 \quad \iff \quad 
\upsilon\ge 1-\frac{\al}{2(1-\al)},
$$
we get
\be\label{Delta_M_neglig}
 \psi_{n, \eps}:=h_n^{-1/\alpha+1}\sup_{|\theta-\theta_0|\geq \eps h_n^{1/\alpha-1}}{|H_n(\theta)-Y^P_n(\theta)|\over |\theta-\theta_0|}=o_\P(1).
 \ee

 Now we can finalize the entire argument. By Proposition \ref{p41}, we have for arbitrary $\eps>0, \wt \eps>0$
$$\ba
\liminf_{n}\P(|\tes-\tz|<\eps h_n^{1/\alpha-1})&=\liminf_{n}\P\left(\inf_{\eps h_n^{1/\alpha-1}\leq |\theta-\theta_0|\leq \wt \eps}{H_{n}(\theta)\over |\theta-\theta_0|}>0\right).
\ea
$$
Recall the decompositions \eqref{decomp_H}, \eqref{decomp_U}, and \eqref{decomp_Y}.
By \eqref{Delta_M_neglig} and \eqref{lower_6},
\begin{align}
& \liminf_{n} \P\left(\inf_{\eps h_n^{1/\alpha-1}\leq |\theta-\theta_0|\leq \wt\eps}{H_{n}(\theta)\over |\theta-\theta_0|}>0\right)
\nn\\
&{}\qquad =\liminf_{n}\P\left(h_n^{-1/\alpha+1}\inf_{\eps h_n^{1/\alpha-1}\leq |\theta-\theta_0|\leq \wt\eps}{H_{n}(\theta)\over |\theta-\theta_0|}>0\right)\nn
\\&{}\qquad \geq \liminf_{n}\P\left(\left({\eps\over 2}\wedge c_R\right)\gamma_{n,R}-C_R(\wt\eps^\gamma+h_n^{\gamma})-\psi_{n,R}>0\right).
\end{align}
We have
$$
\gamma_{n,R}\to \gamma_{0,R}:=\inf_{|\ell|=1}\Big(\Gamma_{0,R}\ell\cdot \ell\Big), \quad n\to \infty
$$
in probability, where
$$
\Gamma_{0,R}=\int_0^T\Big(\nabla_\theta a(\theta_0; X_{t})\Big)^{\otimes 2}\chi(R^{-1}X_{t})\, dt
$$
in the finite observation horizon case (where $V(x)\equiv 1$) and
$$
\Gamma_{0,R}=\int_{\mathbb{R}}V(x)^2 \Big(\nabla_\theta a(\theta_0; x)\Big)^{\otimes 2}\chi(R^{-1}x)\, \pi(\theta_0, dx)
$$
in the infinite observation horizon case. Then,
$$
\liminf_{n}\P\left(\left({\eps\over 2}\wedge c_R\right)\gamma_{n,R}-C_R(\wt\eps^\gamma+h_n^{\gamma})-\psi_{n,R}>0\right)
\geq \P\left(\left({\eps\over 2}\wedge c_R\right)\gamma_{0,R}-C_R\wt\eps^\gamma>0\right),
$$
and we finally get
\begin{align}
\liminf_{n}\P(|\tes-\tz|<\eps h_n^{1/\alpha-1}) &\geq
\lim_{R\to \infty}\lim_{\wt \eps\to 0}\P\left(\left({\eps\over 2}\wedge c_R\right)\gamma_{0,R}-C_R\wt \eps^\gamma>0\right)
\nn\\
&=\lim_{R\to \infty}\P\left(\gamma_{0,R}>0\right)=1,
\nn
\end{align}
where we have used the local identifiability condition in the last identity.
\end{proof}

\subsection{Consistency at sub-optimal rate}\label{s45}

We are now ready to prove \eqref{cons_rate}.  The proof repeats those of Propositions \ref{p41} and \ref{p42} with just one (but substantial)  modification.
  We have just proved that
\be\label{prob1}
\P(\theta_n\in \Theta_n)\to 1
\ee
with $$
\Theta_n := \left\{\theta:\, |\theta-\theta_0|\leq 1\wedge h_{n}^{1/\alpha-1}\right\}.
$$
This additional knowledge now makes it possible to improve the lower bound for $Y_n^P(\theta)$. The key observation here is that, for $\theta\in \Theta_n$ and $|X_{t_{k-1,n}}|\leq 2R$, the term $\kappa_{k,n}(\theta)$ is bounded by $C_R$. Indeed, in this case, by \eqref{F1} one has
  $$
  |\kappa_{k,n}(\theta)|\leq h_n^{1-1/\alpha}|\theta-\theta_0|V(X_{t_{k-1,n}})W(X_{t_{k-1,n}})\leq \sup_{|x|\leq 2R}V(x)W(x).
  $$
  This means that instead of the lower estimate \eqref{linear_quadratic}, which combines a local quadratic bound with a global linear one, we can use just the quadratic bound:
  \be\label{quadratic}
  \int_{[-1,1]} q\left(z, v\right)\, dz\geq c_Q v^2, \qquad |v|\leq Q.
  \ee
  This will give, instead of \eqref{lower_1}, the following improved lower bound:
  \be\label{lower_7}
Y_{n,R}^P(\theta)\geq {c_R\over r_n^2}\sumk  \kappa_{k,n}(\theta)^2\chi(R^{-1}X_{t_{k-1,n}}), \quad \theta\in \Theta_n.
\ee
  For $ \kappa_{k,n}(\theta)^2$ we have the lower bound \eqref{Cauchy}. Since
  $$
  {h_n^{-2/\alpha+2}\over r_n^2}={1\over  n},
  $$
  we finally  get from \eqref{lower_7},
   \be\label{lower_8}
Y_{n,R}^P(\theta)\geq |\theta-\theta_0|^2\Big( c_R  \gamma_{n,R}-C_R(|\theta-\theta_0|^\gamma+h_n^{\gamma})^2\Big), \quad \theta\in \Theta_n.
\ee
Combined with \eqref{Delta_neglig}, \eqref{M_neglig_2}, the previous estimate gives for any $R>0$,
   \be\label{lower_9}
   \ba
H_{n}(\theta)\geq |\theta-\theta_0|^2\Big( c_R  \gamma_{n,R}
-C_R(|\theta-\theta_0|^\gamma+h_n^{\gamma})^2\Big)
- r_n^{-1} \eta_n  |\theta-\theta_0|^\upsilon, \quad \theta\in \Theta_n,
\ea
\ee
with $\eta_n$ bounded in probability. Since $H_n(\wh \theta_n)\leq 0$, \eqref{lower_9} yields that for each $\nu>0$, on the set
$$
\Omega_n^{R, \nu}=\big\{ \gamma_{n,R}\geq \nu, ~~\wh \theta_n\in \Theta_n\big\},
$$
the following bound holds:
\be\label{prob2}
|\wh\theta_n-\theta_0|\leq \eta_n^{R, \nu, \upsilon} r_n^{-1/(2-\upsilon)}
\ee
with $\eta_n^{R, \nu, \upsilon}$ bounded in probability.
By \eqref{prob1},
$$
\liminf_{n\to \infty}\P(\Omega_n^{R, \nu})\geq \liminf_{n\to \infty}\P( \gamma_{n,R}>\nu)\geq \P( \gamma_{0,R}>\nu).
$$
 Under the local identifiability assumption, for any $\eps>0$ we can fix $R$ large enough and $\nu>0$ small enough so that
 $$\P( \gamma_{0,R}>\nu)\geq 1-\eps.
 $$
Therefore, the inequality \eqref{prob2} holds with probability $\geq 1-2\eps$ for $n$ large enough. Since $\eps>0$ here is arbitrary, this shows that
$$
|\wh\theta_n-\theta_0|=O_\P( r_n^{-1/(2-\upsilon)})
$$
for arbitrary $\upsilon <1$. Taking $\upsilon>2-1/\varsigma,$ we complete the proof of \eqref{cons_rate}.

\section{Proofs, II: the main statements}
\label{sec_proof.II}

In the previous section, we have proved that, for any $\varsigma <1$,
$$
\P(\widehat \theta_n\in \Theta_{n, \varsigma})\to 1,
$$
where we denote
$$
\Theta_{n, \varsigma}=\left\{\theta:\, |\theta-\theta_0|\leq  r_n^{-\varsigma}\right\}.
$$
In what follows, we would like to consider the `true' rate $r_n^{-1}$; hence, it is instructive to identify the key difficulty which does not allow us to reach the `ideal' value $\varsigma=1$ using the previously developed technique. The limitation  $\varsigma <1$ comes from the estimate \eqref{lower_9}, where the last (sub-linear) term contains a power function with $\upsilon<1$. We can not get a similar estimate with $\upsilon=1$ (which would lead to consistency at the rate $r_n^{-1}$), because the \emph{uniform} bound for the martingale terms $U^M_n(\theta), Y_n^M(\theta)$, given by Lemma \ref{l42}, only provides the power $\upsilon<1$. Note that the \emph{point-wise} bound from the same lemma has the required power $\upsilon=1$. Thus, the technique, based on the Kolmogorov-Chentsov theorem, appears to be not precise enough to provide the exact consistency rate. This observation gives an insight into the subsequent constructions and considerations: we will use additional considerations to avoid the loss of accuracy caused by the Kolmogorov-Chentsov theorem. 
Put simply, we will use a new contrast function, convex in $\theta$; in this framework, point-wise bounds (convergence) will imply uniform-in-$\theta$ bounds (convergence).

\subsection{Linearization}

Denote
\begin{align}
\bar F_h(\theta;x)&=F_h(\theta_0; x)+h \nabla_\theta a(\theta_0;x)\cdot (\theta-\theta_0),
\nn\\
\bar \kappa_{k,n}(\theta)&=h_{n}^{-1/\alpha}\left(\bar F_{h_{n}}(\theta; X_{t_{k-1,n}})-\bar F_{h_{n}}(\theta_0; X_{t_{k-1,n}})\right)V(X_{t_{k,n}})
\\&=h_{n}^{1-1/\alpha} V(X_{t_{k,n}}) \nabla_\theta a(\theta_0; X_{t_{k-1,n}})\cdot (\theta-\theta_0),
\nn
\end{align}
and put
$$
\bar H_n(\theta)={1\over r_n^2}h_{n}^{-1/\alpha}\sum_{k=1}^n\left(\Big|X_{t_{k,n}}-\bar F_{h_{n}}(\theta; X_{t_{k-1,n}} )\Big|-\Big|X_{t_{k,n}}-\bar F_{h_{n}}(\theta_0; X_{t_{k-1,n}})\Big|\right)V(X_{t_{k,n}}).
$$
Because $F_h(\theta_0; x)=\bar F_h(\theta_0; x)$, we have
$$
\bar H_n(\theta)={1\over r_n^2} \sum_{k=1}^n\left( \left|\zeta_{k,n}-\bar\kappa_{k,n}(\theta)\right|-\left|\zeta_{k,n}\right|\right).
$$
The functions $\bar\kappa_{k,n}(\theta)$ are linear, and thus the new contrast function $\bar H_n(\theta)$ is convex  in $\theta$. We will show that the difference $H_n(\theta)-\bar H_n(\theta)$ is negligible, in a sense. For this purpose, we will compare the parts in the decompositions \eqref{decomp_H} of $H_n(\theta)$ with the similar parts for $\bar H_n(\theta),$ which we denote $\bar U_n(\theta)$ and $\bar Y_n(\theta),$  respectively. Similar notational convention will be used for the decompositions of $\bar U_n(\theta)$ and $\bar Y_n(\theta),$ analogous to \eqref{decomp_U} and \eqref{decomp_Y}.

Recall that $\gamma>0$ is the constant given by Assumption \ref{Ass2}.

\begin{lem}\label{l51} For any $\varsigma\in (0,1)$, we have
\be\label{Delta_neglig_lin}
\sup_{\theta\not=\theta_0, \theta\in \Theta_{n, \varsigma} }{|U_n^P(\theta)-\bar U_n^P(\theta)|\over |\theta-\theta_0|}=o_\P(r_n^{-1}h_n^{\gamma \varsigma}), \quad n\to \infty.
\ee
\end{lem}

\begin{proof}
  We have
  $$
  \rev{U_n^P(\theta)-\bar U_n^P(\theta)}={1\over r_n^2}\sumk (u_{k,n}^P(\theta)-\bar u^P_{k,n}(\theta))
$$
with
$$
\bar u^P_{k,n}(\theta)=-\bar \kappa_{k,n}(\theta)\E[\sgn (\zeta_{k,n})|\Ff_{k-1,n}].
$$
By \eqref{F5},
$$
|\kappa_{k,n}(\theta)-\bar \kappa_{k,n}(\theta)| \leq C h_n^{-1/\alpha+1}|\theta-\theta_0|\Big(|\theta-\theta_0|^\gamma+h_n^\gamma\Big)V(X_{t_{k-1,n}})W(X_{t_{k-1,n}}).
$$
Since $\delta<\delta_{regr}=1+\eta-1/\alpha$ (Assumption \ref{AssStab}), we have $(0\le)\,1-\eta<2-1/\alpha-\delta$ and hence for $\theta\in \Theta_{n, \varsigma}$,
$$
|\theta-\theta_0|^\gamma 
\leq \sqrt{nh_n^{2\delta}}\, h_n^{2-1/\alpha-\delta}
\leq h_n^{\gamma\varsigma}
$$
for $n$ large enough. 
Since $h_n^{\gamma\varsigma} \gg h_n^\gamma, $ this gives the bound
\be\label{kappa1}
|\kappa_{k,n}(\theta)-\bar \kappa_{k,n}(\theta)|\leq C h_n^{-1/\alpha+1+\gamma\varsigma}|\theta-\theta_0|V(X_{t_{k-1,n}})W(X_{t_{k-1,n}}).
\ee
Repeating the proof of Lemma \ref{l41} literally with the estimate \eqref{kappa0} replaced by \eqref{kappa1}, we get the required statement.
\end{proof}

Similarly, using \eqref{kappa1} instead of \eqref{kappa0}, we get the following analogue of the second statement in Lemma \ref{l42}.

\begin{lem}\label{l52} For any $\upsilon\in (0,1)$ and $\varsigma\in (0,1)$,
we have
\be\label{M_neglig_lin}
\sup_{\theta\not=\theta_0, \theta\in \Theta_{n, \varsigma} }{|U^M_n(\theta)-\bar U^M_n(\theta)|+|Y^M_n(\theta)-\bar Y^M_n(\theta)|\over |\theta-\theta_0|^\upsilon}=o_\P(r_n^{-1}h_n^{\gamma \varsigma}), \quad n\to \infty.
\ee
\end{lem}

Finally, we have the following.

\begin{lem}\label{l53} For any  $\varsigma\in (0,1)$,
\be\label{Y_neglig_lin}
\sup_{\theta\not=\theta_0, \theta\in \Theta_{n, \varsigma} }{|Y_n^P(\theta)-\bar Y_n^P(\theta)|\over |\theta-\theta_0|^2}=O_\P(h_n^{\gamma \varsigma}), \quad n\to \infty.
\ee
\end{lem}

\begin{proof} Denote
$$
\psi_{t}(x,z)=t^{1/\alpha}p_t\Big(x,F_t(\theta_0;x)+zt^{1/\alpha}\Big).
$$
Then, we can write
$$\ba
y_{n,k}^P(\theta)&=
\int_{\Re}q\left({y-F_{h_n}(\theta_0;X_{t_{k-1,n}})\over h^{1/\alpha}_n},\kappa_{k,n}(\theta)\right)p_{h_n}(X_{t_{k-1,n}},y)\, dy
\\&=\int_{\Re}q\left(z,\kappa_{k,n}(\theta)\right)\psi_{h_n}(X_{t_{k-1,n}},z)\, dz
\ea$$
and similarly
$$
\bar y_{n,k}^P(\theta)=\int_{\Re}q\left(z,\bar \kappa_{k,n}(\theta)\right)\psi_{h_n}(X_{t_{k-1,n}},z)\, dz.
$$
By Theorem \ref{tA2} we have $\psi_{t}(x,z)$ bounded uniformly in $x,z\in \mathbb{R}$ and $t\in (0,1]$, hence
$$
|y_{n,k}(\theta)-\bar y_{n,k}(\theta)|\leq C\int_{\Re}\left|q\left(z, \kappa_{k,n}(\theta)\right)-q\left(z,\bar \kappa_{k,n}(\theta)\right)\right|\, dz.
$$
On  the other hand, it is easy to verify using the formula \eqref{q} that
$$
\int_{\Re}\left|q\left(z, v)\right)-q\left(z,\bar v\right)\right|\, dz\leq C(|v|+|\bar v|)|v-\bar v|.
$$
We have
$$
|\kappa_{k,n}(\theta)|+|\bar\kappa_{k,n}(\theta)|\leq  C h_n^{-1/\alpha+1}|\theta-\theta_0|V(X_{t_{k-1,n}})W(X_{t_{k-1,n}}),
$$
$$
|\kappa_{k,n}(\theta)-\bar\kappa_{k,n}(\theta)|\leq  C h_n^{-1/\alpha+1+\gamma\varsigma}|\theta-\theta_0|V(X_{t_{k-1,n}})W(X_{t_{k-1,n}})
$$
by  \eqref{kappa0} and \eqref{kappa1}, respectively. Therefore,
$$
|y_{n,k}^P(\theta)-\bar y_{n,k}^P(\theta)|\leq C h_n^{-2/\alpha+2+\gamma\varsigma}|\theta-\theta_0|^2V(X_{t_{k-1,n}})^2W(X_{t_{k-1,n}})^2.
$$
Recall that
$$\frac{h_n^{-2/\alpha+2}}{r_n^{2}}=\frac{1}{n},
$$
hence
$$
|Y_n^P(\theta)-\bar Y_n^P(\theta)|\leq C|\theta-\theta_0|^2h_n^{\gamma\varsigma}\left(\frac{1}{n}\sumk V(X_{t_{k-1,n}})^2W(X_{t_{k-1,n}})^2\right).
$$
Since
$$
\frac{1}{n}\sumk V(X_{t_{k-1,n}})^2 W(X_{t_{k-1,n}})^2 = O_\P(1), \quad n\to \infty,
$$
the last estimate completes the proof.
\end{proof}

\subsection{Localization}

We introduce a local coordinate $u$ in the neighborhood of  $\theta_0$ with the rate $r_n$:
$$
\theta=\theta_0+r_n^{-1}u,
$$
and define the associated re-scaled objective function as
$$\ba
G_n(u)&:=r_n^{2}H_n(\theta_0+r_n^{-1}u)
\\&=h_{n}^{-1/\alpha}\sum_{k=1}^n\left(\Big|X_{t_{k,n}}-F_h(\theta_0+r_n^{-1}u; X_{t_{k-1,n}} )\Big|-\Big|X_{t_{k,n}}-F_h(\theta_0; X_{t_{k-1,n}})\Big|\right).
\ea
$$
We also introduce the linearized version
$$\ba
\bar G_n(u)&:=r_n^{2}\bar H_n(\theta_0+r_n^{-1}u)
\\&=h_{n}^{-1/\alpha}\sum_{k=1}^n\left(\Big|X_{t_{k,n}}-\bar F_h(\theta_0+r_n^{-1}u; X_{t_{k-1,n}} )\Big|-\Big|X_{t_{k,n}}-\bar F_h(\theta_0; X_{t_{k-1,n}})\Big|\right).
\ea
$$
Denote
$$
U_{n,\varsigma}=\{u:\theta_0+r_n^{-1}u\in \Theta_{n,\varsigma}\}
=\left\{u:\, |u|\leq r_n^{1-\varsigma}\right\}.
$$

\begin{prp}\label{p51} There exists $\varsigma\in(0,1)$ such that
$$
\sup_{u\in U_{n, \varsigma}}|G_n(u)-\bar G_n(u)|=o_\P(1), \quad n\to \infty.
$$
\end{prp}
\begin{proof} Denote by $U_n^{P,G}(u), U_n^{M,G}(u), Y_n^{P,G}(u), Y_n^{M,G}(u)$ the corresponding parts of the decomposition of $H_n(\theta)$  after the same localization and re-scaling, e.g.
$$
U_n^{P,G}(u):=r_n^{2}U_n^{P}(\theta_0+r_n^{-1}u).
$$
Denote also by $\bar U_n^{P,G}(u), \bar U_n^{M,G}(u), \bar Y_n^{P,G}(u), \bar Y_n^{M,G}(u)$ the similar parts for $\bar H_n(\theta)$. Then, we have the following.
\begin{itemize}
  \item By Lemma \ref{l51},
  $$
  |U_n^{P,G}(u)-\bar U_n^{P,G}(u)|\leq r_n^2\Big(k_{n}^{U,P,\varsigma} r_n^{-1} h_n^{\gamma\varsigma} \Big) |r_n^{-1}u|\leq k_{n}^{U,P,\varsigma} h_n^{\gamma\varsigma} r_n^{1-\varsigma}, \quad u\in U_{n,\varsigma},
  $$
  with  the family $\{k_{n}^{U,P,\varsigma}\}$  bounded in probability;

   \item By Lemma \ref{l52},
  $$
  |U_n^{M,G}(u)-\bar U_n^{M,G}(u)|+|Y_n^{M,G}(u)-\bar Y_n^{M,G}(u)|\leq r_n^2\Big(k_{n}^{M,\varsigma} r_n^{-1} h_n^{\gamma\varsigma} \Big) |r_n^{-1}u|^\upsilon\leq k_{n}^{M,\varsigma} h_n^{\gamma\varsigma} r_n^{1-\upsilon\varsigma}, \quad u\in U_{n,\varsigma},
  $$
    with  the family $\{k_{n}^{M,\varsigma}\}$  bounded in probability;
  \item   By Lemma \ref{l53},
  $$
  |Y_n^{P,G}(u)-\bar Y_n^{P,G}(u)|\leq r_n^2\Big(k_{n}^{Y,P,\varsigma} h_n^{\gamma\varsigma} \Big) |r_n^{-1}u|^2\leq k_{n}^{Y,\varsigma} h_n^{\gamma\varsigma} r_n^{2-2\varsigma}, \quad u\in U_{n,\varsigma}.
  $$
    with  the family $\{k_{n}^{Y,P,\varsigma}\}$  bounded in probability.
\end{itemize}
By Assumption \ref{AssStab}, we have
$$
h_n^{\gamma\varsigma}=o(n^{-(\gamma \varsigma)/(2\delta)}).
$$
Taking $\upsilon, \varsigma$ close enough to 1, we can guarantee
$$
 1-\upsilon\varsigma<\frac{\gamma\varsigma}{2\delta}, \quad 2-2\varsigma <\frac{\gamma\varsigma}{2\delta},
$$
thus proving the required statement.
\end{proof}

\subsection{Limit theorem for the linearized and localized objective functions}

In this section, we establish the following point-wise limit theorem for the (linearized and localized) random field $\bar G_n(\cdot)$. \rev{As usual, $u\cdot v,u,v\in \mathrm{R}^m$ denotes scalar product.}

\begin{prp}\label{p52} The random field $\bar G_n(u), u\in \mathbb{R}^m$ weakly converges, in the sense of finite-dimensional distributions, to
\be\label{G0}
G_0(u)=\Gamma_0^{1/2}\xi\cdot u+\phi_\alpha(0)\Sigma_0 u\cdot u,\quad u\in \mathbb{R}^m,
\ee
 where $\xi$ is  a standard $m$-dimensional Gaussian random vector independent of $\Gamma_0$ and $\Sigma_0$.
\end{prp}

\begin{proof} Denote
$$
\Xi_n=-\frac{1}{\sqrt{n}}\sumk \Big(\sgn (\zeta_{k,n})- \E[\sgn (\zeta_{k,n})|\Ff_{k-1,n}]\Big)V(X_{t_{k-1,n}})\nabla_\theta a(\theta_0;X_{t_{k-1,n}}),
$$
$$
\Gamma_{n}=\frac{1}{n}\sumk V(X_{t_{k-1,n}})^2 \left(\nabla_\theta a(\theta_0;X_{t_{k-1,n}})\right)^{\otimes 2},
$$

$$
\Sigma_n=\frac{\phi_\al(0)}{n}\sumk \frac{V(X_{t_{k-1,n}})}{\sigma (X_{t_{k-1,n}})}\left(\nabla_\theta a(\theta_0, X_{t_{k-1,n}})\right)^{\otimes 2}.
$$
Let us show that, for any $u\in  \mathbb{R}^m$,
\be\label{linearG}
\bar G_n(u)= \Xi_n \cdot u+\Sigma_n u\cdot u+o_\P(1), \quad n\to \infty.
\ee
We will show in Appendix \ref{sB} below that
\be\label{CLT}
(\Xi_n, \Gamma_n, \Sigma_n)\Longrightarrow (\Gamma_0^{1/2}\xi, \Gamma_0, \Sigma_0),
\ee
which together with \eqref{linearG} will yield the required statement. We have $\bar G_n(0)=0$, hence \eqref{linearG} for $u=0$ is trivial; therefore we consider further the case $u\not=0$, only.

We analyse the terms in the decomposition separately for $\bar G_n(u)$. First, we have simply
$$\ba
\bar U_n^{M,G}(u)&=\bar U_n^{M}(\theta_0+r_n^{-1}u)=\sumk \bar u_{k,n}^{M}(\tz+r_{n}^{-1}u)
\\&=-
\sumk \Big(\sgn (\zeta_{k,n})- \E[\sgn (\zeta_{k,n})|\Ff_{k-1,n}]\Big) \bar \kappa_{k,n}(\tz+r_{n}^{-1}u).
\ea
$$
Since
$$
\bar \kappa_{k,n}(\tz+r_{n}^{-1}u) =h_n^{1-1/\alpha}r_{n}^{-1} V(X_{t_{k-1,n}}) \nabla_\theta a(\theta_0; X_{t_{k-1,n}}) \cdot u
= \frac{1}{\sqrt{n}} V(X_{t_{k-1,n}}) \nabla_\theta a(\theta_0; X_{t_{k-1,n}}) \cdot u, 
$$
we obtain
$$
\bar U_n^{M,G}(u)= \Xi_n\cdot u.
$$
That is, the linear part on the right-hand side of \eqref{linearG} equals the martingale part of the `linear' term $\bar U_n^{G}(u)$. Also, it is easy to show that the corresponding predictable part is negligible:
$$
\bar U_n^{P,G}(u)=o_\P(1), \quad n\to \infty.
$$
Indeed, for $u\not=0$, we have \eqref{Delta_neglig} valid for 
$\bar U^P_n(\theta)$: the proof remains literally the same, with $\kappa_{k,n}(\theta)$ in \eqref{kappa0} replaced by $\bar \kappa_{k,n}(\theta)$. 
Then,
$$
\frac{|\bar U_n^{P,G}(u)|}{|u|}=r_n\,\frac{|\bar U^P_n(\tz+r_n^{-1}u)|}{|r_n^{-1}u|}
\le r_n \sup_{\theta\ne\tz}\frac{|\bar U^P_n(\theta)|}{|\theta-\tz|} = o_\P(1), \quad n\to \infty,
$$
which proves the required convergence.

Next, we analyze the predictable and martingale parts of the `quadratic' term $\bar Y_n^{G}(u)$. Denote
$$
Q(x,\upsilon)=\upsilon^{-2}q(x,\upsilon), \quad \upsilon\not=0.
$$
Then, $Q(x,\upsilon)$ is a family of probability densities with respect to $x$ such that
\be\label{Q_conv}
Q(x,\upsilon)\, dx\Rightarrow\delta_0(dx),\quad \upsilon\to 0.
\ee
We have
\begin{align}
\bar Y_n^{P,G}(u)&=
\sumk \left(\bar{\kappa}_{k,n}(\tz+r_{n}^{-1}u)\right)^{2}
\int_{\Re}Q\left(z,\bar \kappa_{k,n}(\tz+r_{n}^{-1}u)\right)\psi_{h_n}(X_{t_{k-1,n}},z)\, dz;
\label{hm:Y2G_eq}
\end{align}
see the proof of Lemma \ref{l53} for the notation $\psi_t(x,z)$. For any $R>0$, by \eqref{Q_conv} and  Theorem \ref{tA2} we have
$$
\int_{\Re}Q\left(z,\upsilon\right)\psi_{t}(x,z)\, dz\to\frac{1}{\sigma(x)}\phi_\al(0), \quad \upsilon\to 0, \quad  t\to 0+
$$
uniformly in $|x|\leq R$.  We have
$$
|\bar \kappa_{k,n}(\tz+r_{n}^{-1}u)|\leq C_R n^{-1/2}
\quad \hbox{ on } \{|X_{t_{k-1,n}}|\leq R\},
$$
hence for any $R$,
\begin{align}
& \sumk \left(\bar{\kappa}_{k,n}(\tz+r_{n}^{-1}u)\right)^{2}
\nn\\
&{}\qquad \times \bigg|
\int_{\Re}Q\left(z,\bar \kappa_{k,n}(\tz+r_{n}^{-1}u)\right)\psi_{h_n}(X_{t_{k-1,n}},z)\, dz
-\frac{1}{\sigma(X_{t_{k-1,n}})}\phi_\al(0)\bigg|\,1_{|X_{t_{k-1,n}}|\leq R}=o_\P(1).
\nn
\end{align}
Since $\psi_t(x,z)$ is bounded and $\sigma(x)$ is separated from $0$, we get
$$\left|
\int_{\Re}Q\left(z,\bar \kappa_{k,n}(\tz+r_{n}^{-1}u)\right)\psi_{h_n}(X_{t_{k-1,n}},z)\, dz-\frac{1}{\sigma(X_{t_{k-1,n}})}\phi_\al(0)\right|\leq C.
$$
We also have (recall Assumption \ref{Ass2}; we may and do set $h_n\le 1$)
$$\ba
\left(\bar{\kappa}_{k,n}(\tz+r_{n}^{-1}u)\right)^{2}&=\frac{1}{n}\left(\nabla_\theta a(\theta_0, X_{t_{k-1,n}})\cdot u\right)^2V(X_{t_{k-1,n}})^2\leq
\frac{|u|^2}{n}V(X_{t_{k-1,n}})^2W(X_{t_{k-1,n}})^2.
\ea$$
Hence
$$\ba
\Upsilon_{n,R}(u)&:=\sumk \left(\bar{\kappa}_{k,n}(\tz+r_{n}^{-1}u)\right)^{2}
\\&\hspace*{1cm}\times\left|
\int_{\Re}Q\left(z,\bar \kappa_{k,n}(\tz+r_{n}^{-1}u)\right)\psi_{h_n}(X_{t_{k-1,n}},z)\, dz-\frac{1}{\sigma(X_{t_{k-1,n}})}\phi_\al(0)\right|1_{|X_{t_{k-1,n}}|> R}
\\& \leq \frac{C|u|^2}{n} \sumk V(X_{t_{k-1,n}})^2W(X_{t_{k-1,n}})^2  1_{|X_{t_{k-1,n}}|> R}
\\&\leq C|u|^2\left(\frac{1}{n} \sumk V(X_{t_{k-1,n}})^p W(X_{t_{k-1,n}})^p\right)^{2/p} \left(\frac{1}{n} \sumk 1_{|X_{t_{k-1,n}}|> R}\right)^{1-2/p}
\ea$$
for any $p>2$. We have
$$
\limsup_{n\to \infty}\P\left(\frac{1}{n} \sumk 1_{|X_{t_{k-1,n}}|> R}>\eps\right)\to 0, \quad R\to \infty
$$
for any $\eps>0$: for the finite observation horizon, this is trivial; for the infinite observation horizon, this follows by ergodicity of the process $X$, see Appendix \ref{sB2.e}. Thus, using  Assumption \ref{A_moment}, we get
$$
\limsup_{n\to \infty}\P(\Upsilon_{n,R}(u)>\eps)\to 0, \quad R\to \infty,
$$
for any $u\in \mathbb{R}^m, \eps>0$. Summarizing all the above yields that, for any $u\in \mathbb{R}^m$,
$$
\bar Y_n^{G}(u)=\Sigma_n u\cdot u+o_p(1), \quad n\to \infty.
$$

Finally, we show that the martingale part  $\bar Y^{M,G}_n(u)$ of the `quadratic' term is negligible:
\be\label{m1_neglig}
\bar Y^{M,G}_n(u)=\sumk\bar y_{k,n}^{M}(\tz+r_{n}^{-1}u)=o_\P(1), \quad n\to \infty.
\ee
Indeed, we have
$$
\E[q(\zeta_{k,n},\bar \kappa_{k,n}(\theta))^2|\Ff_{k-1,n}]=\int_{\Re}q\left(z,\bar \kappa_{k,n}(\tz+r_{n}^{-1}u)\right)^2\psi_{h_n}(X_{t_{k-1,n}},z)\, dz.
$$
Each of the functions $\psi_{h_n}(x,z)$ is a probability density with respect to $z$, and these functions are bounded. In  addition, $q(x,v)$ is bounded by $2|v|$ and supported by $[-|v|, |v|]$. Summarizing all these properties, we get
$$
  \int_{\Re} q\left(z,v\right)^2\psi_{h_n}(x,z)\, dz\leq C(|v|^2\wedge |v|^3)
  $$
 Without loss of generality we can assume $2p\in (2,3]$ in the moment Assumption \ref{A_moment}; for such $p$, the above estimate yields
 $$
  \int_{\Re} q\left(z,v\right)^2\psi_{h_n}(x,z)\, dz\leq C|v|^{2p}.
  $$
  Then, in the infinite observation horizon case, using Assumption \ref{A_moment}, we get that
  $$\ba
 \E\left[\left( \sumk\bar y_{k,n}^{M,G}(\tz+r_{n}^{-1}u)\right)^2\right] 
 &\leq C\E\left[\sumk|\bar \kappa_{k,n}(\tz+r_{n}^{-1}u)|^{2p}\right]
 \\&\leq
\frac{C|u|^p}{n^p}\sumk V(X_{t_{k-1,n}})^{2p}W(X_{t_{k-1,n}})^{2p}\to 0, \quad n\to \infty,
\ea
  $$
  which proves \eqref{m1_neglig}.
  As for the finite observation horizon case, \eqref{m1_neglig} readily follows from the Lenglart domination inequality \cite[Section 2.1.7]{JacPro12}.
Summarizing all the above, we derive the representation \eqref{G0}, which completes the proof. 
\end{proof}

Combining Proposition \ref{p51} and Proposition \ref{p52}, we get the following.

\begin{cor}\label{c51} For any compact set $K\subset \mathbb{R}^m$, the random fields $\{G_n(u): u\in K\}$ and $\{\bar G_n(u): u\in K\}$  weakly converge to $\{G_0(u): u\in K\}$ in $C(K)$.
\end{cor}
\begin{proof} By Proposition \ref{p51}, it is enough to prove the required weak convergence for $\bar G_n(u), u\in K$. These random fields are convex and the limiting random field $G_0(u), u\in \mathbb{R}^m$ is bounded in probability on every ball. Therefore, it is easy to show, using Proposition \ref{p52}, that the (sub-) gradients of  $\bar G_n(u),u\in \mathbb{R}^m$ are bounded in probability on every ball. This yields that the family of the distributions of $\{\bar G_n(u): u\in K\}$ in $C(K)$ is tight. Combined with the convergence of finite-dimensional distributions, this gives the required weak convergence in $C(K)$.
\end{proof}

\subsection{Completion of the proof}

Now we are ready to complete the proofs of Theorems \ref{t1} and \ref{t2}. First, we prove that
\be\label{cons_rate_1}
\hat u_n:=r_n(\hat \theta_n-\theta_0)=\mathop{\mathrm{argmin}}\limits_{u}G_n(u)=O_\P(1),
\ee
which actually shows that \eqref{cons_rate} holds true for $\varsigma=1$. For that purpose, we make a simple observation that, for a convex function $\mathfrak{G}:\mathbb{R}^m\to \mathbb{R}$ with $\mathfrak{G}(0)=0$, for any $R$,
$$
\mathfrak{G}(u)>1, \ |u|=R \quad \Longrightarrow \quad \mathfrak{G}(u)> 1, \ |u|\geq R.
$$
Fix $R>0$ and let $n$ large enough so that $R<r_n^{1-\varsigma}$. 
By the convexity of $\bar{G}_n$ and Corollary \ref{c51},
$$
\liminf_{n\to\infty}\P\left(\inf_{u\in U_{n, \varsigma}:\,|u|\geq R}\bar G_n(u)> 1\right)
\geq \liminf_{n\to\infty}\P\left(\inf_{|u|= R}\bar G_n(u)> 1\right)
\geq \P\left(\inf_{|u|=R} G_0(u)> 1\right).
$$
Choose $\varsigma<1$ close enough to 1, so that Proposition \ref{p51} holds. Then by \eqref{cons_rate} and Proposition \ref{p51}, we have
$$
\P(\hat u_n\in U_{n, \varsigma})\to 1, \qquad 
\limsup_{n\to\infty}\P\left(\sup_{u\in U_{n, \varsigma}}| G_n(u)-\bar G_n(u)|> \frac12\right)=0.
$$
Recall that $G_n(\hat u_n)\leq G_n(0)=0$.
Building on the above observations, we have for any $R>0$,
\begin{align}
\limsup_{n\to\infty}\P(|\hat u_n|\geq R)
&\le \limsup_{n\to\infty}\P(R\le |\hat u_n|\le r_n^{1-\varsigma})
\nn\\
&\le \limsup_{n\to\infty}\P\left(\inf_{u:\,R\le |u|\le r_n^{1-\varsigma}} G_n(u) \le \frac12\right)
\nn\\
&\le 1- \liminf_{n\to\infty}\P\left(\inf_{u:\,R\le |u|\le r_n^{1-\varsigma}} G_n(u) > \frac12\right)
\nn\\
&\le 1- \liminf_{n\to\infty}\P\left(\inf_{u:\,R\le |u|\le r_n^{1-\varsigma}} \bar{G}_n(u) > 1\right)
\nn\\
&{}\qquad + \limsup_{n\to\infty}\P\left(\sup_{u\in U_{n,\varsigma}} |G_n(u)-\bar{G}_n(u)| > \frac12\right)
\nn\\
&\leq 1 - \P\left(\inf_{|u|=R} G_0(u)> 1\right).
\nn
\end{align}
The limiting random field $G_0(u)$ is a quadratic function with the matrix $\phi_\alpha(0)\Sigma_0$ of the quadratic part; see \eqref{G0}.
This matrix is non-degenerate by the local identifiability assumption, hence
$$
\P\left(\inf_{|u|=R} G_0(u)> 1\right)\to 1, \quad R\to \infty,
$$
which completes the proof of \eqref{cons_rate_1}.

Combining \eqref{cons_rate_1} with Corollary \ref{c51}, 
we can apply the standard argmin argument to deduce that $\hat u_n\Rightarrow \hat u_0$, 
where
$$
\hat u_0 := - \frac1{2\phi_\alpha(0)}\Sigma_0^{-1}\Gamma_0^{1/2}\xi
$$
is the unique minimal point $\hat u_0$ of the random field $G_0(u)$.
This random point has the (mixed) normal distribution with the parameters $0$ and $\left(2\phi_{\al}(0)\right)^{-2}\Sigma_{0}^{-1}\Gamma_{0}\Sigma_{0}^{-1}$, 
thus completing the proofs of Theorems \ref{t1} and \ref{t2}.

\section{Consistent estimator of asymptotic covariance matrix}
\label{hm:sec_ceAV}

The objective of this section is to construct a consistent estimator of the asymptotic possibly random covariance matrix
\begin{equation}
\mathcal{A}_0 := \left(2\phi_{\al}(0)\right)^{-2}\Sigma_{0}^{-1}\Gamma_{0}\Sigma_{0}^{-1}
\label{hm:AV}
\end{equation}
of the LAD estimator.
To that end, we will separately consider $\phi_{\al}(0)$, $\Sigma_0$, and $\Gamma_0$. 

\textit{Throughout this section, we keep the assumptions given in Sections \ref{hm:sec_setup} and \ref{sec:main.results} valid.} 
Just for brevity, we here suppose that $T_n\equiv T\in(0,\infty)$ for the finite observation horizon case.
Recall that, in case of $T_n\to\infty$, we denote by $q>0$ the tail index of the {\lp} $Z$ (see Assumption \ref{A_diss}); this implies that $\E[|Z_1|^{q}]<\infty$. Denote by $\|\cdot\|_{TV}$ the total variation norm and by 
\begin{equation}
    P_t(\tz; x, dy):=\P[X_t\in dy|X_0=x]
\end{equation}
the transition probability function of $X$ \rev{under $\P=\P_{\theta_0}$.} 
Denote by $\P^{X_0}$ the marginal distribution of $X_0$.
Throughout this section, we will work under the following additional conditions
\footnote{
There may be some redundancy with the main assumptions given in Section \ref{sec:main.results}, but no confusion will occur.
}.

\begin{asm}
\label{hm:asm-1}
\begin{enumerate}
    \item The index $\al$ and the finite range drift-H\"{o}lder exponent $\eta\in(0,1]$ satisfy 
    \begin{equation}
    \al > \frac{1}{\eta+1}.
    \end{equation}

    \item When $T_n\to\infty$, we additionally have
    \begin{equation}\nn
    \sup_{s\ge 0}\E[|X_s|^q]<\infty
    \end{equation}
    and also
    \begin{equation}
\label{hm:exp.mix}
    \int \|P_t(\tz; x, dy)-\pi(\tz; dy)\|_{TV} \P^{X_0}(dx) \lesssim (1+t)^{-\mathfrak{a}}
\end{equation}
    for some $\mathfrak{a}>0$. Further, $T_n \gtrsim h_n^{-c'}$ for some $c'>0$.
\end{enumerate}
\end{asm}

\rev{Assumption \ref{hm:asm-1}.1 is stronger than the balance condition \eqref{balance} and  entails} 
\begin{equation}\label{hm:al>1/2}
    \al>\frac12
\end{equation}
\begin{equation}
    \al + \min\{\al, 1\} \, \eta>1.
    \label{hm:lem_1-8}
\end{equation}

Let
\begin{equation}
    f(x,\theta):=(\nabla_\theta a(\theta;x))^{\otimes 2}.
\end{equation}

\begin{asm}
\label{hm:asm-2}
The function $f(x,\theta)$ is continuously differentiable in $\theta$ for each $x$, and we have the following.
\begin{enumerate}
    \item When $T_n\equiv T$,
    \begin{align}
        & \sup_\theta |\partial_\theta^l f(x,\theta)| \lesssim 1+|x|^C, \qquad l\in\{0,1\},\nn\\
        & \sup_\theta |f(x,\theta) - f(y,\theta)| \lesssim (1+|x|+|y|)^C |x-y|.
    \end{align}
    
    \item When $T_n\to\infty$, the (nonnegative and bounded) weight function $V(x)$ is globally Lipschitz and the function $(x,\theta) \mapsto \left(V(x) f(x,\theta), V(x) \partial_\theta f(x,\theta)\right)$ is essentially bounded.
    Moreover,
    \begin{equation}\label{hm:asm-2-2}
        \lim_{\eps\to 0} \sup_\theta  \sup_{|x-y|\le \eps}
        \big|V(x) f(x,\theta) - V(y)f(y,\theta) \big| = 0.
    \end{equation}
\end{enumerate}
\end{asm}

Then, the property \eqref{hm:asm-2-2} also holds with $V(\cdot)$ replaced by $V(\cdot)^2/\sig(\cdot)$.

\subsection{Activity index}
\label{hm:activity.index.estimation}

In this section, we will construct an estimator $\hat{\al}_n$ of the index $\al$ such that
\begin{equation}
\exists \kappa>0,\quad h_n^{-\kappa}(\hat{\al}_n - \al) = O_\P(1)
\label{hm:def_aes}
\end{equation}
with leaving the scale coefficient $\sig$ unknown, for both cases $T_n\to\infty$ and $T_n\equiv T\in(0,\infty)$. 
Under \eqref{hm:def_aes}, we have $\aes\cip\al$ \rev{ and by \eqref{phi_al0} }
\begin{align}
& \phi_{\hat{\al}_n}(0) = \frac{1}{\pi}\Gamma\left(1+\hat{\al}_n^{-1}\right) \cip \frac{1}{\pi}\Gamma\left(1+\al^{-1}\right) =\phi_{\al}(0),
\label{hm:h^al-rate-gam} \\
& h_n^{-1/\hat{\al}_n} = h_n^{-1/\al}(1+o_\P(h_n^{\kappa-\varepsilon}))
\label{hm:h^al-rate}
\end{align}
for any $\varepsilon>0$.
The latter ensures that $\sqrt{n}h_n^{1-1/\alpha}(\tes-\tz)$ and $\sqrt{n}h_n^{1-1/\aes}(\tes-\tz)$ are asymptotically equivalent, hence have the same asymptotic mixed-normal distribution.

Let us denote the $k$th increment of $X$ by $\Delta_k X = X_{t_{k,n}} - X_{t_{k-1,n}}$.
To estimate $(\phi_{\al}(0), \Sigma_0, \Gamma_0)$, we will make use of the second-order increments
\begin{equation}
\Delta_k^2 X := \Delta_k X - \Delta_{k-1}X = X_{t_{k,n}} + X_{t_{k-2,n}} - 2 X_{t_{k-1,n}}
\nonumber
\end{equation}
for $2\le k\le n$, which quantitatively weakens the effect of the drift term. 
We also define the third-order differences with one lag by $\Delta_k^2 X - \Delta_{k-2}^2 X$ for $4\le k\le n$.
Analogous symbols are used for the increments of $Z=Z^{(\al)}+Z^{\triangle}$ (see \eqref{hm:Z-decomposition}). 
Recall that $h_n^{-1/\al} Z_{h_n} = h_n^{-1/\al} Z^{(\al)}_{h_n} + o_p(1) \wc S_\al$ ($h_n\to 0$) under the present assumptions, where $S_\al$ denotes the $\al$-stable distribution corresponding to the density $\phi_\al$. It follows that
\begin{align}
(2h_n)^{-1/\al}\Delta_{k}^2 Z \wc S_{\al},\qquad 
(4h_n)^{-1/\al}(\Delta_{k}^2 Z -\Delta_{k-2}^2 Z) \wc S_{\al}
\nn
\end{align}
for each $k\ge 4$.

For notational convenience, we will often write $\int_{k}$ for $\int_{t_{k-1,n}}^{t_{k,n}}$ and denote $a_s=a(X_s,\tz)$ and $\sig_s=\sig(X_s)$.
With slight abuse of notation, we further abbreviate as $a_{k}=a_{t_{k,n}}$, $\sig_{k}=\sig_{t_{k,n}}$, and so on.
Then, we can write
\begin{align}\label{hm:inc.sum}
(2h_n)^{-1/\al}\Delta_k^2 X &= \sig_{k-2} (2h_n)^{-1/\al}\Delta_{k}^2 Z + 2^{-1/\al} \overline{r}_k,
\end{align}
where $\overline{r}_{j}=r_{a,k}+r_{\sig,k}$ with
\begin{align}
r_{a,k} &:= h_n^{-1/\al} \left(\int_{k} (a_s - a_{k-2})ds - \int_{k-1} (a_s - a_{k-2})ds\right),
\nn\\
r_{\sig,k} &:= h_n^{-1/\al} \left(\int_{k} (\sig_{s-} - \sig_{k-2})dZ_s - \int_{k-1} (\sig_{s-} - \sig_{k-2})dZ_s\right).
\nonumber
\end{align}
Further by \eqref{hm:inc.sum},
\begin{align}
(4h_n)^{-1/\al}(\Delta_{k}^2 X -\Delta_{k-2}^2 X)
&= \sig_{j-4}(4h_n)^{-1/\al}(\Delta_{k}^2 Z -\Delta_{k-2}^2 Z)
\nn\\
&{}\qquad + 
(\sig_{k-2}-\sig_{k-4})(4h_n)^{-1/\al}\Delta_{k}^2 Z
+4^{-1/\al}(\overline{r}_k - \overline{r}_{k-2}).
\label{hm:inc.sum2}
\end{align}


To construct an estimator of $\al$, we introduce
\begin{align}\label{hm:def_PV}
\mathsf{H}_{1,n}(\rho) &:= \frac{1}{n-1}\sum_{k=2}^{n}\big|(2h_n)^{-1/\al}\Delta_k^2 X\big|^{\rho},
\end{align}
where $\rho\in(0,\min\{1,\al\})$, and where $\rho\in(0,\min\{1,\al,q\})$ when $T_n\to\infty$; recall that $q>0$ is the tail index given in Assumption \ref{A_diss}.

We begin with handling the remaining terms. 

\begin{lem}
\label{hm:lem_1}
Pick a constant $\rho>0$ such that
\begin{itemize}
    \item $\rho\in\left(0, \min\{1, \al\}\right)$ for $T_n\equiv T$;
    \item $\rho\in\left(0, \min\{1, \al, q\}\right)$ for $T_n\to\infty$.
\end{itemize}
Then, there exists a constant $\kappa_1>0$ 
such that
\begin{equation}
\max\bigg\{\frac{1}{n-1}\sum_{k=2}^{n}|r_{a,k}|^\rho ,\, \frac{1}{n-1}\sum_{k=2}^{n}|r_{\sig,k}|^\rho \bigg\} = O_\P(h_n^{\kappa_1}).
\nonumber
\end{equation}
\end{lem}

\begin{proof}
Recalling the definitions of $r_{a,k}$ and $r_{\sig,k}$, we will only show that
\begin{align}
\bar{U}_{a,n}(\rho) &:= \frac{1}{n-1}\sum_{k=2}^{n}
\left| h_n^{-1/\al} \int_{k} (a_s - a_{k-2})ds \right|^\rho 
= O_\P(h_n^{\kappa_1}),
\label{hm:lem_1-1}\\
\bar{U}_{\sig,n}(\rho) &:= \frac{1}{n-1}\sum_{k=2}^{n}
\left| h_n^{-1/\al} \int_{k} (\sig_{s-} - \sig_{k-2})dZ_s \right|^\rho 
= O_\P(h_n^{\kappa_1})
\label{hm:lem_1-2}
\end{align}
hold for some $\kappa_1>0$; the other two terms (associated with $\int_{k-1}$) can be handled similarly.
We note that
\begin{align}
\sup_\theta |a(\theta; x)-a(\theta; y)| &\leq C\left(|x-y|^\eta 1_{|x-y|\leq 1} +|x-y| 1_{|x-y|> 1}\right),
\label{hm:a_reg} \\
|\sig(x)-\sig(y)| &\leq C\left(|x-y|^\zeta 1_{|x-y|\leq 1} +|x-y| 1_{|x-y|> 1}\right),
\label{hm:s_reg} 
\end{align}
under the finite range H\"older conditions for the coefficients ($\eta,\zeta\in(0,1]$) in Assumption \ref{Ass1}. In particular, 
\begin{equation}
    \sup_\theta |a(\theta; x)| \le C (1+|x|).
\end{equation}

Consider \eqref{hm:lem_1-1}.
We have $|a_s - a_{k-2}| \lesssim |X_s - X_{t_{k-2,n}}| + |X_s - X_{t_{k-2,n}}|^\eta$ by \eqref{hm:a_reg}, hence it suffices to show that
\begin{align}
\frac{1}{n-1}\sum_{k=2}^{n} 
\left( h_n^{-1/\al} \int_{k} |X_s - X_{t_{k-2,n}}|^{\eta'} ds \right)^\rho = O_\P(h_n^{\kappa_1})
\label{hm:lem_1-3}
\end{align}
for $\eta'\in\{\eta,1\}$. 
Inserting the expression
\begin{equation}\label{hm:index-p+1}
    X_s-X_{t_{k-2,n}} = \int_{t_{k-2,n}}^s \sig_{u-}dZ_u + \int_{t_{k-2,n}}^s a_u du,
\end{equation}
we can bound the left-hand side of \eqref{hm:lem_1-3} by a constant multiple of
\begin{equation}\label{hm:lem_1-4}
    \bar{U}'_{a,n}(\rho) \, h_n^{(1-1/\al+\eta'/\al)\rho} 
    + \bar{U}''_{a,n}(\rho) \, h_n^{(1-1/\al+\eta')\rho},
\end{equation}
where
\begin{align}
\bar{U}'_{a,n}(\rho) &:= 
\frac{1}{n-1}\sum_{k=2}^{n}
\left( \frac{1}{h_n} \int_k \left|h_n^{-1/\al} \int_{t_{k-2,n}}^s \sig_{u-}dZ_u\right|^{\eta'}ds \right)^\rho,
\nn\\
\bar{U}''_{a,n}(\rho) &:= 
\frac{1}{n-1}\sum_{k=2}^{n}
\left\{\frac{1}{h_n} \int_k \left(\frac{1}{h_n} \int_{t_{k-2,n}}^s(1+|X_u|)du\right)^{\eta'}ds \right\}^\rho.
\nn
\end{align}
Since $1-1/\al+\eta'/\al>0$, to complete the proof it suffices to show that $\bar{U}'_{a,n}(\rho) \, h_n^{(1-1/\al+\eta'/\al)\rho} = O_\P(h_n^{\kappa_1})$ for some $\kappa_1>0$ and that $\bar{U}''_{a,n}(\rho) = O_\P(1)$. Then, the quantity \eqref{hm:lem_1-4} becomes $O_\P(h_n^{\kappa_1})$ for a smaller $\kappa_1>0$ if necessary, hence followed by \eqref{hm:lem_1-3}.

For $\bar{U}'_{a,n}(\rho)$, we will use the decomposition $Z=Z^{(\al)}+Z^{\triangle}$ (see \eqref{hm:Z-decomposition} in Section \ref{hm:sec_moments}). 
First, since $Z^{\triangle}_t=\sum_{0<s\le t}\Del Z^{\triangle}_s$ and $\eta',\rho\le 1$, we have
\begin{align}
& \frac{1}{n-1}\sum_{k=2}^{n}
\left(\frac{1}{h_n} \int_k \left|h_n^{-1/\al} \int_{t_{k-2,n}}^s \sig_{u-}dZ^\triangle_u\right|^{\eta'}ds \right)^\rho
\nn\\
&\lesssim 
\frac{1}{n-1}\sum_{k=2}^{n}h_n^{-\eta'\rho/\al}\sum_{t_{k-2,n}<s\le t_{k,n}}|\Del Z^{\triangle}_s|^{\rho\eta'}
\nn\\
&=
\frac{1}{n-1}\sum_{k=2}^{n}h_n^{-\eta'\rho/\al}
\int_{t_{k-2,n}}^{t_{k,n}} |z|^{\rho\eta'}\int N^{\Del}(ds,dz)=O_\P(h^{1-\eta'\rho/\al}).
\label{hm:lem_1-6}
\end{align}
Write $N_\al$ for the Poisson point measure associated with the {\lm} $\mu_\al$, and then denote
\begin{align}
    Z^{(\al)}_t 
    &= \int_0^t \int_{|z|\le 1}z \widetilde{N}_{\al}(ds,dz) + \int_0^t \int_{|z|>1}z N_{\al}(ds,dz)
    \nn\\
    &=: \widetilde{Z}^{(\al),s}_t + Z^{(\al),l}_t.
    \label{hm:lem_1-6.5}
\end{align}
For the large-jump part, we have a similar estimate of \eqref{hm:lem_1-6} with $Z^{\triangle}$ therein replaced by $Z^{(\al),l}$. As for the compensated small-jump (locally $\al$-stable) part, we can 
bound the moment as follows:
\begin{align}
& \E\left[\frac{1}{n-1}\sum_{k=2}^{n}
\left(\frac{1}{h_n} \int_k \left|h_n^{-1/\al} \int_{t_{k-2,n}}^s \sig_{u-} d\widetilde{Z}^{(\al),s}_u \right|^{\eta'}ds \right)^\rho \right]
\nn\\
& \lesssim 
\frac{1}{n-1}\sum_{k=2}^{n} 
\left(\frac{1}{h_n} \int_k \E\left[\left|h_n^{-1/\al} \int_{t_{k-2,n}}^s \sig_{u-} d\widetilde{Z}^{(\al),s}_u \right|^{\eta'} \right] ds \right)^\rho
\nn\\
& \lesssim 
\frac{1}{n-1}\sum_{k=2}^{n} 
\left( \E\left[\E\left[ \sup_{v\in[t_{k-2,n},t_j]}\left|h_n^{-1/\al} \int_{t_{k-2,n}}^v \sig_{u-} d\widetilde{Z}^{(\al),s}_u \right|^{\eta'} \,\middle| \, \mcf_{t_{k-2,n}} \right] \right]
\right)^\rho
\nn\\
&\lesssim h_n^{(1-\eta'/\al)\rho}I(\al<\eta') + (\log(1/h_n))^\rho I(\al=\eta') + 
I(\al>\eta'),
\label{hm:lem_1-7}
\end{align}
where, in the first and last steps, we used the concavity of $y\mapsto y^\rho$ ($y\ge 0$) and Lemma \ref{hm:lem_moment.ineq} below, respectively.
By \eqref{hm:lem_1-6} and \eqref{hm:lem_1-7}, we conclude that 
\begin{equation}\nn
    \bar{U}'_{a,n}(\rho)\lesssim O_\P(h^{1-\eta'\rho/\al}) + O_\P(h_n^{(1-\eta'/\al)\rho}) I(\al<\eta') + O_\P\left(\log(1/h_n)^\rho\right) I(\al\ge \eta'),
\end{equation}
followed by
\begin{align}
\bar{U}'_{a,n}(\rho) \, h_n^{(1-1/\al+\eta'/\al)\rho} 
&\lesssim 
O_\P(h_n^{(1-1/\al)\rho+1}) + O_\P(h_n^{(2-1/\al)\rho}) \nn\\
&{}\qquad + O_\P\left(h_n^{(1-1/\al+\eta'/\al)\rho}\log(1/h_n)^\rho\right).
\end{align}
Under \eqref{hm:al>1/2} and \eqref{hm:lem_1-8}, the right-hand side becomes $O_\P(h_n^{\kappa_1})$ for some $\kappa_1>0$.
Note that, under the essential boundedness of $\sig$, the above argument is valid for both $T_n\to\infty$ and $T_n\equiv T$.

Turning to $\bar{U}''_{a,n}(\rho)$, we first note that
\begin{equation}\nn
    \bar{U}''_{a,n}(\rho) \lesssim 
    \frac{1}{n-1}\sum_{k=2}^{n}
    \left(1+\sup_{s\in[t_{k-2,n},t_{k,n}]}|X_s|^{\eta' \rho}\right).
\end{equation}
Again by substituting the expression \eqref{hm:index-p+1} into $X_s$,
\begin{align}
    \sup_{s\in[t_{k-2,n},t_{k,n}]}|X_s|^{\eta' \rho}
    &\lesssim |X_{t_{k-2,n}}|^{\eta' \rho} + h_n^{\eta' \rho}
    \left(1+\sup_{s\in[t_{k-2,n},t_{k,n}]}|X_s|^{\eta' \rho}\right)
    \nn\\
    &{}\qquad +\sup_{s\in[t_{k-2,n},t_{k,n}]}\left|
    \int_{t_{k-2,n}}^{s}\sig_{u-}dZ_u
    \right|^{\eta' \rho}.
    \nn
\end{align}
This implies that for $h_n$ small enough,
\begin{align}
    \sup_{s\in[t_{k-2,n},t_{k,n}]}|X_s|^{\eta' \rho}
    &\lesssim |X_{t_{k-2,n}}|^{\eta' \rho} + h_n^{\eta' \rho}
    +\sup_{s\in[t_{k-2,n},t_{k,n}]}\left|
    \int_{t_{k-2,n}}^{s}\sig_{u-}dZ_u
    \right|^{\eta' \rho}.
    \nn
\end{align}
Taking the conditional expectation and applying Lemma \ref{hm:lem_moment.ineq} below \rev{(note that $\eta'\rho\in(0,\al)$)} together with recalling the decomposition $Z=Z^{(\al)}+Z^{\triangle}$ and the argument \eqref{hm:lem_1-6}, we have, with localization through $\tau_n=\inf\{t\ge 0:\, |X_t|\ge n\}$ if necessary,
\begin{align}\label{hm:lem_1-5}
    \E\left[\sup_{s\in[t_{k-2,n},t_{k,n}]}|X_s|^{\eta' \rho}
    \,\middle|\, \mcf_{t_{k-2,n}} \right]
    &\lesssim |X_{t_{k-2,n}}|^{\eta' \rho} + h_n^{\eta' \rho} + h_n^{\eta' \rho/\al} \lesssim |X_{t_{k-2,n}}|^{\eta' \rho}+1.
\end{align}
We conclude that $\bar{U}''_{a,n}(\rho)=O_\P(1)$ for both $T_n\to\infty$ and $T_n\equiv T$, by observing the following.
\begin{itemize}
    \item For $T_n\equiv T$, without loss of generality, we may and do suppose that $\sup_{t \le T}\E[|X_t|^K]<\infty$ for every $K>0$.
    \footnote{
Since \eqref{hm:def_aes} is a weak property, to deduce it we may restrict our attention to any convenient event whose probability we can control to be as large as possible, known as the localization procedure; see \cite[Section 4.4.1]{JacPro12} for a detailed account. 
In our case, this essentially amounts to removing large jumps of $Z$. A concise explanation can be found in \cite[Section 6.1]{Mas19spa}.
}
    Hence by \eqref{hm:lem_1-5}, we get $\E[\bar{U}''_{a,n}(\rho)]=O(1)$.

    \item For $T_n\to\infty$, by \eqref{hm:lem_1-5} we directly get
    \begin{equation}
        \E[\bar{U}''_{a,n}(\rho)] \lesssim 
    1+\sup_{s\ge 0}\E\big[|X_s|^{\eta' \rho}\big] \lesssim 
    1+\sup_{s\ge 0}\E\big[|X_s|^{q}\big]< \infty.
    \end{equation}
\end{itemize}
The proof of \eqref{hm:lem_1-1} is thus complete.

We can deduce \eqref{hm:lem_1-2} with the same technicalities as in handling $\bar{U}'_{a,n}(\rho)$, hence we omit the details.
\end{proof}

For $\widetilde{Z}^{(\al),s}$ given in \eqref{hm:lem_1-6.5}, \rev{it follows by construction that $\E[\widetilde{Z}^{(\al),s}_v]=0$ and $\E[|\widetilde{Z}^{(\al),s}_v|^k]=0$ for all $v\ge 0$ and $k>0$.}

\begin{lem}
\label{hm:lem_moment.ineq}
Suppose that $\xi=(\xi_t)_{t\ge 0}$ is $(\mcf_t)$-adapted and essentially bounded.
For $r\in(0,1)$ and $t\ge 0$, we have (a.s.)
\begin{equation}
    \E\left[\sup_{v\in[t,t+h]}\left|\int_t^{v} \xi_{u-}d\widetilde{Z}^{(\al),s}_u\right|^r \,\middle|\,\mcf_t\right]
    \le C
    \left\{
    \begin{array}{cl}
      h & (\al<r) \\
      h \log(1/h)  & (\al=r) \\
      h^{r/\al} & (\al>r)
    \end{array}
    \right.,
\end{equation}
where the constant $C$ does not depend on $h>0$.
\end{lem}

\begin{proof}
\rev{
We note that, if $\xi\equiv 1$, then direct applications of Theorem 1, Theorem 3, and Theorem 2 of \cite{LusPag08} give the claimed upper bounds for $\al<r$, $\al=r$, and $\al>r$, respectively; in our case, we have $\underline{l}(t)\equiv 1$ in their notation. 
With minor modifications, we can adapt the proofs of the three theorems through the Poisson compensation formula and the Burkholder inequality (see \cite[Section 2.1.5]{JacPro12} for a systematic account), which are valid for general stochastic integrals with respect to a L\'{e}vy process and a predictable integrand. 
For example, for $\al<r\le 1$, we have $\int |z|^r \mu(dz)<\infty$ and we can bound the expectation by
\begin{equation}
    \E\left[\sum_{v\in[t,t+h]} |\xi_{v-}|^r |\Delta \widetilde{Z}^{(\al),s}_v|^r \,\middle|\,\mcf_t\right]
    =
    \E\left[\int_t^{t+h}\int |\xi_{v-}|^r |z|^r \mu(dz)\,\middle|\,\mcf_t\right] \lesssim h.
\end{equation}
We can deduce the other cases in similar manners to \cite{LusPag08}; the details are omitted.
}
\end{proof}

\medskip

We proceed with further estimates to deduce the asymptotic behavior of $\mathsf{H}_{1,n}(\rho)$ given by \eqref{hm:def_PV}, and to estimate the convergence rate. By Lemma \ref{hm:lem_1} and \eqref{hm:inc.sum}, we have
\begin{align}\label{hm:lem_2-1}
\left| \mathsf{H}_{1,n}(\rho) - \frac{1}{n-1}\sum_{k=2}^{n} \sig_{k-2}^\rho \big|(2h_n)^{-1/\al}\Delta_{k}^2 Z\big|^{\rho} \right|
&\le \frac{1}{n-1}\sum_{k=2}^{n}|\overline{r}_k|^\rho = O_\P(h_n^{\kappa_1}).
\end{align}
Denote by $\mathfrak{m}_\al(\rho)$, $\rho\in(0,\al)$, the $\rho`$th absolute moment of $S_{\al}$ (see \cite[Example 25.10]{Sat99}):
\begin{equation}\label{hm:def_m.al}
\mathfrak{m}_\al(\rho) := \E\big[|Z^{(\al)}_1|^\rho\big] = \frac{2^\rho \,\Gamma((\rho+1)/2)}{\sqrt{\pi}\,\Gamma(1-\rho/2)}\, \Gamma(1-\rho/\al).
\end{equation}
By the independence of the increments and the triangular inequality,
\begin{align}
& \left|\frac{1}{n-1}\sum_{k=2}^{n} \sig_{k-2}^\rho \big|(2h_n)^{-1/\al}\Delta_{k}^2 Z\big|^{\rho} 
- \mathfrak{m}_\al(\rho) \frac{1}{n-1}\sum_{k=2}^{n} \sig_{k-2}^\rho \right|
\nn\\
&\le \frac{1}{\sqrt{n-1}} \left|\frac{1}{\sqrt{n-1}}\sum_{k=2}^{n} \sig_{k-2}^\rho \left(\big|(2h_n)^{-1/\al}\Delta_{k}^2 Z\big|^{\rho} 
- \E\left[\big|(2h_n)^{-1/\al}\Delta_{k}^2 Z\big|^{\rho}\right]\right)\right|
\nn\\
&{}\qquad 
+ \left|\E\left[\big|(2h_n)^{-1/\al}\Delta_{2}^2 Z\big|^{\rho}\right] - \mathfrak{m}_\al(\rho) \right| \frac{1}{n-1}\sum_{k=2}^{n} \sig_{k-2}^\rho
=: \delta'_{n} + \delta''_{n}.
\label{hm:lem_2-2}
\end{align}

In what follows, we will assume that
\begin{equation}\label{hm:alpha.esti.rho}
    \rho\in\left\{
    \begin{array}{ll}
        \ds{\left(0,\, 1_{\{\beta=0\}}+1_{\{\beta>0\}}\beta \right)}, & T_n\equiv T, \\[2mm]
        \ds{\left(0,\, \left( 1_{\{\beta=0\}}+1_{\{\beta>0\}}\beta \right) \wedge \frac{q}{2} \right)}, & T_n\to\infty.
    \end{array}
    \right.
\end{equation}

\begin{lem}
\label{hm:lem_2}
Under \eqref{hm:alpha.esti.rho},
\begin{enumerate}
    \item $\delta'_{n}=O_\P(n^{-1/2})$, and
    \item $\delta''_{n}=O_\P(h^{\kappa_2})$ for some $\kappa_2>0$.
\end{enumerate}
These are valid for both $T_n\equiv T$ and $T_n\to\infty$.
\end{lem}

\begin{proof}
The first claim easily follows from the Burkholder inequality 
together with the boundedness of $\sig$. 
For the second one, we apply the moment estimate Lemma \ref{hm:lem_moment-1} in Section \ref{hm:sec_moments} together with the representation $Z=Z^{(\al)}+Z^\triangle$, the fact that $(2h_n)^{-1/\al}\Delta_{2}^2 Z^{(\al)}$ has the same distribution as $Z^{(\al)}_1$, and the inequality $||x+\eps|^r - |x|^r|\le |\eps|^r$ valid for any $r\in(0,1]$:
\begin{align}
\left| \E\left[\big|(2h_n)^{-1/\al}\Delta_{2}^2 Z\big|^{\rho}\right] - \mathfrak{m}_\al(\rho) \right|
&\lesssim \E\left[\big|(2h_n)^{-1/\al}\Delta_{2}^2 Z^{\triangle}\big|^{\rho}\right] = O(h_n^{\kappa_2})
\nn
\end{align}
for some $\kappa_2>0$.
\end{proof}

Under the sampling-design condition \eqref{h_step}, the estimates \eqref{hm:lem_2-1} and \eqref{hm:lem_2-2} conclude that
\begin{align}\label{hm:al.est+1}
\left| \mathsf{H}_{1,n}(\rho) - 
\mathfrak{m}_\al(\rho) \frac{1}{n-1}\sum_{k=2}^{n} \sig_{k-2}^\rho \right|
&= O_\P(h_n^{\kappa_3}).
\end{align}
for some $\kappa_3>0$.

Now, we introduce the a.s. finite quantity:
\begin{equation}
\mathsf{H}_0(\rho):=
\mathfrak{m}_\al(\rho) \times
\left\{
\begin{array}{ll}
\displaystyle{\frac1T \int_0^T \sig_s^\rho ds} & (T_n\equiv T) \\[3mm]
\displaystyle{\int \sig(x)^\rho\pi(\tz,dx)} & (T_n\to\infty).
\end{array}
\right.
\nonumber
\end{equation}

\begin{lem}
\label{hm:lem_4}
Under \eqref{hm:alpha.esti.rho}, we have for some $\kappa>0$,
\begin{equation}
\mathsf{H}_{1,n}(\rho) = \mathsf{H}_0(\rho) + O_\P(h_n^{\kappa}).
\nonumber
\end{equation}
\end{lem}

\begin{proof}
First, consider the case $T_n\equiv T$. 
Thanks to \eqref{hm:lem_1-2} and the boundedness of the scale coefficient $\sig$, 
we have
\begin{align}
& \left|\frac{1}{n-1}\sum_{k=2}^{n} \sig_{k-2}^\rho - \frac{1}{T}\int_0^T \sig_s^\rho ds\right|
\nn\\
&= \left|\frac{1}{n-1}\sum_{k=2}^{n} \sig_{k-2}^\rho 
- \frac{1}{n}\sum_{k=1}^{n}\frac{1}{h_n}\int_k \sig_{s}^\rho ds \right|
\nn\\
&\lesssim \frac{1}{n}\sum_{k=1}^{n}\frac{1}{h_n}\int_k |\sig_{s}^\rho - \sig_{k-1}^\rho| ds 
+ O_\P\left(\frac1n\right)
\nn\\
&\lesssim 
\frac{1}{n}\sum_{j=1}^{n}\frac{1}{h_n}\int_k \max\left\{|X_s-X_{t_{k-1}}|^\rho,\,
|X_s-X_{t_{j-1}}|^{\zeta \rho} \right\} ds 
+ O_\P\left(\frac1n\right).
\label{hm:lem_4-1}
\end{align}
In a quite similar way to prove \eqref{hm:lem_1-3} (with assuming that $\sup_{s\ge 0}\E[|X_s|^{\rho}]<\infty$ without loss of generality), we can deduce that the first term in \eqref{hm:lem_4-1} is $O_\P(h_n^{\kappa})$ for some $\kappa>0$. This completes the proof for $T_n\equiv T$. 

Next, we consider the case $T_n\to\infty$. Noting that
\begin{equation}
    \E[\sig_{k-2}^\rho]=\iint \sig(y)^\rho P_{t_{k-2,n}}(\tz; x, dy)\P^{X_0}(dx),
\end{equation}
we will make use of the following estimate:
\begin{equation}
    \left|\frac{1}{n-1}\sum_{k=2}^{n} \sig_{k-2}^\rho - \int \sig(x)^\rho\pi(\tz,dx)\right|
\le \overline{\delta}'_n(\rho) + \overline{\delta}''_n(\rho),
\end{equation}
where
\begin{align}
\overline{\delta}'_n(\rho)&:= 
\left|\frac{1}{n-1}\sum_{k=2}^{n} \iint \sig(y)^\rho \left(
P_{t_{k-2,n}}(\tz; x, dy) -\pi(\tz; dy) \right)\P^{X_0}(dx) \right|,
\nn\\
\overline{\delta}''_n(\rho) &:= 
\left| \frac{1}{n-1}\sum_{k=2}^{n} \left(\sig_{k-2}^\rho - \E[\sig_{k-2}^\rho]\right) \right|.
\nn
\end{align}

For $\overline{\delta}'_n(\rho)$, by Assumption \ref{hm:asm-1}.2 we get
\begin{align}
\overline{\delta}'_n(\rho)
&\lesssim 
\frac{1}{n-1}\sum_{k=2}^{n} \int 
\| P_{t_{k-2,n}}(\tz; x, \cdot) -\pi(\tz; \cdot) \|_{TV} \P^{X_0}(dx) \nn\\
&\lesssim 
\frac{1}{n-1}\sum_{k=2}^{n} (1+t_{k-2,n})^{-\mathfrak{a}} 
= \frac{1}{n-1}\sum_{k=0}^{n-2} \left(1+\frac{k}{n}T_n\right)^{-\mathfrak{a}} 
\nn\\
&\lesssim \frac1n + \sum_{k=1}^{n}\int_{(k-1)/n}^{k/n}(1+u T_n)^{-\mathfrak{a}}du 
\lesssim \frac1n + \frac{1}{T_n} \int_{0}^{1}(1+y)^{-\mathfrak{a}}dy
\nn\\
&\lesssim \frac1n + 
\left\{
\begin{array}{cl}
  T_n^{-\min\{\mathfrak{a},1\}}   & (\mathfrak{a}\ne 1) \\
  T_n^{-1}\log T_n   & (\mathfrak{a}= 1)
\end{array}
\right.
\lesssim h_n^{\kappa_4}
\nn
\end{align}
for a sufficiently small $\kappa_4>0$, where we used \eqref{h_step} in the last step.

We turn to $\overline{\delta}''_n(\rho)$. \rev{Assumption}  \eqref{hm:exp.mix} implies that $X$ is polynomially $\beta$-mixing, hence polynomially $\al$-mixing (strong-mixing) as well in the following sense:
\begin{equation}
\al_X(t) := \sup_{s\ge 0}\sup_{A \in \mcf^X_{[0,s]},~ B \in \mcf^X_{[s+t,\infty)}} |\P[A\cap B] - \P[A]\P[B]|
\lesssim (1+t)^{-\mathfrak{a}}.
\nonumber
\end{equation}
Then, \rev{the covariance inequality for bounded random variables (e.g. \cite{VolRoz59})} says that
\begin{equation}
\label{hm:Ibragimov.ineq}
    \sup_{s\ge 0}\left|\mathrm{Cov}[\sig_{s},\sig_{s+\tau}]\right|
    \lesssim 
    (1+\tau)^{-\mathfrak{a}}.
\end{equation}
By the boundedness of $\sig$,
\begin{align}
    \Var[\overline{\delta}''_n(\rho)]
    \lesssim \frac{1}{n^2} \sum_{k=1}^{n}\sum_{l=1}^{n} \mathrm{Cov}[\sig_{k-1},\sig_{l-1}] 
    \lesssim 
    \frac{1}{n} +\frac{1}{n^2} \sum_{k\ne l} \mathrm{Cov}[\sig_{k-1},\sig_{l-1}].
\end{align}
By \eqref{hm:Ibragimov.ineq}, the second term on the rightmost side can be bounded by a constant multiple of
\begin{align}
\frac{1}{n^2} \sum_{k=0}^{n-1} \sum_{l=k+1}^{n} (1+(l-k)h_n)^{-\mathfrak{a}}
&= \frac{1}{n^2} \sum_{k=0}^{n-1} \sum_{m=1}^{n-k} (1+m h_n)^{-\mathfrak{a}}
\nn\\
&\lesssim \frac{1}{n^2} \sum_{k=0}^{n-1} \int_0^{(n-k)h_n} (1+y)^{-\mathfrak{a}}dy
\nn\\
&\lesssim \left\{
\begin{array}{cl}
    n^{-1} & (\mathfrak{a}>1) \\[2mm]
    n^{-1}\log T_n & (\mathfrak{a}=1) \\[2mm]
    n^{-1}T_n^{1-\mathfrak{a}} & (\mathfrak{a}<1) \\
\end{array}
\right.
\nn\\
&=O_\P(h_n^{2\kappa})
\nn
\end{align}
for some $\kappa\in(0,\kappa_4\wedge (1/2))$. 
Thus, we get $\Var[\overline{\delta}''_n(\rho)]\lesssim h_n^{2\kappa}(1+(nh_n^{2\kappa})^{-1})\lesssim h_n^{2\kappa}$. 
Since the summands of $\overline{\delta}''_n(\rho)$ are zero-mean, it follows that $\overline{\delta}''_n(\rho)=O_\P(h_n^\kappa)$.
\end{proof}

\medskip

We now construct our estimator of $\al$ under \eqref{hm:alpha.esti.rho}, in a similar manner to that of \cite[Section 4]{Tod13}. 
To this end, by combining \eqref{hm:alpha.esti.rho} with the conditions on $\rho$ given in Lemma \ref{hm:lem_1}, we assume that
\begin{align}\label{hm:alpha.esti.rho+1}
    \rho\in\left\{
    \begin{array}{ll}
        \ds{\left(0,\, \{1_{\{\beta=0\}}+1_{\{\beta>0\}}\beta\}\wedge \al \wedge 1 \right)}, & T_n\equiv T, \\[2mm]
        \ds{\left(0,\, \{1_{\{\beta=0\}}+1_{\{\beta>0\}}\beta\}\wedge \al \wedge \frac{q}{2} \wedge 1 \right)}, & T_n\to\infty.
    \end{array}
    \right.
\end{align}

In the same way as the above lemmas, we can derive the following for the normalized power-variation based on the third-order differences:
\begin{align}
\mathsf{H}_{2,n}(\rho) 
&:= \frac{1}{n-3}\sum_{k=4}^{n}\big|(4h_n)^{-1/\al}(\Delta_k^2 X - \Delta_{k-2}^2 X)\big|^{\rho}
\nn\\
&= \frac{1}{n-3}\sum_{k=4}^{n} \sig_{k-4}^\rho \big|(4h_n)^{-1/\al}(\Delta_{k}^2 Z - \Delta_{k-2}^2 Z)\big|^{\rho} + O_\P(h_n^\kappa)
\nn\\
&= \left(\mathfrak{m}_\al(\rho) + o(1) \right) \frac{1}{n-3}\sum_{k=4}^{n} \sig_{j-4}^\rho + O_\P(h_n^\kappa)
\nn\\
&= \mathsf{H}_0(\rho) + O_\P(h_n^\kappa),
\label{hm:V2_se}
\end{align}
where the exponent $\kappa>0$ can be the same as before, by making it smaller if necessary.
On the one hand, since $\mathsf{H}_0(\rho)>0$, from Lemma \ref{hm:lem_4} and \eqref{hm:V2_se},
\begin{equation}\label{hm:index.esti.eq}
\frac{\mathsf{H}_{2,n}(\rho)}{\mathsf{H}_{1,n}(\rho)} = 1 + O_\P(h_n^{\kappa}).
\end{equation}
On the other hand, the left-hand side of the identity
\begin{equation}
\frac{\mathsf{H}_{2,n}(\rho)}{\mathsf{H}_{1,n}(\rho)} 
= 2^{-\rho/\al}
\frac{(n-1)\sum_{k=4}^{n}\big|\Delta_k^2 X - \Delta_{k-2}^2 X\big|^{\rho}}{(n-3)\sum_{k=2}^{n}\big|\Delta_k^2 X\big|^{\rho}}
\nonumber
\end{equation}
is a statistic. By ignoring the $O_\P(h_n^\kappa)$ term in \eqref{hm:index.esti.eq}, namely by setting the left-hand side of the previous display equal to $1$, we obtain the following explicit estimator of $\al$:
\begin{equation}
\hat{\al}_n(\rho) \coloneq 
\rho \log 2
\left/
\log\left(\frac{(n-1)\sum_{k=4}^{n}\big|\Delta_k^2 X - \Delta_{k-2}^2 X\big|^{\rho}}{(n-3)\sum_{k=2}^{n}\big|\Delta_k^2 X\big|^{\rho}}
\right.
\right).
\label{hm:def_alpha.est}
\end{equation}
This $\hat{\al}_n(\rho)$ differs slightly from the estimator in \cite[Section 4]{Tod13}, where an \textit{aggregated} version (specifically, the statistics $\widetilde{V}_n^2(p,X)$ therein) was used, in contrast to our \textit{third-order-difference-based} statistic $V_{2,n}(\rho)$.

Thus, by the delta method, we end up with the following lemma:
\footnote{
Yet another way would be to use the bipower version instead of $V_{2,n}(\rho)$:
\begin{align}
B_{n}(\rho/2,\rho/2) &:= \frac{1}{n-2}\sum_{j=3}^{n}\big|(2h_n)^{-1/\al}\Delta_k^2 X\big|^{\rho/2}\big|(2h_n)^{-1/\al}\Delta_{k-2}^2 X\big|^{\rho/2}.
\nonumber
\end{align}
}

\begin{lem}
\label{hm:lem_index.consistency}
For any $\rho$ satisfying \eqref{hm:alpha.esti.rho+1}, there exist a constant $\kappa=\kappa(\rho)>0$ for which $\aes(\rho)$ defined by \eqref{hm:def_alpha.est} satisfies that
\begin{equation}
h_n^{-\kappa}\left( \hat{\al}_n(\rho)- \al \right) = O_\P(1),
\end{equation}
regardless of whether $T_n\equiv T$ or $T_n\to\infty$.
\end{lem}

\subsection{Riemann integrals}

Having obtained an estimator $\hat{\al}_n$ that satisfies \eqref{hm:def_aes}, we proceed to estimate $\Sigma_0= \Sigma_0(\theta_0)$ and $\Gamma_0= \Gamma_0(\theta_0)$ in \eqref{hm:AV}.
Let us recall their specific forms:
\begin{itemize}
\item For $T_n\equiv T$ (with $V(x)\equiv 1$),
\begin{align}
\Gamma_0 &=\frac1T \int_0^T\Big(\nabla_\theta a(\theta_0; X_{t})\Big)^{\otimes 2}\, dt,
\nn\\
\Sigma_0 &=\frac1T \int_0^T\frac{1}{\sigma(X_{t})}\Big(\nabla_\theta a(\theta_0; X_{t})\Big)^{\otimes 2}\, dt.
\nonumber
\end{align}
\item For $T_n\to\infty$,
\begin{align}
\Gamma_0 &=\int_{\mathbb{R}}V(x)^2 \Big(\nabla_\theta a(\theta_0; x)\Big)^{\otimes 2}\, \pi(\theta_0, dx),
\nn\\
\Sigma_0 &=\int_{\mathbb{R}}\frac{V(x)}{\sigma(x)}\Big(\nabla_\theta a(\theta_0; x)\Big)^{\otimes 2}\, \pi(\theta_0, dx).
\nonumber
\end{align}
\end{itemize}

\subsubsection{Basic convergence in probability}
\label{hm:sec_integ.basic.lim}

Again, we use the shorthand $f(x,\theta)=(\nabla_\theta a(\theta;x))^{\otimes 2}$. 
We denote by $\psi:\,\mbbr\times\overline{\Theta} \to \mbbr$ either:
\begin{align}
    & f(x,\theta) \quad \text{or} \quad \frac{1}{\sig(x)}f(x,\theta) \qquad \text{for $T_n\equiv T$;} \nn\\
    & V(x)f(x,\theta)\quad \text{or} \quad \frac{V(x)^2}{\sig(x)}f(x,\theta)
    \qquad \text{for $T_n\to\infty$.} \nn\\
\end{align}
By Assumption \ref{hm:asm-2}, $\sup_\theta |\partial_\theta^l \psi(\cdot,\tz)| \lesssim 1+|x|^C$ for $l\in\{0,1\}$ when $T_n\equiv T$, and $\psi(\cdot,\tz) \in L^1(\pi(\tz,dx))$ when $T_n\to\infty$.
Let
\begin{align}
\Psi_n(\theta) &:= \frac1n \sumk \psi(X_{t_{k-1,n}},\theta), 
\nn\\
\Psi_0(\theta) 
&:= \left\{
\begin{array}{ll}
\displaystyle{\frac1T \int_0^T \psi(X_t,\theta)dt} & (T_n\equiv T) \\[3mm]
\displaystyle{\int \psi(x,\theta)\pi(\tz,dx)} & (T_n\to\infty) 
\end{array}
\right. .
\end{align}

\begin{lem}\label{hm:lem_integ1}
We have
\begin{equation}
    \sup_\theta \left| \Psi_n(\theta) - \Psi_0(\theta) \right| \cip 0.
\end{equation}   
\end{lem}

\begin{proof}
Suppose for a moment that
\begin{equation}\label{hm:lem_integ1-1}
    \sup_\theta \left| \Psi_n(\theta) - \frac{1}{T_n} \int_0^{T_n} \psi(X_{t},\theta)dt \right| \cip 0.
\end{equation}
For $T_n\equiv T$, \eqref{hm:lem_integ1-1} is exactly what we want to show. 
For $T_n\to\infty$, the assumed ergodicity ensures that 
$T_n^{-1} \int_0^{T_n} \psi(X_{t},\theta)dt \cip \Psi_0(\theta)=\int \psi(x,\theta)\pi(\tz,dx)$ for each $\theta$, and its uniformity in $\theta$ follows from the boundedness $\sup_{(x,\theta)} |\partial_\theta \psi(x,\theta)| <\infty$, which holds by Assumption \ref{hm:asm-2}.
Thus, it remains to show \eqref{hm:lem_integ1-1}.

Note that
\begin{align}\label{hm:lem_integ1-2}
    \sup_\theta \left| \Psi_n(\theta) - \frac{1}{T_n} \int_0^{T_n} \psi(X_{t},\theta)dt \right|
    &\le \frac1n \sumk \frac{1}{h_n} \int_k \sup_\theta | \psi(X_t,\theta) - \psi(X_{t_{k-1,n}},\theta) | dt.
\end{align}
First, we consider $T_n\equiv T$; as before, we may and do suppose that $\sup_{t\le T}\E[|X_t|^K]<\infty$ for any $K>0$. 
By the global Lipschitz property of $\sig(x)$, the right-hand side of \eqref{hm:lem_integ1-2} can be bounded by a constant multiple of the following quantity:
\begin{align}
    & \frac1n \sumk \frac{1}{h_n} \int_k \sup_\theta| f(X_t,\theta) - f(X_{t_{k-1,n}},\theta) | dt
    \nn\\
    &{}\qquad + \frac1n \sumk \sup_\theta |f(X_{t_{k-1,n}},\theta)|\,\frac{1}{h_n} \int_k | \sig(X_t) - \sig(X_{t_{k-1,n}}) | dt
    \nn\\
    &\lesssim \frac1n \sumk (1+|X_{t_{k-1,n}}|^C) \,
    \frac{1}{h_n} \int_k | X_t -X_{t_{k-1,n}}| \max\left\{ 1,\,| X_t -X_{t_{k-1,n}}|\right\}^C dt
    \nn\\
    &{}\qquad + \frac1n \sumk (1+|X_{t_{k-1,n}}|^C) \frac{1}{h_n} \int_k | X_t -X_{t_{k-1,n}}| dt.
\end{align}
We take the expectation in \eqref{hm:lem_integ1-2} with the conditioning with respect to $\mcf_{t_{k-1,n}}$: through a similar manner to the proof of \eqref{hm:lem_1-3} via the expression \eqref{hm:index-p+1} with $X_{t_{k-2,n}}$ replaced by $X_{t_{k-1,n}}$, we get
\begin{align}
    & \E\left[ \sup_\theta \left| \Psi_n(\theta) - \frac{1}{T} \int_0^{T} \psi(X_{t},\theta)dt \right|\right]
    \nn\\
    & \lesssim \frac1n \sumk \frac{1}{h_n} \int_k 
    \E\left[ (1+|X_{t_{k-1,n}}|^C) \E\left[| X_t -X_{t_{k-1,n}}|^C \mid \mcf_{t_{k-1,n}} \right]\right] dt = o(1).
\end{align}
This establishes \eqref{hm:lem_integ1-1}, completing the proof for $T_n\equiv T$.

We now consider $T_n \to \infty$. 
Under the present assumptions, the function $\psi(x,\theta)$ is essentially bounded, say $\sup_{x,\theta}|\psi(x,\theta)|\le c_\psi$. 
Denote by $N^Z(ds,du)$ the Poisson point measure associated with $Z$, and let $A_{t}(h):=\{N^Z((t,t+h], \{|u|\ge 1\}) = 0\}$ for $t\ge 0$ and $h>0$. Then,
\begin{align}
    \sup_{t\ge 0} 
    \E\left[ \sup_\theta \left| \psi(X_{t+h},\theta) - \psi(X_t,\theta) \right| ;\, A_{t}(h)^c\right]
    &\le 2 c_\psi \P\bigl[A_0(h)^c\bigr] = O(h).
\end{align}
Fix any $\eps'>0$. Then, there exists a constant $\delta'>0$ for which $\sup_{|x-y|\le \delta'} \sup_\theta | \psi(x,\theta) - \psi(y,\theta)|<\eps'/2$. 
We have
\begin{align}\label{hm:202510251444}
    \sup_{t\ge 0} \E\left[ \sup_\theta \left| \psi(X_{t+h},\theta) - \psi(X_t,\theta) \right| ;\, A_{t}(h)\right]
    &\le \frac{\eps'}{2} + 2 c_\psi \sup_{t\ge 0} \P\left[A_{t}(h) \cap \{|X_{t+h}-X_t| \ge \delta'\}\right].
\end{align}
Since the process $v\mapsto Z_v-Z_t$ for $v\in(t,t+h]$ has only bounded jumps on $A_t(h)$, under Assumption \ref{hm:asm-1}.2 and the linear growth property of the coefficients of $X$, it is standard to show that the second term in the upper bound of \eqref{hm:202510251444} goes to $0$ as $h\to 0$. 
Piecing together the above observations and \eqref{hm:lem_integ1-2} concludes \eqref{hm:lem_integ1-1}. The proof is complete.
\end{proof}

By Lemma \ref{hm:lem_integ1} and the continuity of $\theta \mapsto \psi(x,\theta)$ for each $x$, we have
\begin{equation}
 \hat{\Psi}_n \coloneqq \Psi_n(\tes) \xrightarrow{\P} \Psi_0(\tz)
\end{equation}
for the LADE $\tes$. 
Let
\begin{equation}
    \hat{f}_{k-1} \coloneqq f(X_{t_{k-1,n}},\tes).
\end{equation}
For both $T_n\to\infty$ and $T_n\equiv T$ (with $V(x)\equiv 1$), it holds that
\begin{align}
\hat{\Gamma}_n &:= \frac1n \sumk V_{k-1}^2 
\hat{f}_{k-1}
\cip \Gamma_0,
\nn\\
\Sigma^{\ast}_n &\coloneqq \frac1n \sumk 
\frac{V_{k-1}}{\sigma_{k-1}}
\hat{f}_{k-1}
\cip \Sigma_0.
\label{hm:Sig.ast.conv}
\end{align}
Note that $\Sigma^{\ast}_n$ is not a statistic, as it depends on the unknown function $\sig(\cdot)$. In the next subsection, we address this by plugging in a suitable estimator.

\medskip

Before proceeding, we remark on the case of a constant scale, namely $\sig(x)\equiv \sig_0>0$), focusing on the weight function $V\equiv 1$ even in the case of $T_n\to\infty$.
Then, the asymptotic covariance matrix takes the following simpler form:
\begin{align}
    \mathcal{A}_0 
    = \sig_0^2 \left(2\phi_{\al}(0)\right)^{-2} \times \left\{
    \begin{array}{ll}
    \displaystyle{\left(\frac1T \int_0^T f(\theta_0; X_{t}) dt\right)^{-1}} & (T_n\equiv T) \\[3mm]
    \displaystyle{\left(\int f(\theta_0, x) \pi(\tz,dx)\right)^{-1}} & (T_n\to\infty) 
    \end{array}
    \right. .
\end{align}
By Lemma \ref{hm:lem_integ1} and recalling \eqref{hm:h^al-rate-gam}, we have
\begin{equation}
    \hat{\sig}_n^2 \left(2\phi_{\aes}(0)\right)^{-2}
    \left( \frac1n \sumk \hat{f}_{k-1} \right)^{-1} \cip \mathcal{A}_0
\end{equation}
for any $\hat{\sig}_n \cip \sig_0$.
Hence, it suffices to have a consistent estimator of $\sig_0$. 
By Lemma \ref{hm:lem_index.consistency} and \eqref{hm:h^al-rate}, we have for some $\kappa'>0$,
\begin{align}
    \hat{\mathsf{H}}_{1,n}(\rho) 
    &\coloneqq \frac{1}{n-1}\sum_{k=2}^{n}\big|(2h_n)^{-1/\aes}\Delta_k^2 X\big|^{\rho}
    \nn\\
    &= \mathsf{H}_{1,n}(\rho) (1+o_\P(h_n^{\kappa-\eps}))^\rho
    = \left( \mathsf{H}_0(\rho) + O_\P(h_n^{\kappa})\right)(1+o_\P(h_n^{\kappa-\eps}))^\rho
    \nn\\
    &= \mathsf{H}_0(\rho) + O_\P(h_n^{\kappa'}). 
\end{align}
By the continuity of $\al \mapsto \mathfrak{m}_\al(\rho)$ (recall the expression \eqref{hm:def_m.al}), 
we get
\begin{align}\label{hm:sigma.estimator}
    \hat{\sig}_n=\hat{\sig}_n(\rho) \coloneq \left(\frac{\hat{\mathsf{H}}_{1,n}(\rho)}{\mathfrak{m}_{\hat{\al}_n}(\rho)}\right)^{1/\rho} \cip \sig_0.
\end{align}


\subsubsection{Dealing with unknown scale coefficient}

Our goal here is to construct statistics $\hat{\Sigma}_n(\rho)$ such that
\begin{equation}\label{hm:scale.infe_goal}
    \hat{\Sigma}_n(\rho) \cip \Sigma_0,
\end{equation}
with focusing on an unknown non-constant $\sigma(x)$. 
We will consider approximating each instantaneous scales $\sigma_{k-1}=\sig(X_{t_{k-1}})$ over rolling windows with shrinking width $h_n l_n$, where $(l_n)_{n\ge 1}$ is a given divergent positive-integer sequence such that
\begin{equation}\label{hm:l-rate}
    l_n \gtrsim h_n^{-\mathsf{c}},\qquad l_n h_n \to 0
\end{equation}
for some $\mathsf{c}>0$; in particular, $l_n=o(n)$.
This amounts to normalizing the increments $\Del^2_k X$ in a non-parametric way.

For $k=2,\dots,n - l_n$, we introduce the statistics
\begin{align}
    \widetilde{\sigma}^{\flat}_{k-1,n}(\rho)=\widetilde{\sigma}^{\flat}_{k-1}(\rho)
    &\coloneq \left( \mathfrak{m}_{\aes}(\rho)^{-1}
    \hat{S}_{k}(\rho) \right)^{1/\rho},
\end{align}
where
\begin{align}
    \hat{S}_{k}(\rho) = \hat{S}_{k,n}(\rho)
    &\coloneq \frac{1}{l_n} \sum_{j=k+1}^{k+l_n} \big|(2h_n)^{-1/\aes}\Delta_j^2 X\big|^{\rho},
\end{align}
For $\mathsf{c}_\sigma := \inf_x \sigma(x)>0$, we define
\begin{equation}
    \hat{\sigma}_{k-1,n}(\rho) = \hat{\sigma}_{k-1}(\rho)
    \coloneq \max\left\{\widetilde{\sigma}^{\flat}_{k-1}(\rho), \, \frac{\mathsf{c}_\sigma}{2}\right\}.
\end{equation}
This $\hat{\sigma}_{k-1}(\rho)$ will serve as a spot-scale estimator of $\sigma_{k-1}=\sigma(X_{t_{k-1,n}})$; we have $\widetilde{\sigma}^{\flat}_{k-1}(\rho)>0$ a.s., and we do not need to specify the value $\mathsf{c}_\sigma$ in practice.

\begin{lem}\label{hm:lem-s1}
    Under the above-mentioned setup, we have \eqref{hm:scale.infe_goal} for ($V_{k-1} \coloneq V(X_{t_k-1,n})$)
    \begin{equation}\label{hm:scale.est-1}
    \hat{\Sigma}_n(\rho) \coloneq 
    \left\{
    \begin{array}{ll}
    \displaystyle{\frac1n \sumk \frac{1}{\hat{\sigma}_{k-1}(\rho)}\hat{f}_{k-1}} & (T_n\equiv T) \\[5mm]
    \displaystyle{\frac1n \sumk \frac{V_{k-1}}{\hat{\sigma}_{k-1}(\rho)}\hat{f}_{k-1}} & (T_n\to\infty) 
    \end{array}
    \right. .
    \end{equation}
\end{lem}

\begin{proof}
(i) We first consider $T_n \equiv T$.
Let
\begin{equation}
    B_{k-1} \coloneq \left\{ \widetilde{\sigma}^{\flat}_{k-1,n}(\rho) \ge \frac{\mathsf{c}_\sigma}{2} \right\}.
\end{equation}
Since $\hat{\sigma}_{k-1,n}(\rho)\ge \mathsf{c}_\sigma/2$, we have
\begin{align}
    \left| \hat{\Sigma}_n(\rho) - \Sigma^\ast_n \right|
    & \le \frac1n \sumk \left|
    \frac{1}{\sigma_{k-1}} - \frac{1}{\hat{\sigma}_{k-1}(\rho)}
    \right| |\hat{f}_{k-1}|
    \nn\\
    &\lesssim 
    \frac1n \sumk |\hat{f}_{k-1}|
    \big| \widetilde{\sigma}^{\flat}_{k-1}(\rho) - \sigma_{k-1}\big| I_{B_{k-1}}
    + \frac1n \sumk |\hat{f}_{k-1}| I_{B_{k-1}^c}
    \nn\\
    &\eqqcolon D_{1,n}+D_{2,n}.
    \label{hm:lem-s1-p6}
\end{align}
From this, combined with \eqref{hm:Sig.ast.conv} with $V(x)\equiv 1$, \eqref{hm:scale.infe_goal} follows once we show 
both $D_{1,n} \cip 0$ and $D_{2,n}\cip 0$. 
Again, without loss of generality, we may and do suppose that $J$ has only bounded jumps.

First, we make some preliminary observations. 
By \eqref{hm:h^al-rate},
\begin{align}\label{hm:lem-s1-p1}
    \hat{S}_{k}(\rho) = \left(1+o_{P}(h_n^{\kappa-\varepsilon})\right)^{\rho} S_k(\rho),
\end{align}
where
\begin{equation}
    S_k(\rho) \coloneq \frac{1}{l_n} \sum_{j=k+1}^{k+l_n} \big|(2h_n)^{-1/\al}\Delta_j^2 X\big|^{\rho}.
\end{equation}
Let us generically denote by $\overline{R}_{k}(\rho)$ any non-negative random variable such that 
$n^{-1} \sum_k \E[\overline{R}_{k}(\rho)^2] = O(1)$. 
By the expression \eqref{hm:inc.sum} of $(2h_n)^{-1/\al}\Delta_k^2 X$ and mimicking the argument in the proof of Lemma \ref{hm:lem_1}, we can deduce the following estimates:
\begin{align}
    \left| S_{k}(\rho) - \frac{1}{l_n} \sum_{j=k+1}^{k+l_n} 
    \sig_{j-2}^\rho \big|(2h_n)^{-1/\al}\Delta_j^2 Z\big|^{\rho} \right|
    \lesssim \frac{1}{l_n} \sum_{j=k+1}^{k+l_n} |\overline{r}_j|^\rho
    \lesssim h_n^{\kappa_1} \, \overline{R}_{k}(\rho).
    \label{hm:lem-s1-p2}
\end{align}
Further, by the Burkholder inequality and the moment estimates in Lemma \ref{hm:lem_moment-1},
\begin{align}
    & \Biggl|
    \frac{1}{l_n} \sum_{j=k+1}^{k+l_n} 
    \sig_{j-2}^\rho \left( \big|(2h_n)^{-1/\al}\Delta_j^2 Z\big|^{\rho} - \E\left[\big|(2h_n)^{-1/\al}\Delta_2^2 Z\big|^{\rho} \right]\right)
    \nn\\
    &{}\qquad 
    - \frac{1}{l_n} \sum_{j=k+1}^{k+l_n} 
    \sig_{j-2}^\rho 
    \left( \E\left[\big|(2h_n)^{-1/\al}\Delta_2^2 Z\big|^{\rho} \right] - \mathfrak{m}_\al(\rho)\right)
    \Biggr| \le \frac{1}{\sqrt{l_n}} \overline{R}_{k}(\rho) + C h^{c}.
    \label{hm:lem-s1-p3}
\end{align}
Here again, $c$ denotes a generic positive constant which may vary at each appearance.
Combining \eqref{hm:lem-s1-p1}, \eqref{hm:lem-s1-p2}, and \eqref{hm:lem-s1-p3} gives
\begin{align}
    \hat{S}_{k}(\rho) &= 
    \mathfrak{m}_\al(\rho)\,\frac{1}{l_n} \sum_{j=k+1}^{k+l_n} \sig_{j-2}^\rho     
    + h_n^{c}\,\overline{R}_{k}(\rho).
\end{align}
We can further proceed as follows, with $c>0$ being small enough:
\begin{align}
    \hat{S}_{k}(\rho) 
    &= \mathfrak{m}_{\aes}(\rho)\,\sigma_{k-1}^{\rho}
    + \mathfrak{m}_{\al}(\rho)\,\frac{1}{l_n} \sum_{j=k+1}^{k+l_n} (\sig_{j-2}^\rho - \sigma_{k-1}^{\rho}) + h_n^{c}\,\overline{R}_{k}(\rho)
    \nn\\
    &= \mathfrak{m}_{\aes}(\rho)\,\sigma_{k-1}^{\rho} + h_n^{c}\,\overline{R}_{k}(\rho),
\end{align}
where the second step is due to \eqref{hm:l-rate}.
Since $\mathfrak{m}_{\al}(\rho)>0$ for any $\al$ in a compact subset of the open interval $(1/2,2)$, we arrive at
\begin{equation}
    {\widetilde{\sigma}^{\flat}_{k-1}(\rho)}^{\rho}
    = \mathfrak{m}_{\aes}(\rho)^{-1} \hat{S}_{k}(\rho) 
    = \sigma_{k-1}^{\rho} + h_n^{c}\,\overline{R}_{k}(\rho).
\end{equation}
Note that for $\rho\in(0,1]$ and $x,y\ge 0$,
\begin{align}
    |x-y| = \left| (x^\rho)^{1/\rho} - (y^\rho)^{1/\rho}\right|
    &= \left| \int_0^1 \frac{1}{\rho}\left( y^\rho + s(x^\rho - y^\rho)\right)^{1/\rho-1}ds(x^\rho - y^\rho) \right|
    \nn\\
    &\lesssim (1+|x|+|y|)^{1/\rho-1}|x^\rho - y^\rho|.
\end{align}
The previous two displays yield
\begin{equation}\label{hm:lem-s1-p5}
    \left|\widetilde{\sigma}^{\flat}_{k-1}(\rho) - \sigma_{k-1} \right|
    \le h_n^{c}\,\overline{R}_{k}(\rho)
\end{equation}
for $c>0$ small enough.

Now, for $D_{1,n}$, we note $|\hat{f}_{k-1}| \lesssim 1+|X_{t_{k-1,n}}|$, so that $n^{-1}\sumk |\hat{f}_{k-1}|^K = O_\P(1)$ for any $K>0$. Then, the estimate \eqref{hm:lem-s1-p5} and the H\"{o}lder inequality give $D_{1,n} = O_\P(h_n^c) \cip 0$. 

Next, $D_{2,n} \cip 0$ follows from 
$n^{-1} \sumk \P[B_{k-1}^c] \to 0$. 
Observe that, by \eqref{hm:lem-s1-p5}, for a sufficiently small $c>0$ and any $M>0$,
\begin{align}
    \frac1n \sumk \P[B_{k-1}^c] 
    &\le \frac1n \sumk \P\left[ 
    \sigma_{k-1} \le \frac{\mathsf{c}_\sigma}{2} 
    + \left| \widetilde{\sigma}^{\flat}_{k-1}(\rho) - \sigma_{k-1}\right|    
    \right]
    \nn\\
    &\le \frac1n \sumk \bigg( \underbrace{\P\left[ 
    \sigma_{k-1} \le \frac{\mathsf{c}_\sigma}{2} + h_n^{c/2} M
    \right]}_{=0~\text{for every $n$ large enough.}} 
    + P\left[ h_n^{c/2} \overline{R}_{k}(\rho) \ge M \right] \bigg)
    \nn\\
    &\lesssim h_n^c \to 0.
\end{align}
This concludes the proof of \eqref{hm:scale.infe_goal} for $T_n\equiv T$.

\medskip

(ii) The case of $T_n \to\infty$ is now almost done; in this case, we are assuming the boundedness of $(x,\theta) \mapsto V(x)f(x,\theta)$. To derive \eqref{hm:scale.infe_goal}, we can follow the same argument as in (i), starting from the estimate \eqref{hm:lem-s1-p6}, 
by replacing $\hat{f}_{k-1}$ with the bounded term $V_{k-1} \hat{f}_{k-1}$.
\end{proof}



\appendix
\section{Properties of the transition density}
\label{sA1}

In this section, we discuss the properties of the transition density of the solution $X$ for the SDE \eqref{sde} considered, for each $\theta$, as a time-homogeneous Markov process. This discussion will be based mainly on the paper \cite{AK}, where locally stable \textit{L\'evy-type processes} were studied. It is easy to re-formulate our SDE setting to the one adopted in \cite{AK}; in particular, the L\'evy kernel $\mu(x; du)$ in the formula \cite[Eq.(2.1)]{AK} for the generator of the process is just the image of our L\'evy measure $\nu(du)$ under the mapping 
$u\mapsto \sigma(x)u$, and the principal and nuisance parts in the decomposition \cite[Eq.(2.2)]{AK} of this kernel are the images under the same mapping of the similar parts of the decomposition \eqref{LM} of $\nu$ in our setting. The principal part is the stable kernel with  
$$
\lambda(x)=c_\al\sigma(x)^{\alpha}, \quad \rho(x)=0,
$$
 and it is easy to verify that Assumption \ref{Ass1} guarantees all the assumptions of \cite[Theorem 3.3]{AK}. Therefore, the latter theorem guarantees that the solution $X$ for the SDE \eqref{sde} with a fixed $\theta$ is a (well-defined) time-homogeneous Markov process, which admits a transition density $p_t(\theta; x,y)$ with respect to the Lebesgue measure: $\P(X_{t}\in dy|X_{0}=x)=p_{t}(x,y)dy$. Moreover, this density has a semi-explicit representation, which we will explain below. 
 
 Note that the \textit{compensated drift coefficient} introduced in \cite[Section~3.1]{AK}, in our setting, coincides with the drift coefficient $a(\theta;x)$ because the driving L\'evy process is pure jump; recall \eqref{psi_pj} and also \eqref{pur_jump}.
Define a \emph{dynamically mollified drift coefficient}
$$
 A_t(\theta; x)=\int_{\Re}a(\theta;x-z){1\over 2\sqrt{\pi} t^{1/\alpha}}e^{-z^2t^{-2/\alpha}}d z.
$$
This coefficient satisfies
\be\label{e1}
|\prt_x A_t(\theta;x)|\leq Ct^{-1+\delta_{drift}},
\ee
\rev{ see Assumption \ref{AssStab} for the definition of $\delta_{drift}$.} Hence the usual Picard iteration scheme provides that the Cauchy problem
$$
d \mathfrak{f}_t(\theta; x)= A_t(\theta; \mathfrak{f}_t(\theta; x))\, dt, \quad \mathfrak{f}_0(\theta; x)=x
$$
has a unique solution. Moreover, it is easy to check that
$$
|A_t(\theta;x)- a(\theta;x)|\leq Ct^{{\eta}/{\al}},
$$
which yields that, for any solution $f_t(\theta;x)$ to the ODE \eqref{ODE},
$$
|\mathfrak{f}_t(\theta;x)- f_t(\theta;x)|\leq Ct^{1+\frac{\eta}{\al}}=Ct^{1/{\al}+\delta_{drift}}.
$$
Recall that we assume $\delta_{regr}\leq \delta_{drift},$ which means that assumption \eqref{F0} remains true with a true solution $f_t(\theta;x)$ replaced the an approximate  solution $\mathfrak{f}_t(\theta;x)$; see Remark \ref{remAss2}. 
Also, recall that we are given the constant (see Assumption \ref{AssStab})
\begin{equation}
 \delta<\min\{\delta_{regr}, \delta_\sigma, \delta_\nu\} \le \min\{\delta_{drift}, \delta_\sigma, \delta_\nu\}.   
\end{equation}
Furthermore, we denote
$$
\sigma_t(x)=\left({1\over t}\int_{0}^{t}\sigma(\mathfrak{f}_s(\theta_0; x))^\alpha\, ds\right)^{1/\alpha}.
$$
The following statement directly follows from part I of \cite[Theorem 3.4]{AK}.

\begin{thm}[Properties of $p_t(x,y)$]\label{tA1} Let Assumptions \ref{Ass1} and \ref{AssStab} hold true. Then
\be\label{decomp}
p_t(x,y)=\wt p_t(x,y)+r_t(x,y),
\ee
with the leading term given explicitly by 
\be\label{p_zero}
\wt p_t(x,y)
={1\over \sigma_t(x)t^{1/\alpha}}
\phi_{\al}
\left({y-\mathfrak{f}_t(\theta_0;x)\over \sigma_t(x)t^{1/\alpha}}\right),
\ee
where $\phi_\al$ is defined by \eqref{hm:def_phi.al}, and with the residual term satisfying
\be\label{Residue_1}
\sup_{x\in \Re}\int_{\Re}|r_t(x,y)|\, dy\leq Ct^\delta. 
\ee
\end{thm}

The estimate \eqref{Residue_1} tells us that the term we called the residual is indeed negligible in the integral sense. We have used this bound, combined with the stability assumption \eqref{h_step}, in order to show that the `linear part' $U_n(\theta)$ of the contrast function is `essentially martingale' in the sense that its predictable part is negligible; see Lemma \ref{l41}.

In order to analyze the `quadratic part' $Y(\theta),$ we require the residual term $r_t(x,y)$ to be small in another, point-wise sense. Namely, we have used in the proof of Proposition \ref{p41} that  
\be\label{Residue_2}
\sup_{x,y}|r_t(x,y)|\leq Ct^{-1/\alpha+\delta'}
\ee
with some $\delta'>0.$ There is a substantial difference between the integral and point-wise bounds for the residual term  $r_t(x,y)$ in \eqref{p_zero}; see \cite[Example~3.7]{AK} which shows that the basic Assumption \ref{Ass1}, while providing the integral bound \eqref{Residue_1}, is not sufficient for the point-wise bound \eqref{Residue_2}. This was the actual reason for us to introduce the additional Assumption \ref{AssReg}. 

\begin{thm}\label{tA2} 
Let Assumptions \ref{Ass1} and \ref{AssReg} hold. Then the residue term in the decomposition \eqref{decomp} satisfies \eqref{Residue_2} with 
$$
\delta'=1-\frac{\beta'}{\al}>0.
$$
\end{thm}

\begin{proof} By part II of \cite[Theorem 3.4]{AK}, the following additional condition is sufficient for \eqref{Residue_2}  to hold true for some $\delta'\leq \delta$:
$$
\sup_{w\in \mathbb{R}}\left|t^{-1/\al}\int_{\mathbb{R}}\nu\left(|\sigma(x)u|>t^{1/\al}, \, |x+\sigma(x)u-w|<t^{1/\al}\right)\, dx\right|\leq Ct^{-1+\delta'}.
$$
Using additivity of $\nu(du)$, it is easy to deduce the above condition from a similar one with arbitrary (but fixed) $\eps>0$:
$$
\sup_{w\in \mathbb{R}}\left|t^{-1/\al}\int_{\mathbb{R}}\nu(|\sigma(x)u|>t^{1/\al}, \, |x+\sigma(x)u-w|<\eps t^{1/\al})\, dx\right|\leq Ct^{-1+\delta'}.
$$
Taking  
$$
m=\inf_x\frac1{\sigma(x)}, \quad M=\sup_{x}\frac{1}{\sigma(x)},
$$
we obtain another sufficient condition for the previous one: 
\begin{equation}
\label{hm:A_add1}
\sup_{w\in \mathbb{R}}t^{-1/\al}\int_{\mathbb{R}}|\nu|\left(|u|>mt^{1/\al}, \, \left|u-\frac{w-x}{\sigma(x)}\right|<\eps M t^{1/\al}\right)\, dx\leq Ct^{-1+\delta'}.
\end{equation}

Now we are ready to verify this condition using Assumption \ref{AssReg}. Fix $x,w$ and denote 
$
z=\frac{w-x}{\sigma(x)}.
$
We have 
$$
|u|>mt^{1/\al}, |u-z|\leq \eps Mt^{1/\al}\Longrightarrow |u-z|\leq \eps Mt^{1/\al}, \quad |z|> m t^{1/\al} -\eps M t^{1/\al}.
$$
Take $\eps=\frac{m}{3M},$ then 
$$
 \quad r: = \eps M t^{1/\al}=2(m t^{1/\al} -\eps M t^{1/\al})<|z|,
$$
and applying Assumption \ref{AssReg} we get 
$$
|\nu|\left(|u|>mt^{1/\al}, \, \left|u-\frac{w-x}{\sigma(x)}\right|<\eps M t^{1/\al}\right)
\leq |\nu|(|u-z|<r)
\leq Cr^{\kappa}|z|^{-\beta'-\kappa}
$$
for $|z|> m t^{1/\al} -\eps M t^{1/\al},$ and 
$$
|\nu|\left(|u|>mt^{1/\al}, \, \left|u-\frac{w-x}{\sigma(x)}\right|<\eps M t^{1/\al}\right)=0
$$
otherwise. Note that
$$
m|x-w|\leq |z|\leq M|x-w|, 
$$
hence we can verify \eqref{hm:A_add1} as follows:
$$\ba
t^{-1/\al}&\int_{\mathbb{R}}|\nu|\left(|u|>mt^{1/\al}, \left|u-\frac{w-x}{\sigma(x)}\right|<\eps M t^{1/\al}\right)\, dx
\\&\leq t^{-1/\al}(\eps M t^{1/\al})^\kappa\int_{|x-w|\geq M^{-1}(m  -\eps M) t^{1/\al}} |x-w|^{-\kappa-\beta'}\, dw
\\&=Ct^{-1/\alpha+\kappa/\al}(t^{1/\al})^{1-\kappa-\beta'}
=Ct^{-\beta'/\al}=Ct^{-1+\delta'},
\ea
$$
where we used that $\kappa+\beta'>1$.  
\end{proof}

\bigskip


\section{Limit theorems}\label{sB} 

In this section we prove the weak convergence \eqref{CLT} for the triple $(\Xi_n, \Gamma_n, \Sigma_n).$ This will require substantially different tools in the finite and infinite observation horizon cases, which we thus consider separately.

\subsection{Finite observation horizon}\label{sB2}

Recall that, for the finite observation horizon case $T_n\to T$, we take $V(x)\equiv 1$; that is, we do not need additional weights in the construction. Since the functions 
$$
f(x):= \left(\nabla_\theta a(\theta_0;x)\right)^{\otimes 2}, \quad g(x)=\frac{1}{\sigma (x)}\left(\nabla_\theta a(\theta_0, x)\right)^{\otimes 2}
$$
are continuous and trajectories of $X$ are c\`adl\`ag, we have 
$$
\Gamma_{n}=\frac{1}{n}\sumk f(X_{t_{k-1,n}})\mathop{\to}\limits^{\mathbb{P}} \frac{1}{T}\int_0^Tf(X_t)\, dt=\Gamma_0,
$$
$$
\Sigma_{n}=\frac{1}{n}\sumk g(X_{t_{k-1,n}})\mathop{\to}\limits^{\mathbb{P}} \frac{1}{T}\int_0^Tg(X_t)\, dt=\Sigma_0
$$
simply by the convergence of Riemann sums to the Riemann integral. Therefore, in order to prove \eqref{CLT}, it is enough to show that 
\be\label{CLT_s}
\Xi_n\mathop{\to}\limits^{s-\mathcal{L}} \Xi_0,
\ee
where $\mathop{\to}\limits^{s-\mathcal{L}}$ denotes \textit{stable convergence in law}. We refer to \cite[Section 2]{Jac97} for the definition and basic properties of this mode of convergence. To prove \eqref{CLT_s} it is enough apply  \cite[Theorem 3.2]{Jac97} to the sequence of processes
$$
\Xi_t^n=\sum_{k=1}^{[nt]} \chi_k^n,\quad t\in [0,1],
$$
with
$$
\chi_k^n=\mu_{k,n}\xi_{k,n}, \quad \mu_{k,n}=-\frac{1}{\sqrt{n}}\Big(\sgn (\zeta_{k,n})- \E[\sgn (\zeta_{k,n})|\Ff_{k-1,n}]\Big), \quad   \xi_{k,n}=\nabla_\theta a(\theta_0;X_{t_{k-1,n}}).
$$

There are five assumptions in \cite[Theorem 3.2]{Jac97}, two of which are satisfied trivially in our setting. Namely, both $\{\mu_{k,n}\}$ and $\{\chi^n_k\}$ are martingale difference sequences, hence the first assumption \cite[Eq.(3.10)]{Jac97} holds trivially with $B_t\equiv 0.$ Next, we will take the \textit{reference martingale} $M$ equal identically 0, hence \cite[Eq.(3.12)]{Jac97} holds trivially with $G_t\equiv 0.$ It is left to verify the following three assumptions: for any $t\in [0,1],$ 
\begin{itemize}
    \item \cite[Eq.(3.11)]{Jac97}: 
    \be\label{item1}
    \sum_{k=1}^{[nt]} \E[(\chi_k^n)^{\otimes 2} |\Ff_{k-1,n}]\mathop{\to}\limits^{\mathbb{P}} F_t := 
    \frac{1}{T}\int_0^t f(X_s)\, ds;
    \ee
      \item \cite[Eq.(3.13)]{Jac97}: for some $\eps>0,$
     \be\label{item2}
      \sum_{k=1}^{[nt]} \E[|\chi_k^n|^{2}1_{|\chi_k^n|>\eps} |\Ff_{k-1,n}] \mathop{\to}\limits^{\mathbb{P}} 0;
      \ee
  \item \cite[Eq.(3.14)]{Jac97}: for \textit{any} bounded martingale $N$ on the  filtered probability space $(\Omega,\mathcal{F},(\mathcal{F}_{t})_{t\in[0,\infty)},\P)$,
      \be\label{item3}
      \sum_{k=1}^{[nt]} \E[(N_{t_{k,n}}-N_{t_{k-1,n}})\chi_k^n |\Ff_{k-1,n}] \mathop{\to}\limits^{\mathbb{P}} 0.
     \ee     
\end{itemize}

The first two properties are easy to check. Indeed, 
$$
|\chi_k^n|^{2}\leq \frac{1}{n}|\nabla_\theta a(\theta_0, X_{t_{k-1,n}})|^2,
$$
and since the trajectory $X_t, t\in [0,T]$ is bounded and $\nabla_\theta a(\theta_0, \cdot)$ is continuous, \eqref{item2} holds true.
Next, we have 
$$
\E[(\chi_k^n)^{\otimes 2} |\Ff_{k-1,n}]=\E[(\mu_k^n)^{\otimes 2} |\Ff_{k-1,n}] f(X_{t_{k-1,n}})=\frac{1}{n}\Big(1-(\E[\sgn(\zeta_{k,n}) |\Ff_{k-1,n}])^2\Big) f(X_{t_{k-1,n}}).
$$
We have 
$$
 \frac{1}{n}\sum_{k=1}^{[nt]}f(X_{t_{k-1,n}})\mathop{\to}\limits^{\mathbb{P}} 
 F_t:=\frac1T \int_0^t f(X_s)\, ds
 $$
 by the convergence of Riemann sums. 
 Moreover,
$$
|\E[\sgn(\zeta_{k,n}) |\Ff_{k-1,n}]|\leq Ch^\delta_n W(X_{t_{k-1,n}})
$$
(see \eqref{decomp2} and subsequent estimates) and the function $W(x)$ is locally bounded. Hence,
$$
 \frac{1}{n}\sum_{k=1}^{[nt]}(\E[\sgn(\zeta_{k,n})])^2f(X_{t_{k-1,n}})\mathop{\to}\limits^{\mathbb{P}} 0,
 $$
which gives \eqref{item1}.

The main difficulty is provided by the `orthogonality' condition \eqref{item3}, which is required to hold \textit{for any} (bounded) martingale $N$.  To describe the space $\mathcal{H}^1$ of all martingales on our given filtered space, we will use the version of the Jacod-Yor theorem, which we outline below, following the exposition in \cite[Section~18.3]{CohEll15}.

\rev{Recall that we have assumed that our filtered probability space is canonical}: $\Omega=\mathbb{D}([0,\infty))$, the filtration $(\mathcal{F}_{t})_{t\in[0,\infty)}$ is the corresponding natural filtration, and $\P$ is the distribution in $\mathbb{D}([0,\infty))$ of the solution to \eqref{sde} \rev{with the given $\theta=\theta_0$}. We note that, on this specific probability space, there exists a L\'{e}vy process with the characteristic exponent \eqref{psi} such that \eqref{sde} holds. In general, we understand the solution to \eqref{sde} in the weak sense; that is, as a pair $(X,Z)$  on \textit{some} probability space. Note however that the process $Z$ in this pair can be obtained from $X$ in a measurable way because $\sigma(x)$ is separated from $0$. Indeed, denote $\Delta X_t=X_t-X_{t-}$ for the jump size at time $t$ and consider the random point measure
\be\label{NZ}
N^Z(A)=\#\left\{\rev{(s,u)}: \Big(s, \frac{1}{\sigma(X_{s-})}\Delta X_s\Big)\in A\right\}, \quad \rev{A\in \mathcal{B}(\mathbb{R}^+\times\mathbb{R}).}
\ee
This will be the random point measure corresponding to the jumps of the process $Z$. Then, $N^Z(ds,du)$ is a Poisson point measure with the intensity measure $ds\mu(du)$, and by the It\^o-L\'{e}vy decomposition we have the representation
\begin{equation}
\begin{split}
& Z_t=bt+\int_0^t\int_{|u|<1}u\widetilde{N^Z}(ds,du)+\int_0^t\int_{|u|\geq 1}u{N^Z}(ds,du),\\
& \qquad \widetilde{N^Z}(ds,du)={N^Z}(ds,du)-ds\mu(du).  
\end{split}
\label{Z}
\end{equation}
The same procedure can be repeated on the canonical probability space. Namely,  we can define $N^{Z,can}$ and $Z^{can}$ by \eqref{NZ} and \eqref{Z} with $ X^{can}$ instead of $X$. Since the laws of $ X^{can}$ and  $X$ are the same, their images under the same measurable construction are also the same. That is, $N^{Z,can}(ds,du)$ is a Poisson point measure with the intensity measure $ds\mu(du)$ and $Z^{can}$ is a L\'evy process with the characteristic exponent \eqref{psi}. Finally,  $(X^{can},Z^{can})$ satisfy \eqref{sde} by the construction.

\rev{ From now on, we operate on the canonical probability space only and  omit the superscripts, writing $X,Z$ instead of $X^{can}, Z^{can}$.} 
Denote by
$$
\mathcal{L}f(x)=a(\theta_0;x)f'(x)+\mathrm{P.V.}\int_{\mathbb{R}}(f(x+\sigma(x) u)-f(x))\nu(du),\quad f\in C^2_b(\mathbb{R}),
$$
the generator of the solution to \eqref{sde}. Denote by $\mathcal{N}$ the set of all the processes of the form 
$$
f(X_t)-\int_0^t\mathcal{L}f(X_s)\, ds, \quad t\in[0, \infty).
$$
Denote by $\mathfrak{P}(\mathcal{N})$ the set of all probability measures on $\Omega=\mathbb{D}([0,\infty))$ such that every process from $\mathcal{N}$ is a martingale; that is, the set of all \textit{solutions to the martingale problem}  $(\mathcal{L}, C_b^2(\mathbb{R}))$. 
Denote by $\mathcal{H}^\mathsf{p}(\P)$ for $\mathsf{p}\in[1,\infty]$ the space of all martingales $M$ such that $\E[\sup_t |M_t|^\mathsf{p}]<\infty$. 
The Jacod-Yor theorem (e.g. \cite[Theorem~18.3.6]{CohEll15}) states, in particular, that the following two statements are equivalent for a given probability $\P\in \mathfrak{P}(\mathcal{N})$:
\begin{itemize}
    \item the set of all stochastic integrals with respect to elements of $\mathcal{N}$ with bounded predictable integrands is dense in $\mathcal{H}^1(\P)$;
    \item $\P$ is an extremal point of $\mathfrak{P}(\mathcal{N})$, considered as a convex subset of the space of all probability measures on $(\Omega, \mathcal{F}).$
\end{itemize}

By \cite[Theorem~3.3]{AK}, the martingale problem $(\mathcal{L}, C_b^2(\mathbb{R}))$ is \textit{well posed}, i.e. has exactly one solution, and the law $\P$ of the (unique) weak solution to \eqref{sde} is this unique solution to the martingale problem. That is, 
$\mathfrak{P}(\mathcal{N})=\{\P\}$ and thus $\P$ is the extremal point of the one-point set $\mathfrak{P}(\mathcal{N})$. This means that 
stochastic integrals with respect to elements of $\mathcal{N}$ with bounded predictable integrands form a dense subset in $\mathcal{H}^1(\P)$.

We have already shown that, on the canonical probability space, 
the process $X$ satisfies \eqref{sde} with the L\'evy process which has the It\^o-L\'evy decomposition \eqref{Z}. Applying the It\^o formula, we see that arbitrary element of $\mathcal{N}$ can be written as
$$
f(X_t)-\int_0^t\mathcal{L}f(X_s)\, ds=\int_0^t\int_{\mathbb{R}} \Big[f(X_{s-}+\sigma(X_{s-}u)-f(X_{s-})\Big]\widetilde{N^Z}(ds,du).
$$
It follows that any $\P$-martingale $N$ can be obtained as a limit in $\mathcal{H}^1(\P)$ of linear combinations of the martingales of the form
\be\label{mart_appr_1}
I_t^{f,G}=\int_0^t G_s\int_{\mathbb{R}} \Big[f(X_{s-}+\sigma(X_{s-}u)-f(X_{s-})\Big]\widetilde{N^Z}(ds,du)
\ee
with $f\in C^2_b(\mathbb{R})$ and bounded predictable processes $G$. Each such linear combination is an element of the linear space 
 $\mathcal{J}_1$ of the processes of the form 
 \be\label{J}
 J_t=\int_0^t \int_{\mathbb{R}} H(s,u) \widetilde{N^Z}(ds,du)
 \ee
with predictable processes $H(s,u)$ such that $|H(s,u)|\leq C(|u|^2\wedge 1)$ a.s. for some constant $C$. 

Let us show that any \textit{bounded} martingale $N$ can be approximated by elements of $\mathcal{J}_1$  in  $\mathcal{H}^2(\P)$. By the Burkh\"older-Davis-Gundy inequality (e.g. \cite[Theorem~11.5.5]{CohEll15}), for any $J\in \mathcal{J}_1$
 $$
 c_1\E\left[[J-H]^{1/2}_1\right] \leq \|J-H\|_{\mathcal{H^1}}
 =\E\left[\sup_{t\in[0,1]}|J_t-H_t|\right] \leq C_1\E\left[ [J-H]^{1/2}_1\right]
 $$
for positive universal constants $c_1$ and $C_1$.
We have
\begin{align}
[J-N]=[J]-2[J,N]+[N]
&=\sum_{s\leq t}(\Delta J_s)^2-2 \sum_{s\leq t}\Delta J_s \Delta N_s + \langle N^c\rangle +\sum_{s\leq t}(\Delta N_s)^2
\nn\\
&= \sum_{s\leq t}(\Delta J_s - \Delta N_s)^2+ \langle N^c\rangle,
\nn
\end{align}
where we have used that $J$ does not have the continuous martingale part: $J^c=0$.
Let $\sup_t |N_t|\leq R$ a.s. and define a new process $J^R\in \mathcal{J}_1$ by changing the integrand $H(s,u)$ to 
$$
H^R(s,u)=F_R(H(s,u)), \qquad F_R(x)=\left\{
  \begin{array}{ll}
    x, & |x|\leq 2R; \\
    2R\sgn(x), & \hbox{otherwise.}
  \end{array}
\right.
$$
The time moments of jumps of $J$ and $J^R$ coincide and, because $|\Delta N_s|\leq 2R,$ we have 
$$
(\Delta J_s-\Delta N_s)^2\leq (\Delta J^R_s - \Delta N_s)^2
$$
for each jump. This yields that
$$
[J-N]\leq [J^R-N];
$$
that is, a martingale $N$ with $\sup_t |N_t|\leq R$ can be approximated in $\mathcal{H}^1(\P)$ by processes from $\mathcal{J}_1$ with their jumps bounded by $2R$. We can also stop these processes at the moment of their exit from $[-2R, 2R]$, the operation of which corresponds to the multiplication of $H(s,u)$ by $1_{[0,\tau]}(s)$ with the corresponding exit time $\tau$ and the stopped process remains in the class $\mathcal{J}_1$. This means that in $\mathcal{H}^1(\P)$, we can approximate a martingale $N$ bounded by $R$ by a sequence $\{J^n\}\subset\mathcal{J}_1$ bounded by $4R$. By the dominated convergence theorem, this approximation then holds for $\mathcal{H}^2(\P)$, as well. 

Denote by $\mathcal{J}_0$ the class of processes from $\mathcal{J}_1$ such that, in the representation \eqref{J}, the function $H(s,u)$ is bounded and equals $0$ for $|u|\leq c$, where $c>0$ may depend on the process. Clearly, by the It\^o isometry  $\mathcal{J}_0$ is dense in $\mathcal{J}_1$ in $\mathcal{H}^2$. Summarizing the above considerations, we proved that, for any bounded martingale $N$ and any $\eps>0,$ there exists $J^\eps\in \mathcal{J}_0$ such that   
\be\label{J2}
\E\left[\sup_{t\in[0,1]}(J_t^\eps-N_t)^2\right] \leq \eps.
\ee

Now we can prove the required assertion \eqref{item3}. First, we observe that \eqref{item3} holds true with $N$ replaced by any $J\in \mathcal{J}_0$. Indeed,
$$
\chi_k^n =\mu_{k,n} \xi_{k,n},
$$
where $|\mu_{k,n}|\leq n^{-1/2}$. Since \eqref{item3} states convergence in probability, we can use the standard localization technique (e.g. as in the proof of Lemma \ref{l42}) to restrict ourselves to the case $\sup_{s\in [0,T}|X_s|\leq K$. In this case, $|\xi_{k,n}|\leq C$ and thus 
\be\label{chi}
|\chi_k^n|\leq Cn^{-\frac{1}{2}}.
\ee
On the other hand, $J$ is a compensated compound Poisson process with bounded jumps, hence
\begin{align}
& \left|\E[(J_{t_{k,n}}-J_{t_{k-1,n}})\chi_k^n|\Ff_{k-1,n}]\right|
\nn\\
&{}\quad \leq Ch_nn^{-\frac{1}{2}}+Cn^{-\frac{1}{2}}\E\left[|J_{t_{k,n}}-J_{t_{k-1,n}}|
\,1_{
N^Z((t_{k-1,n},t_{k,n}],\mathbb{R}\setminus\{0\}) \ge 1
}\middle|\,\Ff_{k-1,n}\right]
\nn\\
&{}\quad \leq  Ch_nn^{-\frac{1}{2}}.
\end{align}
Since 
$$
\sumk  Ch_nn^{-\frac{1}{2}}= Cn^{-\frac{1}{2}}\to 0,
$$
this proves \eqref{item3} for $J\in \mathcal{J}_0$. 

For arbitrary bounded martingale $N$ and arbitrary $\eps>0$, take $J^\eps\in \mathcal{J}_0$ such that \eqref{J2} holds. Again, using the localization we can assume \eqref{chi}. Then,
$$\ba
  \sum_{k=1}^{[nt]} \E[(N_{t_{k,n}}-N_{t_{k-1,n}})\chi_k^n |\Ff_{k-1,n}]& = \sum_{k=1}^{[nt]} \E[(J_{t_{k,n}}-J_{t_{k-1,n}})\chi_k^n |\Ff_{k-1,n}]
  \\
  &{}\quad + \sum_{k=1}^{[nt]} \E\left[\Big((N_{t_{k,n}}-N_{t_{k-1,n}})-(J_{t_{k,n}}-J_{t_{k-1,n}})\Big)\chi_k^n \middle| \Ff_{k-1,n}\right]. 
  \ea
  $$
The first sum converges to $0$ in probability. For the second sum, we have 
\begin{align}
& \hspace{-1cm}\E\left[ \sum_{k=1}^{[nt]} \E\left[\Big((N_{t_{k,n}}-N_{t_{k-1,n}})-(J_{t_{k,n}}-J_{t_{k-1,n}})\Big)\chi_k^n \Big|\Ff_{k-1,n}\right] \right]
\\&=\E\left[\sum_{k=1}^{[nt]} \Big((N_{t_{k,n}}-N_{t_{k-1,n}})-(J_{t_{k,n}}-J_{t_{k-1,n}})\Big)\chi_k^n \right]
\\&\leq \left(\E\left[\sum_{k=1}^{[nt]} \Big((N_{t_{k,n}}-N_{t_{k-1,n}})-(J_{t_{k,n}}-J_{t_{k-1,n}})\Big)^2\right]\right)^{1/2}
\left(\E\left[\sum_{k=1}^{[nt]} (\chi_k^n)^2\right]\right)^{1/2}
\\&\leq C \E\left[\Big((N_1-N_0)-(J_1-J_0)\Big)^2\right]
\\&\leq C\eps.
\end{align}
Thus for any $\gamma>0$,
$$
\limsup_{n\to \infty}\P\left(\left|
\sum_{k=1}^{[nt]} \E[(N_{t_{k,n}}-N_{t_{k-1,n}})\chi_k^n |\Ff_{k-1,n}]
\right|>\gamma\right)\leq \frac{\eps}{\gamma}.
$$
Since $\eps>0$ is arbitrary, this actually gives  
$$
\lim_{n\to \infty}\P\left(\left|
\sum_{k=1}^{[nt]} \E[(N_{t_{k,n}}-N_{t_{k-1,n}})\chi_k^n |\Ff_{k-1,n}]
\right|>\gamma\right)=0
$$
and completes the proof of \eqref{item3}.  We have checked all the assumptions required in \cite[Theorem 3.2]{Jac97}. Applying this theorem, we get \eqref{CLT_s} and complete the proof.

\subsection{Infinite observation horizon}
\label{sB2.e}

We have the following.

\begin{thm} \label{thm--ergodic}
Let Assumption \ref{A_diss} hold true. Then, the process $X$ is ergodic with respect to $\P^{\theta_0}$, and its transition probabilities admit the following convergence rates to the invariant measure:
    \be\label{rate}
    \|P_t(\tz; x, dy)-\pi(\tz, dy)\|_{TV}\leq U(x) r(t), \quad t\geq 0,
    \ee
    where
    \begin{itemize}
        \item for $\kappa\geq 1$,
        $$
        U(x)=1+|x|^q, \quad r(t)= Ce^{-ct};
        $$
    \item for $\kappa< 1$,
        $$
        U(x)=1+|x|^{q+\kappa-1}, \quad r(t)= C(1+ct)^{-(q+\kappa-1)/(1-\kappa)}
        $$    
    \end{itemize}
    with some $C,c>0$.
\end{thm}

This result follows by applying \cite[Theorem~3.4.11]{Kul18} and \cite[Theorem~3.4.12]{Kul18} for the cases $\kappa\geq 1$ and $\kappa<1,$ respectively. The Dobrushin condition requested in the preamble to these theorems can be verified easily using the properties of the  
transition density from Section \ref{sA1} and the argument from \cite[Section~3.2.2]{Kul18}.  

Given the ergodicity of the underlying process, we have the following.

\begin{prp}\label{prop}
    For any $f\in L_1(\pi(\tz\;\cdot)),$
    \be\label{LLN}
    \frac{1}{n}\sumk f(X_{t_{k-1,n}})\mathop{\to}\limits^{\P} \int_{\mathbb{R}}f(y)\pi(\tz,dy).
    \ee
\end{prp}
\begin{proof} This statement looks very much like the Birkhoff ergodic theorem, but we cannot use the latter one here because of the $n$-dependent discretization step $h_n$. Instead, we use, with minimal changes, the arguments from \cite[Section~5]{Kul18}. Namely, consider first the stationary version of $X$, i.e. the solution to \eqref{sde} with random $X_0$ whose distribution is equal to $\pi(\tz;\cdot).$ It follows from \eqref{rate} that 
$$
\Big|\mathrm{Cov}(f(X_t), f(X_s))\Big|\leq 2U(X_s)r(t)\|f\|^2_\infty,
$$
see \cite[Corollary~5.1.8]{Kul18}, where one should take $\gamma=1, W(x)\equiv 1, V(x)=U(x).$ Then it is an easy calculation to show that for \textit{bounded} $f$ convergence \eqref{LLN} holds true in $L_2(\P)$ sense. Using the $L_1(\P)$ isometry for stationary $\P$, we extend this convergence to whole $L_1(\pi(\tz\;\cdot)),$ though in this general case \eqref{LLN} holds true in $L_1(\P)$ sense. Finally,
by \eqref{rate}, a non-stationary process with a prescribed probability arbitrarily close to 1 can be coupled in a finite time with the stationary version of $X$, which yields \eqref{LLN} in the general non-stationary setting. We omit the details of such a quite common `coupling' trick, referring the reader to \cite{Kul18} for details and references.        
\end{proof}

Assumption \ref{A_moment} yields that $G:=VW^2\in L_1(\pi(\tz\;\cdot)).$ Indeed, let $G^N=G1_{G\leq N},$ then for any $N>1$ by Proposition \ref{prop},
$$
\int_{\mathbb{R}}G^N(y)\pi(\tz;dy)=\lim_{n\to \infty}\E\left[\frac{1}{n}\sumk G^N(X_{t_{k,n}})\right]\leq C:=\sup_n \E\left[\frac{1}{n}\sumk V(X_{t_{k,n}})W(X_{t_{k,n}})^2\right].
$$
Taking the limit for $N\to \infty$, we obtain the required statement. Since $V$ is bounded and $\nabla_\theta a$ is dominated by $W$, this means that 
the functions 
$$
f(x):= V(x)^2(\left(\nabla_\theta a(\theta_0;x)\right)^{\otimes 2}, \quad g(x)=\frac{V(x)^2}{\sigma (x)}\left(\nabla_\theta a(\theta_0, x)\right)^{\otimes 2}
$$
belong to $L_1(\pi(\tz\;\cdot)),$ and applying Proposition \ref{prop} we get that 
$$
\Gamma_n\mathop{\to}\limits^{\P}  \Gamma, \quad \Sigma_n\mathop{\to}\limits^{\P}\Sigma_0. 
$$
The same argument shows that, for every $\theta$ and any $R,$ 
$$
\Lambda_{n,R}(\theta)\mathop{\to}\limits^{\P} \Lambda_{0,R}(\theta),
$$
see the notation prior to \eqref{Lambda}. Because, for a given $R$, the random fields $\Lambda_{n, R}$ are uniformly Lipschitz in $\theta$, this yields \eqref{Lambda}. 

Recall the notation:
$$
\Xi_t^n=\sum_{k=1}^{[nt]} \chi_k^n,\quad \chi_k^n=\mu_{k,n}\xi_{k,n},
$$
$$
\mu_{k,n}=-\frac{1}{\sqrt{n}}\Big(\sgn (\zeta_{k,n})- \E[\sgn (\zeta_{k,n})|\Ff_{k-1,n}]\Big), \quad   \xi_{k,n}=V(X_{t_{k-1,n}})\nabla_\theta a(\theta_0;X_{t_{k-1,n}}).
$$
Recall that 
$$
\sup_k\Big|\E[n(\mu_{k,n})^{\otimes 2} |\Ff_{k-1,n}]-1\Big|\mathop{\to} 0,
$$
Then by Proposition \ref{prop},
$$
 \sum_{k=1}^{[nt]} \E[(\chi_k^n)^{\otimes 2} |\Ff_{k-1,n}]\mathop{\to}\limits^{\mathbb{P}} t\Gamma_0.
$$
Using that $|\xi_{k,n}|$ is dominated by $V(X_{t_{k-1,n}})W(X_{t_{k-1,n}})$ and the moment Assumption \ref{A_moment}, we get that, for any  $\eps>0,$
$$  
      \sum_{k=1}^{[nt]} \E[|\chi_k^n|^{2}1_{|\chi_k^n|>\eps} |\Ff_{k-1,n}] \mathop{\to}\limits^{\mathbb{P}} 0.
$$      
Then, by the central limit theorem for martingale difference arrays (e.g. \cite{Dvo77}), we obtain
$$
\Xi_n=\Xi_1^n\Rightarrow N(0, \Gamma_0),
$$
which completes the proof of \eqref{CLT}.

\section{Some moment estimates} 
\label{hm:sec_moments}

In this section, we prove 
the moment estimate 
\begin{equation}\label{hm:lem_moment-1-gap.ineq}
    \left| \E\left[|h^{-1/\al}Z_h|^\rho\right] - \E\left[|h^{-1/\al}Z^{(\al)}_h|^\rho\right] \right| 
    \le C h^{\delta}
\end{equation}
for some $\rho>0, \delta>0$, used in Section \ref{hm:sec_ceAV}. For that purpose, we first construct a pair of processes $Z, Z^{(\alpha)}$ on the same probability space in such a way that they are close, in a sense.

Recall the decomposition \eqref{LM} of the {\lm} of $Z$: $\mu(dz)=\mu_\alpha(dz)+\nu(dz)$. The Hahn-Jordan decomposition of the finite-variation signed measure $\nu(dz)$ gives the representation by the positive and negative parts:
$$
\nu(dz)=\nu^+(dz)-\nu^-(dz).
$$
Denote
$$
\mu^{\min}(dz)=\mu_\alpha(dz)-\nu^-(dz).
$$
Then,
$$
\mu_\al(dz)=\mu^{\min}(dz)+\nu^-(dz), \quad \mu(dz)=\mu^{\min}(dz)+\nu^+(dz).
$$
Define on a certain probability space three \textit{independent} Poisson point measures
$$
N^{\min}(dz,dt), \quad N^{+}(dz,dt), \quad N^{-}(dz,dt),
$$
with the intensity measures $\mu^{\min}(dz)dt$, $\nu^+(dz)dt$, and $\nu^-(dz)dt$, respectively. Denote
$$
N_\al(dz,dt)=N^{\min}(dz,dt)+N^{-}(dz,dt), \quad N(dz,dt)=N^{\min}(dz,dt)+N^{+}(dz,dt),
$$
which are now \textit{dependent} Poisson point measures with the intensity measures $\mu_{\al}(dz)dt$ and $\mu(dz)dt$, respectively.

Define the process
$$
Z^{(\al)}_{t}=\int_0^t\int_{|z|\leq 1}z\widetilde{N}_\al(dz,ds)+\int_0^t\int_{|z|>1}z {N}_\al(dz,ds),
$$
where $\widetilde{N}_\al$ denotes the compensated $N_\al$.
Write
$$
Z_t=bt+\int_0^t\int_{|z|\leq 1}z\widetilde{N}(dz,ds)+\int_0^t\int_{|z|>1}z {N}(dz,ds)
$$
for our driving {\lp}. Moreover, let
$$
{N}^\triangle(dz,dt):={N}^+(dz,dt)+{N}^-(d(-z),dt),
$$
which defines a Poisson point measure with the intensity measure $\nu^\triangle(dz)dt$, where
$$
\nu^\triangle(dz) := \nu^+(dz)+\nu^-(d(-z)).
$$
Then,
$$
Z^\triangle_t:=Z_t-Z^{(\al)}_{t}=bt+\int_0^t\int_{|z|\leq 1}z\widetilde{N}^\triangle (dz,ds)+\int_0^t\int_{|z|>1}z {N}^\triangle(dz,ds).
$$
We have
$$
\int_{|z|\leq 1}z\nu^\triangle(dz)=\int_{|z|\leq 1}\Big(z\nu^+(dz)-z\nu^-(dz)\Big)=\int_{|z|\leq 1}z\nu(dz)=b;
$$
see \eqref{pur_jump}. 
This finally gives the representation
\begin{equation}\label{hm:Z-decomposition}
    Z=Z^{(\al)} + Z^\triangle
\end{equation}
with
$$
Z^\triangle_t = \int_0^t\int_{\mathbb{R}}z{N}^\triangle(dz,ds).
$$
Write $Z^\triangle = Z^{\triangle, small} + Z^{\triangle, large}$, where
$$
Z^{\triangle, small}_t:=\int_0^t\int_{|z|\leq1}z{N}^\triangle(dz,ds), \qquad Z^{\triangle, large}_t:=\int_0^t\int_{|z|>1}z{N}^\triangle(dz,ds).
$$
Recall that $\beta\in[0,\al/2)\subset[0,1)$ denotes the Blumenthal-Getoor index of the nuisance part of $Z$, hence of $Z^{\triangle, small}$; see Remark \ref{hm:rem_A2.3} 
and also Examples \ref{ex_A2.4} and \ref{ex_A2.4-2}.

Building on the above descriptions, in the following lemma we prove small-time moment estimates for the two components of $h^{-{1\over\alpha}}Z_h^\triangle$.
Note that Assumption \ref{A_diss} entails that $\mathbb{E}[|Z^{\triangle, large}_1|^q] < \infty$ when $T_n \to \infty$, whereas no moment conditions are required when $T_n \equiv T$. In the sequel, we suppress the explicit dependence on $q>0$ to present our claims more concisely.

\begin{lem}
\label{hm:lem_moment-1}
We have the following for $h\in(0,1]$.
\begin{enumerate}
    \item For $\mathfrak{l}\in(0,1]$ for which $\E[|Z_1^{\triangle, large}|^\mathfrak{l}]<\infty$,
    \begin{align}
    \E\left[|h^{-{1\over\alpha}}Z_h^{\triangle, large}|^\mathfrak{l}\right] &\leq Ch^{1-\frac{\mathfrak{l}}{\al}}.
    \label{hm:lem_moment-1-ineq.l}
    \end{align}
    \item If $\beta=0$, then for $\mathfrak{s}\in(0,1]$,
    \begin{align}
    \E\left[|h^{-{1\over\alpha}}Z_h^{\triangle, small}|^\mathfrak{s}\right] &\leq Ch^{1-\frac{\mathfrak{s}}{\al}}.
    \label{hm:lem_moment-1-ineq.s-zero}
    \end{align}
    If $\beta>0$, then for $\mathfrak{s}\in(0,\beta)$,
    \begin{align}
    \E\left[|h^{-{1\over\alpha}}Z_h^{\triangle, small}|^\mathfrak{s}\right] &\leq Ch^{{\mathfrak{s}\over\gamma}-{\mathfrak{s}\over\alpha}}
    \label{hm:lem_moment-1-ineq.s}
    \end{align}   
    for any $\gamma\in (\beta, \alpha\wedge 1)$.
    \end{enumerate}
\end{lem}

\begin{proof}
Both $Z^{\triangle,large}_t$ and $Z^{\triangle, small}_t$ are just sums of some jumps of the process $Z^{\triangle}$: the first one is finite, the second one is a sum of an a.s. absolutely convergent series. This observation, together with the elementary inequality 
$|\sum_{i}a|^p\leq \sum_i|a_i|^p$ for $p\in (0,1]$ and $\{a_i\}\subset \mbbr$ 
gives the estimate 
$$
\E\left[\left|\int_0^t\int_U zN^\triangle(dz, ds)\right|^p\right] \leq
\E\left[ \int_0^t\int_U |z|^pN^\triangle(dz, ds) \right] = t\,\int_U |z|^p\nu^\triangle(dz)
$$
for any $p\in (0,1]$ and $U\subset \mathbb{R}$ such that 
$$
\int_U |z|^p\nu^\triangle(dz)<\infty.
$$
Taking in this estimate $p=\mathfrak{l}$ and $U=\{|u|>1\},$ we get \eqref{hm:lem_moment-1-ineq.l}. Also, taking $p=\mathfrak{s}$ and $U=\{|u|\leq 1\},$ we get \eqref{hm:lem_moment-1-ineq.s-zero}. Finally, applying the same inequality with $p=\gamma$ and $U=\{|u|\leq 1\},$ we get 
$$
 \E\left[|h^{-{1\over\alpha}}Z_h^{\triangle, small}|^\gamma\right] \leq Ch^{1-{\gamma\over\alpha}},
 $$
where we used that $|\nu^\triangle|(dz)\leq 2|\nu|(dz)$ has the Blumenthal-Getoor index $\beta<\gamma$, and therefore
$$
\int_{|u|\leq 1} |z|^\gamma\nu^\triangle(dz)<\infty.
$$
Then \eqref{hm:lem_moment-1-ineq.s} follows by the Lyapunov inequality. 

The estimate \eqref{hm:lem_moment-1-gap.ineq} is immediate from the previous three ones.
\end{proof}

Combining \eqref{hm:lem_moment-1-ineq.l}, \eqref{hm:lem_moment-1-ineq.s-zero}, and \eqref{hm:lem_moment-1-ineq.s}, we get the following.

\begin{cor} For any $\rho\in(0,1_{\{\beta=0\}}+1_{\{\beta>0\}}\beta)$ and $\gamma\in (\beta, \alpha\wedge 1)$, \eqref{hm:lem_moment-1-gap.ineq} holds true with 
$$\delta=\left(1_{\{\beta=0\}}+1_{\{\beta>0\}}\frac{\rho}{\gamma}\right)-{\rho\over\alpha}>0.
$$ 
\end{cor}

It is straightforward to extend Lemma \ref{hm:lem_moment-1} and the corresponding moment bounds to stochastic integrals: 
for $t\ge 0$, $h\in(0,1]$, $\mathfrak{s}\in(0,\beta)$, and for $\mathfrak{l}\in(0,1]$ for which $\E[|Z_1^{\triangle, large}|^\mathfrak{l}]<\infty$, we can find a 
$\lambda> 1/2$ such that
\begin{align}
\E\left[\left| h^{-{1\over\alpha}} \int_t^{t+h} \xi_{s-} 
dZ_s^{\triangle, large}\right|^\mathfrak{l}\right] &\leq C h^{1-\frac{\mathfrak{l}}{\al}} 
\sup_{t\in [t, t+h]}\E\left[|\xi_t|^{\mathfrak{l}}\right],
\nn\\
\E\left[\left| h^{-{1\over\alpha}} \int_t^{t+h} \xi_{s-} 
dZ_s^{\triangle, small}\right|^\mathfrak{s}\right] &\leq C h^{\lambda}
\sup_{t\in [t, t+h]}\E\left[|\xi_t|^{\mathfrak{s}}\right].
\nn
\end{align}
See also Lemma \ref{hm:lem_moment.ineq} for the related moment estimate associated with the locally $\al$-stable small-jump part.

The following Example \ref{hm:rate-example} is more than necessary for our use in proving the second item in Lemma \ref{hm:lem_2}, but is important in statistical applications with the stable quasi-likelihood function. See \cite{CleGlo19}, \cite{CleGlo20}, \cite{Mas19spa}, and also \cite{Mas_SQMLE_TechRep_2021}.

\begin{exmpl}
\label{hm:rate-example}
Suppose that the {\lm} $\mu(dz)$ admits a Lebesgue density on $\mbbr\setminus\{0\}$:
\begin{equation}
\mu(dz) = \frac{c_\al}{|z|^{\al+1}} \,\mathfrak{m}(z) dz
\end{equation}
for some function $\mathfrak{m}(z)=1+\mathfrak{m}_0(z)$ with $\mathfrak{m}_0(z)\ne 0$ ($z\ne 0$) satisfying that $\mathfrak{m}_0(z)=O(|z|)$ for $|z|\to 0$.
This is related to \eqref{ex_A2.4-2-eq1} in Example \ref{ex_A2.4-2} (namely, $\mathfrak{n}(z)=c_\al |z|^{-(\al+1)}\mathfrak{m}_0(z)$) and is a rather particular case compared with the general setting. Still, it includes many popular specific examples, such as the (exponentially) tempered stable, the normal tempered stable, the generalized hyperbolic, including the Student-$t$, the Meixner, and so on. 

Suppose that
\begin{equation}\label{hm:c.ex-1}
    \beta = 0 \vee (\al-1),
\end{equation}
which is satisfied in many contexts, including the aforementioned examples.
Under the present assumptions, Lemma \ref{hm:lem_moment-1} gives \eqref{hm:lem_moment-1-ineq.l}, \eqref{hm:lem_moment-1-ineq.s-zero}, and \eqref{hm:lem_moment-1-ineq.s}, with $\beta<\gamma<(\al\wedge 1)$. 
We further suppose that
\begin{equation}
    \mathfrak{m} := (\mathfrak{l} \vee \mathfrak{s}) < \frac{\al}{2}.
\end{equation}
Then, we have $1-\mathfrak{m}/\al>1/2$, so that the upper bounds in \eqref{hm:lem_moment-1-ineq.l} and \eqref{hm:lem_moment-1-ineq.s-zero} become $o(h^{1/2})$ for $h\to 0$.
Let us observe that in the estimate \eqref{hm:lem_moment-1-ineq.s}, stated for $\beta>0$ and $\gamma\in (\beta, \alpha\wedge 1)$, we can take
\begin{equation}\label{hm:moment_rem1-1}
    {\mathfrak{s}\over\gamma}-{\mathfrak{s}\over\alpha} 
    > \frac12.
\end{equation}
To see this, it suffices to consider $\al>1$ since $\beta=0$ otherwise; then, $\beta=\al-1>0$. 
As was seen in the proof of Lemma \ref{hm:lem_moment-1}, we can take $\gamma(>\beta)$ as close to $\beta$ as possible, say $\gamma=\beta+\eps$ with $\eps>0$ small enough.
Then, \eqref{hm:moment_rem1-1} is equivalent to
\begin{equation}
\mathfrak{s} > \frac12 \left( \frac{1}{\al-1+\eps} - \frac{1}{\al}\right)^{-1}.
\end{equation}
This together with the condition $\mathfrak{s}<\beta=\al-1$ leads to
\begin{equation}
\al > 1 + \frac12 \left( \frac{1}{\al-1+\eps} - \frac{1}{\al}\right)^{-1}.
\end{equation}
The lower bound is monotonically decreasing to $1+\frac12 \al(\al-1)$ for $\eps\downarrow 0$, and the inequality $\al > 1+ \frac12 \al(\al-1)$ $\iff$ $\al\in(1,2)$ always holds, thus concluding \eqref{hm:moment_rem1-1}.

Hence, the upper bounds in \eqref{hm:lem_moment-1-ineq.l}, \eqref{hm:lem_moment-1-ineq.s-zero}, and \eqref{hm:lem_moment-1-ineq.s} are all $o(h^{1/2})$ for $h\to 0$. It follows from \eqref{hm:lem_moment-1-gap.ineq} that
\begin{equation}\label{hm:moment_rem1-2}
    \frac{1}{\sqrt{h}}\left| \E\left[|h^{-1/\al}Z_h|^\rho\right] - \E\left[|h^{-1/\al}Z^{(\al)}_h|^\rho\right] \right| \to 0
\end{equation}
for any $\rho\in(0,1_{\{\beta=0\}}+1_{\{\beta>0\}}\beta)$ for which $\E[|Z_1|^\rho]<\infty$.
The estimate \eqref{hm:moment_rem1-2} is an $L^1$-local limit theorem with convergence rate. It serves as a variant of \cite[Proposition 4.1]{CleGlo20}.
\end{exmpl}

\color{black}

\section{Numerical experiments}
\label{hm:sec_simulations}

For simulations, we consider the following SDE with basis-expansion type drift:
\begin{equation}\label{hm:sim-1}
dX_t = \left(\theta_1 a_1(X_t) + \theta_2 a_2(X_t)\right) dt + dZ_t
\nonumber
\end{equation}
with $X_0=0$, where $a_1$ and $a_2$ are known measurable functions and $Z$ is a symmetric $\al$-stable {\lp} corresponding to the L\'{e}vy measure $\mu_\al(du)$, hence $\sig(x)\equiv 1$ here; recall \eqref{LM}.
For the regressor choice, we focus on the Euler approximation (Example \ref{exF1}).
To generate sample paths and implement our estimation procedure, we used the \texttt{yuima} R package; see \cite{YUIMA14} and \cite{IacYos18} for details.

We consider the two specific choices:
$(a_1(x),a_2(x))=(1,-x)$ (Ornstein-Uhlenbeck process, Example \ref{ex:linear.drift}) in Section \ref{hm:sec_ou}, and $(a_1(x),a_2(x))=(-|x|^{\langle 0.8 \rangle},-x)$ (random perturbation of Bernoulli's ODE, Example \ref{ex:Bernoulli.ODE.drift}) in Section \ref{hm:sec_bernoulli}. 
For both cases, we set the true value of $\theta$ as $\tz=(\theta_{1,0},\theta_{2,0})=(3,1)$.
We use the following experimental scheme and design for both models.
\begin{itemize}
\item First, we generate $L= 1000$ independent Euler-approximation sample path $\mathbb{X}^{\ast}_T\coloneqq(X_{s_{l}})_{l=0}^{N^\ast_T}$ in $[0,T]$ for $N^\ast_T=10000$ (namely, $s_l=l \,T/N^\ast_T$), where we set $T=1$ and $50$.

\item Then, for each sample path and each $T$, we obtain discrete-time samples $X^{(n)} \coloneqq (X_{t_{k,n}})_{k=0}^{n}$ for $n = 500$, $2000$, and $4000$ by thinning the original path $\mathbb{X}^{\ast}_T$, that is, by retaining every $m$th point with $m = 40$, $10$, and $5$, respectively; this mimics discrete-time sampling from the continuous-time model.

\item For each $X^{(n)}$, we compute the LAD estimator $\tes$ and also $(\aes(\rho),\hat{\sigma}_n(\rho))$ defined by \eqref{hm:def_alpha.est} and \eqref{hm:sigma.estimator} with $\rho=0.5$; for $T=50$, we supposed the large-$T$ asymptotics and used the weight function $V(x)=1/(1+x^2)$ in the LAD contrast function \eqref{lad.cf}.

\end{itemize}
For a single path, the estimates, summary statistics, and figures are obtained almost instantly, and all the numerical results and figures based on the $L$ paths took only a few minutes to compute on a Mac Studio with an Apple M1 Max chip and 64 GB of memory.
Simulations with a non-constant scale coefficient $\sigma(x)$ showed similar finite-sample performance; we omit the details.

\subsection{Ornstein-Uhlenbeck process}
\label{hm:sec_ou}

We consider estimating $\theta=(\theta_1,\theta_2)$ of
\begin{equation}
dX_t = (\theta_1 -\theta_2 X_t) dt + dZ_t,
\nonumber
\end{equation}
with $\al=1.5$ for the index of $Z$. 
The process $X$ is exponentially ergodic; see Theorem \ref{thm--ergodic} together with Assumption \ref{A_diss}.

Tables \ref{tab:alpha-sigma-ou-alpha-weighted-T1} and \ref{tab:lad-ou-alpha-weighted-T1} (resp. Tables \ref{tab:alpha-sigma-ou-alpha-weighted-T50} and \ref{tab:lad-ou-alpha-weighted-T50}) summarize the estimates for $T=1$ (resp. $T=50$), where IQR denotes the interquartile range.
In Tables \ref{tab:lad-ou-alpha-weighted-T1} and \ref{tab:lad-ou-alpha-weighted-T50}, we also report the Monte Carlo mean and SD of the Studentized statistics based on the consistent estimators of the asymptotic variances discussed in Section \ref{hm:sec_ceAV}.
Figures \ref{fig:ou-1} to \ref{fig:ou-3} (resp. Figures \ref{fig:ou-4} to \ref{fig:ou-6}) visualize the results for $T=1$ (resp. $T=50$).
From these results, we observe the following.

\begin{itemize}
\item Both estimators $\hat{\alpha}_n(\rho)$ and $\hat{\sigma}_n(\rho)$ perform reasonably well, although the latter shows a slight upward bias; this is typical in estimating a positive parameter. As $T$ increases, the performance of $\hat{\alpha}_n(\rho)$ remains nearly unchanged, whereas that of $\hat{\sigma}_n(\rho)$ improves substantially.

\item For the LAD estimator $\tes$, the rate of convergence is $\sqrt{n}\,h_n^{1-1/\al} = T^{1/3} n^{1/6}$.
For each $T\in\{1,50\}$, the performance improves as $n$ increases, or equivalently, as $h_n$ decreases.
We can clearly see that increasing $T$ yields a much greater improvement than increasing $n$.
For $T=1$, the Studentized statistics have SD close to one but a positive mean shift, which decreases only slowly because the rate $T^{1/3} n^{1/6}$ grows slowly in $n$. Note that the bias almost disappears for $T=50$.

\item The bottom right panels of Figures \ref{fig:ou-3} and \ref{fig:ou-6} show that the effect of the sampling frequency diminishes as $T$ increases; that is, for larger $T$, the cases $m=10$ and $m=5$ 
lead to similar performance.
\end{itemize}

\begin{table}[htbp]
\centering
\footnotesize
\setlength{\tabcolsep}{2.5pt}
\caption{Estimators of the index and scale for $dX_t=(\theta_1-\theta_2 X_t)\,dt+dZ_t$ with $T=1$, based on $L=1000$ replications and $\rho=0.5$. True values: $\alpha=1.5$, $\sigma=1$. The column s.e.\ is the Monte Carlo standard error of the bias.}
\label{tab:alpha-sigma-ou-alpha-weighted-T1}
\begin{tabular}{rrrrrrrrrr}
\toprule
$m$ & $n$ & $h$ & Mean & Bias & s.e. & Median & SD & IQR & RMSE \\
\midrule
\multicolumn{10}{l}{$\hat{\alpha}_n(\rho)$} \\
40 & 250 & 0.004 & 1.5387 & 0.0387 & 0.0072 & 1.5100 & 0.2280 & 0.2971 & 0.2311 \\
10 & 1000 & 0.001 & 1.5072 & 0.0072 & 0.0034 & 1.5018 & 0.1075 & 0.1455 & 0.1077 \\
5 & 2000 & 0.0005 & 1.5012 & 0.0012 & 0.0023 & 1.4988 & 0.0740 & 0.0917 & 0.0740 \\
\midrule
\multicolumn{10}{l}{$\hat{\sigma}_n(\rho)$} \\
40 & 250 & 0.004 & 1.0705 & 0.0705 & 0.0185 & 0.9697 & 0.5848 & 0.5159 & 0.5887 \\
10 & 1000 & 0.001 & 1.0379 & 0.0379 & 0.0108 & 0.9879 & 0.3412 & 0.3402 & 0.3431 \\
5 & 2000 & 0.0005 & 1.0276 & 0.0276 & 0.0080 & 0.9960 & 0.2525 & 0.2507 & 0.2539 \\
\bottomrule
\end{tabular}
\end{table}

\begin{table}[htbp]
\centering
\footnotesize
\setlength{\tabcolsep}{2.5pt}
\caption{LAD estimates for $dX_t=(\theta_1-\theta_2 X_t)\,dt+dZ_t$ with $T=1$, based on $L=1000$ replications. True values: $\theta_1=3$, $\theta_2=1$. The column s.e.\ is the Monte Carlo standard error of the bias; $z$ is standardized with $\hat{\alpha}_n(\rho)$, $\rho=0.5$, and the estimated asymptotic covariance.}
\label{tab:lad-ou-alpha-weighted-T1}
\begin{tabular}{rrrrrrrrrrrr}
\toprule
$m$ & $n$ & $h$ & Mean & Bias & s.e. & Median & SD & IQR & RMSE & Mean($z$) & SD($z$) \\
\midrule
\multicolumn{12}{l}{$\theta_1$} \\
40 & 250 & 0.004 & 3.6847 & 0.6847 & 0.0512 & 3.4765 & 1.6198 & 1.8920 & 1.7578 & 0.3423 & 1.0745 \\
10 & 1000 & 0.001 & 3.3764 & 0.3764 & 0.0411 & 3.1954 & 1.2987 & 1.5423 & 1.3515 & 0.2223 & 1.0934 \\
5 & 2000 & 0.0005 & 3.3322 & 0.3322 & 0.0363 & 3.2132 & 1.1484 & 1.3245 & 1.1949 & 0.2114 & 1.0607 \\
\midrule
\multicolumn{12}{l}{$\theta_2$} \\
40 & 250 & 0.004 & 1.6431 & 0.6431 & 0.0421 & 1.4412 & 1.3310 & 1.4448 & 1.4776 & 0.4683 & 1.0520 \\
10 & 1000 & 0.001 & 1.3716 & 0.3716 & 0.0335 & 1.2536 & 1.0579 & 1.1554 & 1.1208 & 0.3456 & 1.0441 \\
5 & 2000 & 0.0005 & 1.2982 & 0.2982 & 0.0292 & 1.1897 & 0.9235 & 1.0259 & 0.9700 & 0.3121 & 1.0405 \\
\bottomrule
\end{tabular}
\end{table}

\begin{table}[htbp]
\centering
\footnotesize
\setlength{\tabcolsep}{2.5pt}
\caption{Estimators of the index and scale for $dX_t=(\theta_1-\theta_2 X_t)\,dt+dZ_t$ with $T=50$, based on $L=1000$ replications and $\rho=0.5$. True values: $\alpha=1.5$, $\sigma=1$. The column s.e.\ is the Monte Carlo standard error of the bias.}
\label{tab:alpha-sigma-ou-alpha-weighted-T50}
\begin{tabular}{rrrrrrrrrr}
\toprule
$m$ & $n$ & $h$ & Mean & Bias & s.e. & Median & SD & IQR & RMSE \\
\midrule
\multicolumn{10}{l}{$\hat{\alpha}_n(\rho)$} \\
40 & 250 & 0.2 & 1.5582 & 0.0582 & 0.0071 & 1.5185 & 0.2239 & 0.2867 & 0.2312 \\
10 & 1000 & 0.05 & 1.5083 & 0.0083 & 0.0033 & 1.5017 & 0.1057 & 0.1420 & 0.1059 \\
5 & 2000 & 0.025 & 1.5015 & 0.0015 & 0.0023 & 1.4988 & 0.0733 & 0.0926 & 0.0732\\
\midrule
\multicolumn{10}{l}{$\hat{\sigma}_n(\rho)$} \\
40 & 250 & 0.2 & 1.0093 & 0.0093 & 0.0045 & 0.9938 & 0.1432 & 0.1423 & 0.1434 \\
10 & 1000 & 0.05 & 1.0038 & 0.0038 & 0.0030 & 0.9942 & 0.0956 & 0.1144 & 0.0956 \\
5 & 2000 & 0.025 & 1.0058 & 0.0058 & 0.0027 & 0.9991 & 0.0862 & 0.1043 & 0.0864 \\
\bottomrule
\end{tabular}
\end{table}

\begin{table}[htbp]
\centering
\footnotesize
\setlength{\tabcolsep}{2.5pt}
\caption{LAD estimates for $dX_t=(\theta_1-\theta_2 X_t)\,dt+dZ_t$ with $T=50$ and weight $V(x)=1/(1+x^2)$, based on $L=1000$ replications. True values: $\theta_1=3$, $\theta_2=1$. The column s.e.\ is the Monte Carlo standard error of the bias; $z$ is standardized with $\hat{\alpha}_n(\rho)$, $\rho=0.5$, and the estimated asymptotic covariance.}
\label{tab:lad-ou-alpha-weighted-T50}
\begin{tabular}{rrrrrrrrrrrr}
\toprule
$m$ & $n$ & $h$ & Mean & Bias & s.e. & Median & SD & IQR & RMSE & Mean($z$) & SD($z$) \\
\midrule
\multicolumn{12}{l}{$\theta_1$} \\
40 & 250 & 0.2 & 2.8335 & $-$0.1665 & 0.0190 & 2.7851 & 0.5999 & 0.7821 & 0.6223 & $-$0.3411 & 0.9835 \\
10 & 1000 & 0.05 & 3.0003 & 0.0003 & 0.0161 & 2.9938 & 0.5081 & 0.6655 & 0.5079 & $-$0.0514 & 1.0269 \\
5 & 2000 & 0.025 & 3.0418 & 0.0418 & 0.0141 & 3.0270 & 0.4469 & 0.5811 & 0.4486 & 0.0579 & 1.0138 \\
\midrule
\multicolumn{12}{l}{$\theta_2$} \\
40 & 250 & 0.2 & 0.9471 & $-$0.0529 & 0.0064 & 0.9309 & 0.2035 & 0.2644 & 0.2102 & $-$0.3067 & 0.9885 \\
10 & 1000 & 0.05 & 1.0034 & 0.0034 & 0.0054 & 0.9973 & 0.1695 & 0.2125 & 0.1694 & $-$0.0176 & 1.0072 \\
5 & 2000 & 0.025 & 1.0163 & 0.0163 & 0.0047 & 1.0154 & 0.1488 & 0.1787 & 0.1496 & 0.0871 & 1.0045 \\
\bottomrule
\end{tabular}
\end{table}


\begin{figure}[htbp]
  \centering
  \begin{minipage}{0.48\textwidth}
    \centering
    \includegraphics[width=\linewidth]{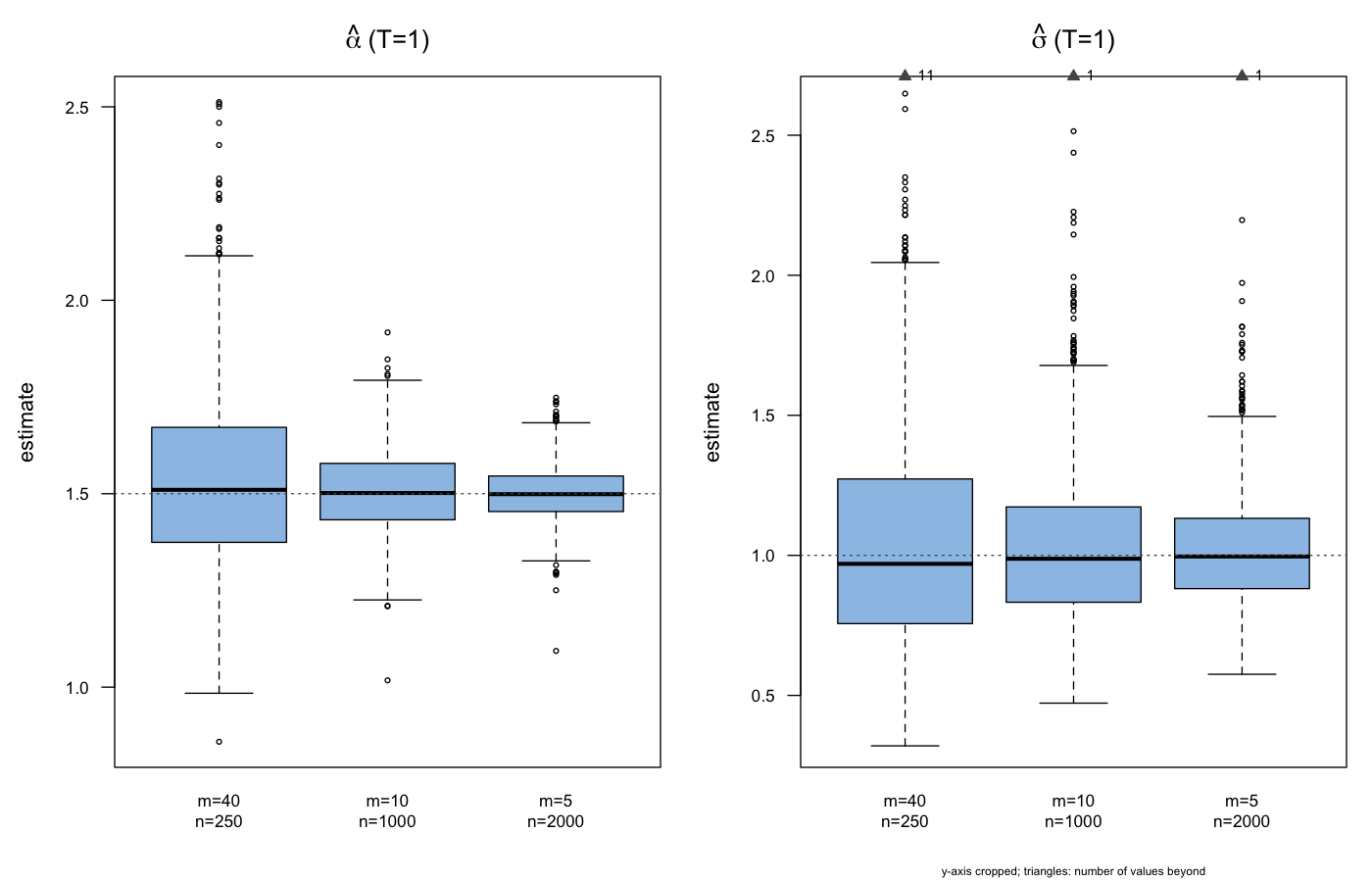}
  \end{minipage}
  \hfill
  \begin{minipage}{0.48\textwidth}
    \centering
    \includegraphics[width=\linewidth]{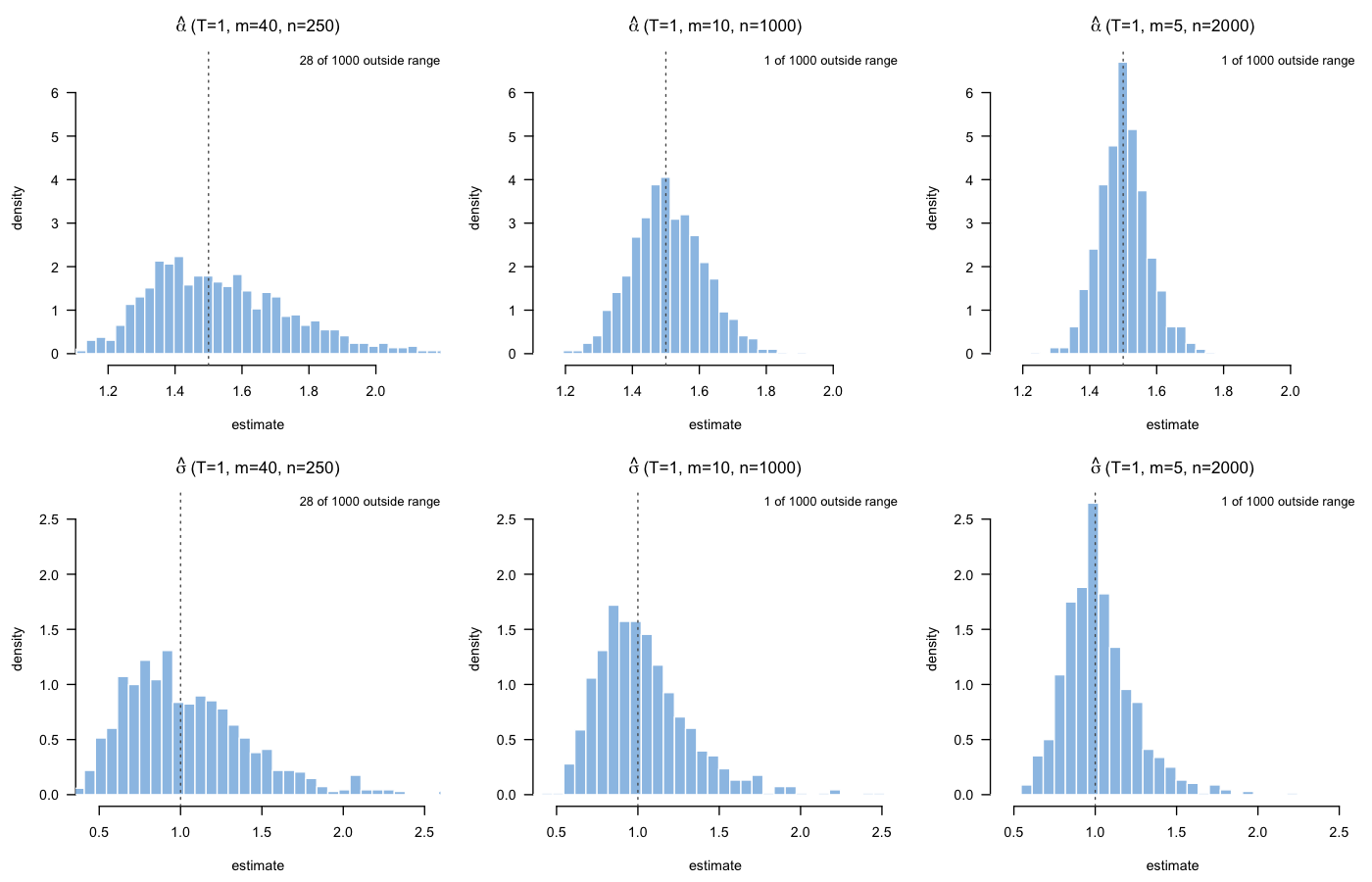}
  \end{minipage}
  \caption{The box plots and histograms of $\hat{\alpha}_n(\rho)$ and $\hat{\sigma}_n(\rho)$ for $T=1$, corresponding to Table \ref{tab:alpha-sigma-ou-alpha-weighted-T1}.}
  \label{fig:ou-1}
\end{figure}

\begin{figure}[htbp]
  \centering
  \begin{minipage}{0.48\textwidth}
    \centering
    \includegraphics[width=\linewidth]{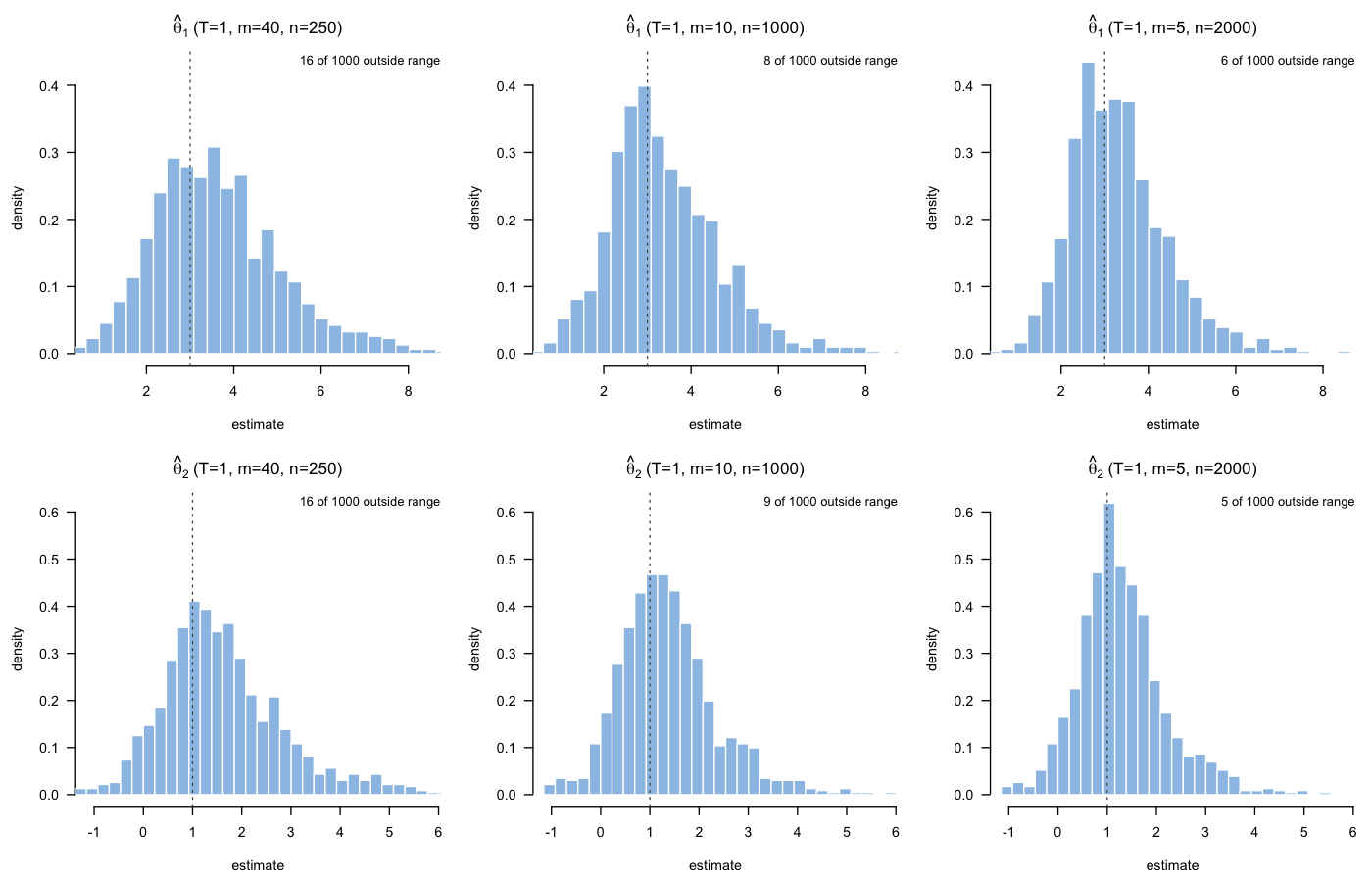}
  \end{minipage}
  \hfill
  \begin{minipage}{0.48\textwidth}
    \centering
    \includegraphics[width=\linewidth]{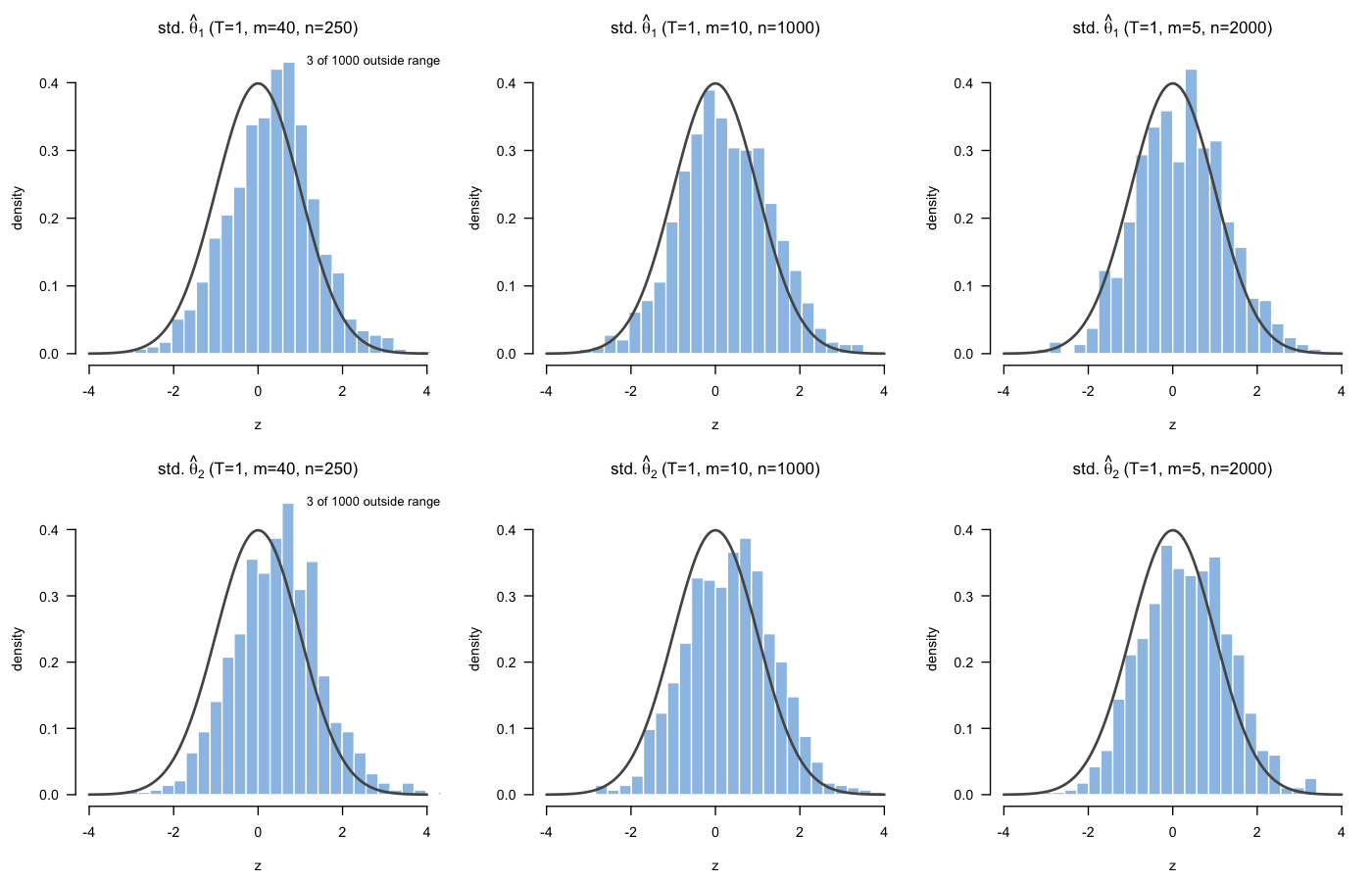}
  \end{minipage}
  \caption{The histograms of the unstandardized and standardized (Studentized) LAD estimators $\tes$ for $T=1$, corresponding to Table \ref{tab:lad-ou-alpha-weighted-T1}; the solid lines in the right panel show the target standard normal density.
  }
  \label{fig:ou-2}
\end{figure}

\begin{figure}[htbp]
  \centering
  \includegraphics[width=\textwidth]{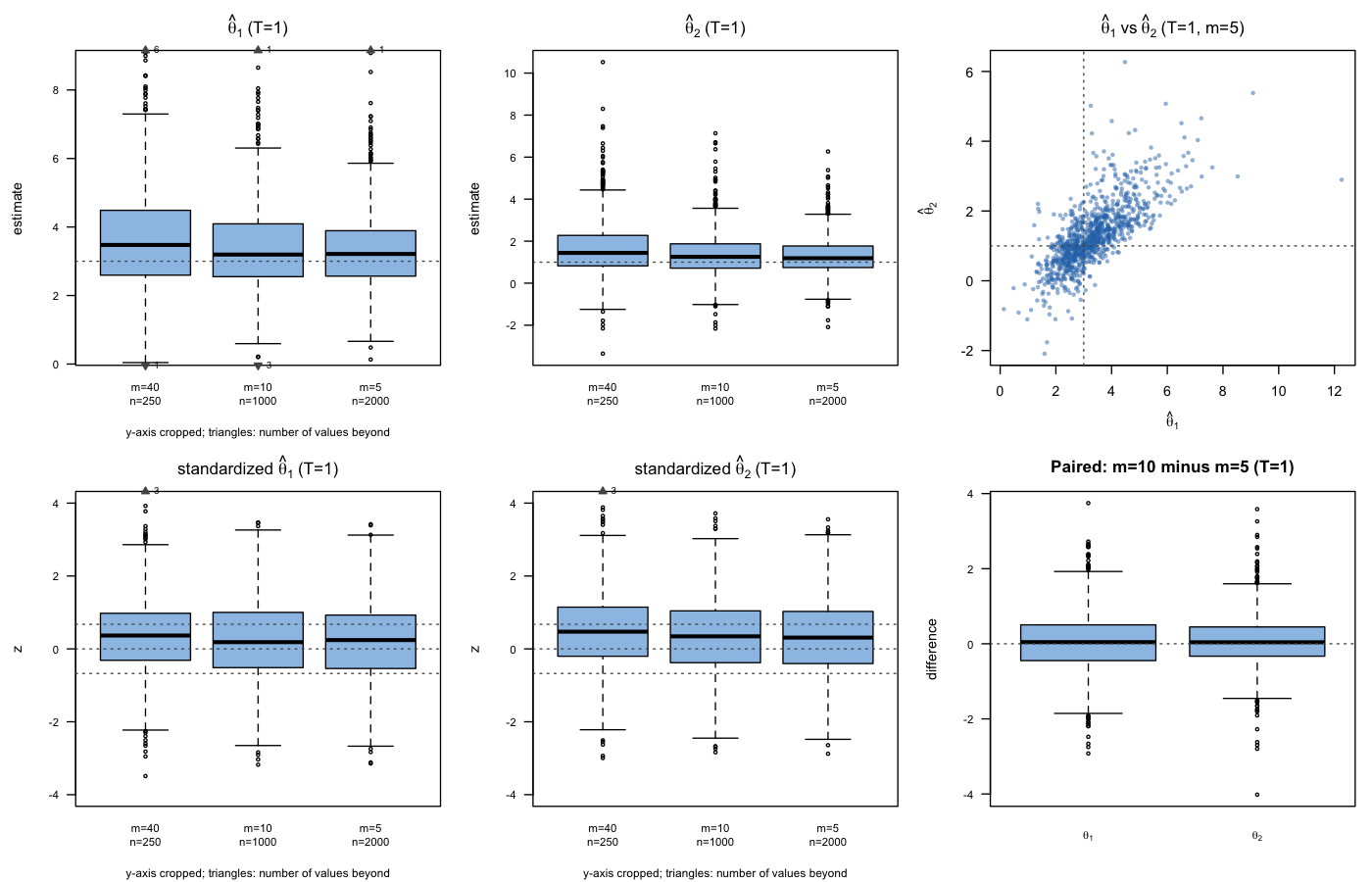}
  \caption{The box plots and histograms of the LAD estimators $\tes$ for $T=1$ (left and middle colomns), corresponding to Table \ref{tab:lad-ou-alpha-weighted-T1}.
  The upper right panel shows the scatter plot of $\hat{\theta}_{1,n}$ and $\hat{\theta}_{2,n}$.
  The bottom right panel shows the boxplots of the differences of each estimate ($\hat{\theta}_{1,n}$ and $\hat{\theta}_{2,n}$) for $m=10$ and $5$, showing their biases and variations.}
  \label{fig:ou-3}
\end{figure}

\begin{figure}[htbp]
  \centering
  \begin{minipage}{0.48\textwidth}
    \centering
    \includegraphics[width=\linewidth]{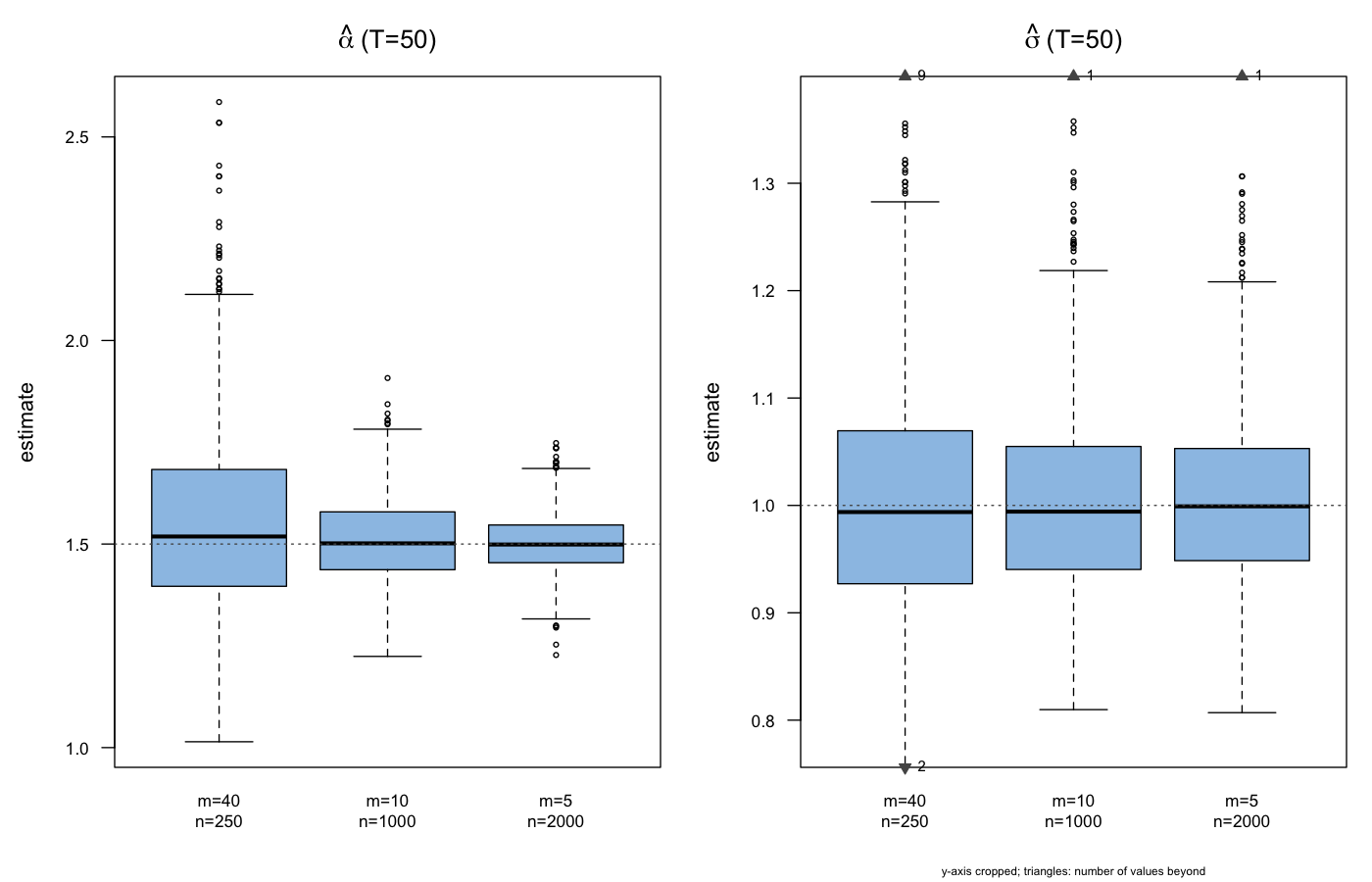}
  \end{minipage}
  \hfill
  \begin{minipage}{0.48\textwidth}
    \centering
    \includegraphics[width=\linewidth]{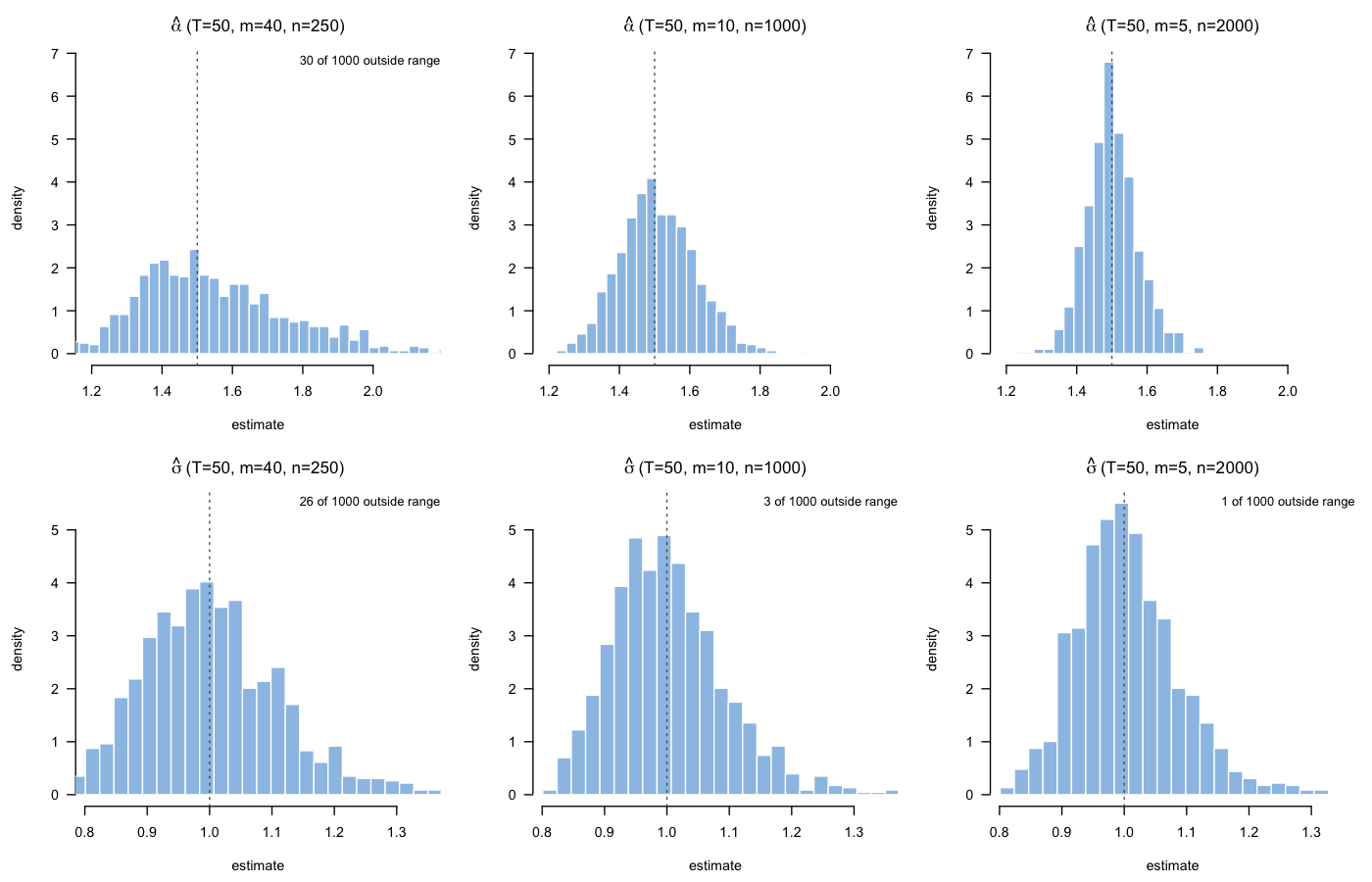}
  \end{minipage}
  \caption{The box plots and histograms of $\hat{\alpha}_n(\rho)$ and $\hat{\sigma}_n(\rho)$ for $T=50$, corresponding to Table \ref{tab:alpha-sigma-ou-alpha-weighted-T50}.}
  \label{fig:ou-4}
\end{figure}

\begin{figure}[htbp]
  \centering
  \begin{minipage}{0.48\textwidth}
    \centering
    \includegraphics[width=\linewidth]{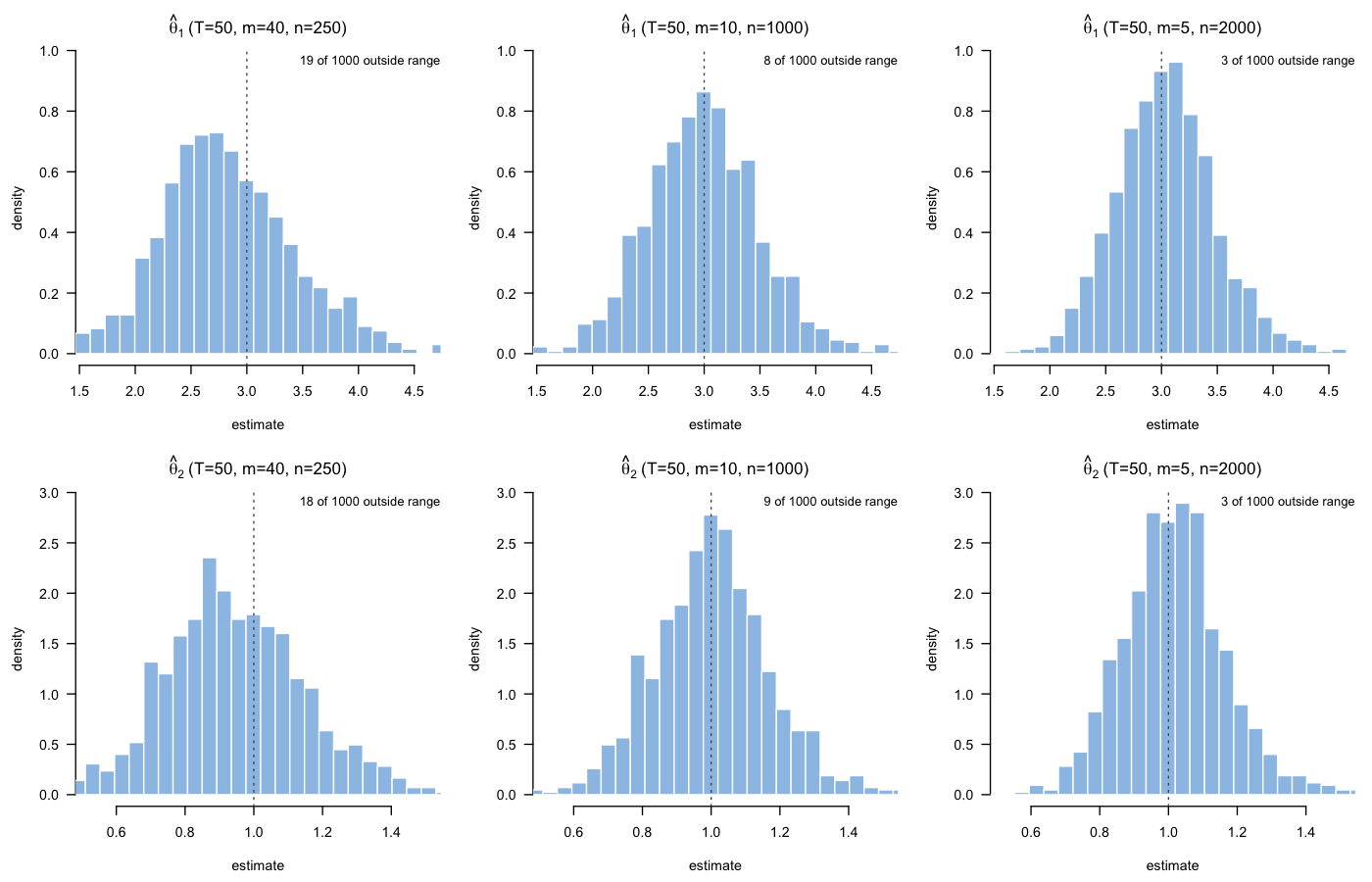}
  \end{minipage}
  \hfill
  \begin{minipage}{0.48\textwidth}
    \centering
    \includegraphics[width=\linewidth]{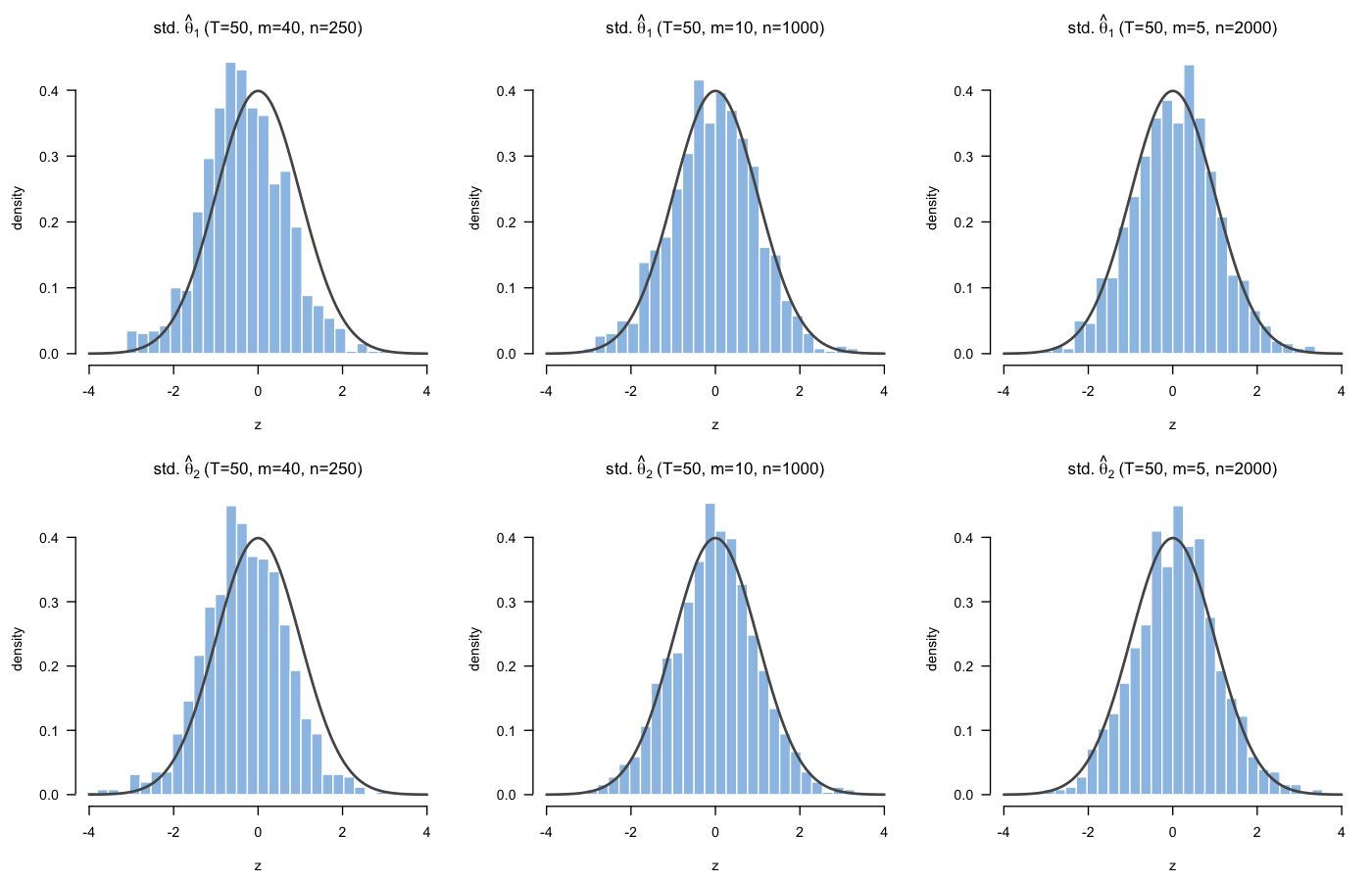}
  \end{minipage}
  \caption{The histograms of the unstandardized and standardized (Studentized) LAD estimators $\tes$ for $T=50$, corresponding to Table \ref{tab:lad-ou-alpha-weighted-T50}; the solid lines in the right panel show the target standard normal density.}
  \label{fig:ou-5}
\end{figure}

\begin{figure}[htbp]
  \centering
  \includegraphics[width=\textwidth]{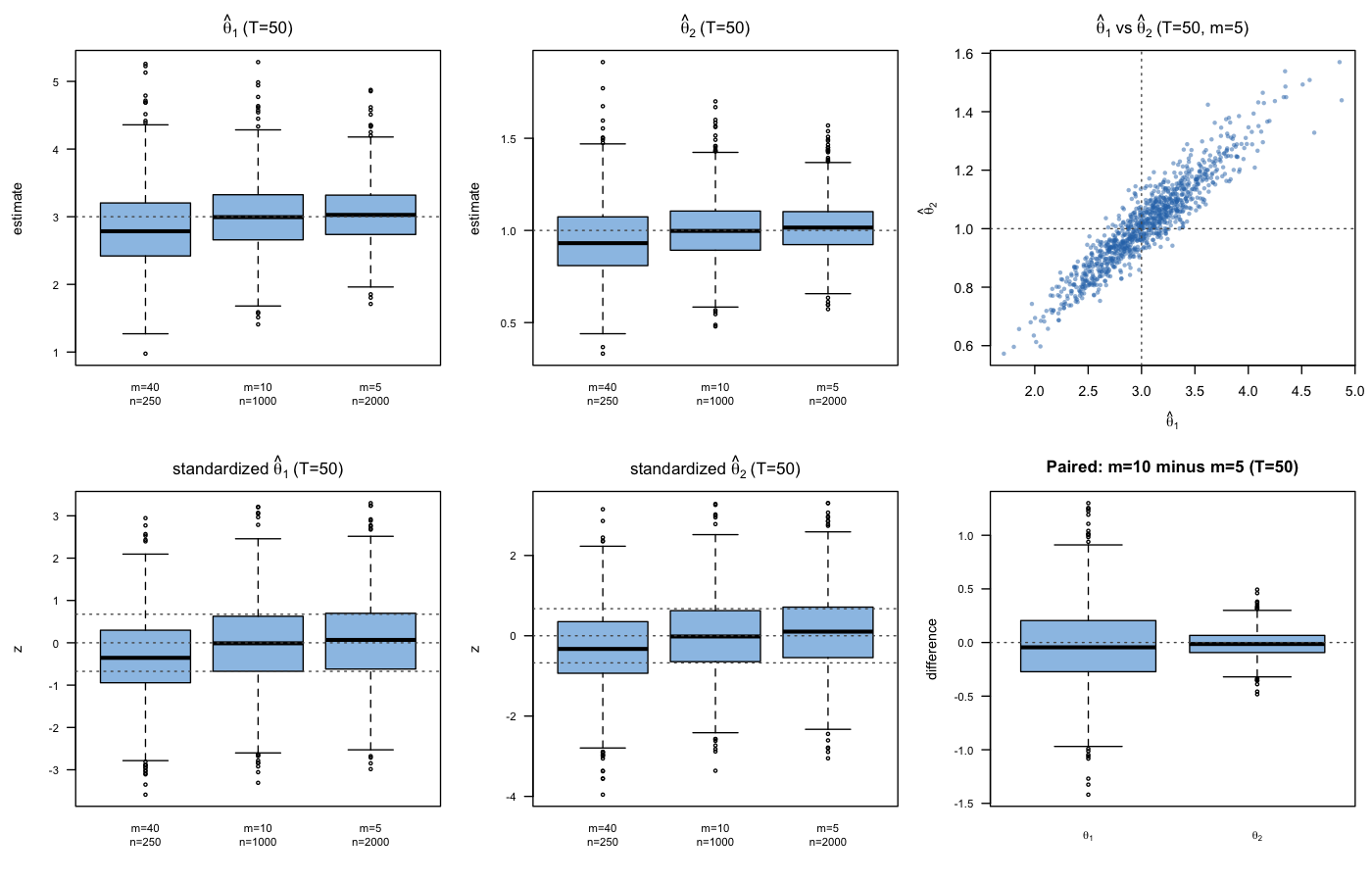}
  \caption{The box plots and histograms of the LAD estimators $\tes$ for $T=50$ (left and middle colomns), corresponding to Table \ref{tab:lad-ou-alpha-weighted-T50}.
  The upper right panel shows the scatter plot of $\hat{\theta}_{1,n}$ and $\hat{\theta}_{2,n}$.
  The bottom right panel shows the boxplots of the differences of each estimate ($\hat{\theta}_{1,n}$ and $\hat{\theta}_{2,n}$) for $m=10$ and $5$, showing their biases and variations.}
  \label{fig:ou-6}
\end{figure}


\subsection{Random perturbation of Bernoulli's ODE}
\label{hm:sec_bernoulli}

Next, we consider estimating $\theta=(\theta_1,\theta_2)$ of
\begin{equation}
dX_t = (-\theta_1 X_t^{\langle 0.8 \rangle} -\theta_2 X_t) dt + dZ_t,
\nonumber
\end{equation}
with $\al=1.3$ for the index of $Z$.
Again, Theorem \ref{thm--ergodic} implies that the process $X$ is exponentially ergodic.

The results are shown in Tables \ref{tab:alpha-sigma-nonlinear2-alpha-weighted-T1} to \ref{tab:lad-nonlinear2-alpha-weighted-T50}, and Figures \ref{fig:nl-1} to \ref{fig:nl-6}.
Overall, we observe tendencies similar to those in the case of the Ornstein--Uhlenbeck process, in line with the theoretical findings. Nevertheless, several points are worth mentioning.

\begin{itemize}
\item The drift $-\theta_1\,x^{\langle 0.8 \rangle}-\theta_2 x$ has a unique stable equilibrium at $0$, so the sample paths stay near the origin from the start (now $X_0=0$), where $V(x)=1/(1+x^2)$ is close to one in the case of $T=50$; in the present setting, the weighted LAD estimator does not suffer from a small effective sample size, and almost no extreme outliers are observed even for $T=50$.
\item The estimators $\hat{\alpha}_n(\rho)$ and $\hat{\sigma}_n(\rho)$ perform well, except for $\hat{\alpha}_n(\rho)$ with $T=50$ and $m=40$ ($h_n=0.2$), which shows a clear upward bias. 
This is presumably because the strong mean reversion during each sampling interval makes the second-order differences no longer dominated by the jumps; the bias almost vanishes for $m=10$ and $5$. 
\item For $T=1$, the means of $\hat{\theta}_{1,n}$ and $\hat{\theta}_{2,n}$ are biased, 
but the biases are small relative to the SDs. 
The Studentized statistics 
behave in good agreement with the standard normal limit. 
For $T=50$ and $m=5$ ($h_n=0.025$), the bias is small, 
and the Studentized statistics are close to $N(0,1)$. 
For coarser sampling, $\hat{\theta}_{1,n}$ is biased downward and $\hat{\theta}_{2,n}$ upward; the bias decreases roughly in proportion to $h_n$, 
which suggests a discretization bias, and the paired differences between $m=10$ and $m=5$ are significant.

\item The scatter plots show a strong negative correlation between $\hat{\theta}_{1,n}$ and $\hat{\theta}_{2,n}$ with slope about $-1$ for both $T=1$ and $T=50$. Since $x^{\langle 0.8\rangle}=\mathrm{sgn}(x)|x|^{0.8}\approx x$ for $|x|$ near $1$, the drift is approximately $-(\theta_1+\theta_2)x$ there, so the sum $\theta_1+\theta_2$ is well identified while the individual parameters are harder to separate; this also explains why the biases of $\hat{\theta}_{1,n}$ and $\hat{\theta}_{2,n}$ have opposite signs.
\end{itemize}

\begin{table}[htbp]
\centering
\footnotesize
\setlength{\tabcolsep}{2.5pt}
\caption{Estimators of the index and scale for $dX_t=(-\theta_1\,X_t^{\langle 0.8 \rangle}-\theta_2 X_t)\,dt+dZ_t$ with $T=1$, based on $L=1000$ replications and $\rho=0.5$. True values: $\alpha=1.3$, $\sigma=1$. The column s.e.\ is the Monte Carlo standard error of the bias.}
\label{tab:alpha-sigma-nonlinear2-alpha-weighted-T1}
\begin{tabular}{rrrrrrrrrr}
\toprule
$m$ & $n$ & $h$ & Mean & Bias & s.e. & Median & SD & IQR & RMSE \\
\midrule
\multicolumn{10}{l}{$\hat{\alpha}_n(\rho)$} \\
40 & 250 & 0.004 & 1.3377 & 0.0377 & 0.0061 & 1.3169 & 0.1935 & 0.2435 & 0.1971 \\
10 & 1000 & 0.001 & 1.3088 & 0.0088 & 0.0030 & 1.3060 & 0.0936 & 0.1244 & 0.0940 \\
5 & 2000 & 0.0005 & 1.3030 & 0.0030 & 0.0021 & 1.3030 & 0.0659 & 0.0786 & 0.0659 \\
\midrule
\multicolumn{10}{l}{$\hat{\sigma}_n(\rho)$} \\
40 & 250 & 0.004 & 1.1011 & 0.1011 & 0.0299 & 0.9443 & 0.9441 & 0.5703 & 0.9491 \\
10 & 1000 & 0.001 & 1.0549 & 0.0549 & 0.0185 & 0.9665 & 0.5862 & 0.4057 & 0.5885 \\
5 & 2000 & 0.0005 & 1.0395 & 0.0395 & 0.0139 & 0.9827 & 0.4400 & 0.3074 & 0.4416 \\
\bottomrule
\end{tabular}
\end{table}

\begin{table}[htbp]
\centering
\footnotesize
\setlength{\tabcolsep}{2.5pt}
\caption{LAD estimates for $dX_t=(-\theta_1\,X_t^{\langle 0.8 \rangle}-\theta_2 X_t)\,dt+dZ_t$ with $T=1$, based on $L=1000$ replications. True values: $\theta_1=3$, $\theta_2=1$. The column s.e.\ is the Monte Carlo standard error of the bias; $z$ is standardized with $\hat{\alpha}_n(\rho)$, $\rho=0.5$, and the estimated asymptotic covariance.}
\label{tab:lad-nonlinear2-alpha-weighted-T1}
\begin{tabular}{rrrrrrrrrrrr}
\toprule
$m$ & $n$ & $h$ & Mean & Bias & s.e. & Median & SD & IQR & RMSE & Mean($z$) & SD($z$) \\
\midrule
\multicolumn{12}{l}{$\theta_1$} \\
40 & 250 & 0.004 & 0.3001 & $-$2.6999 & 0.3796 & 1.5928 & 12.0025 & 12.0104 & 12.2966 & $-$0.1894 & 1.0551 \\
10 & 1000 & 0.001 & 1.6289 & $-$1.3711 & 0.2506 & 2.1343 & 7.9262 & 7.6249 & 8.0400 & $-$0.1575 & 0.9916 \\
5 & 2000 & 0.0005 & 2.1050 & $-$0.8950 & 0.2063 & 2.4903 & 6.5239 & 6.5716 & 6.5818 & $-$0.1204 & 0.9791 \\
\midrule
\multicolumn{12}{l}{$\theta_2$} \\
40 & 250 & 0.004 & 4.8259 & 3.8259 & 0.4937 & 2.5712 & 15.6113 & 13.3826 & 16.0657 & 0.2034 & 1.0569 \\
10 & 1000 & 0.001 & 2.8975 & 1.8975 & 0.3194 & 1.9738 & 10.0990 & 8.5480 & 10.2707 & 0.1673 & 0.9941 \\
5 & 2000 & 0.0005 & 2.2688 & 1.2688 & 0.2626 & 1.4974 & 8.3037 & 7.4369 & 8.3959 & 0.1304 & 0.9798 \\
\bottomrule
\end{tabular}
\end{table}

\begin{table}[htbp]
\centering
\footnotesize
\setlength{\tabcolsep}{2.5pt}
\caption{Estimators of the index and scale for $dX_t=(-\theta_1\,X_t^{\langle 0.8 \rangle}-\theta_2 X_t)\,dt+dZ_t$ with $T=50$, based on $L=1000$ replications and $\rho=0.5$. True values: $\alpha=1.3$, $\sigma=1$. The column s.e.\ is the Monte Carlo standard error of the bias.}
\label{tab:alpha-sigma-nonlinear2-alpha-weighted-T50}
\begin{tabular}{rrrrrrrrrr}
\toprule
$m$ & $n$ & $h$ & Mean & Bias & s.e. & Median & SD & IQR & RMSE \\
\midrule
\multicolumn{10}{l}{$\hat{\alpha}_n(\rho)$} \\
40 & 250 & 0.2 & 1.5364 & 0.2364 & 0.0069 & 1.5001 & 0.2178 & 0.2680 & 0.3214 \\
10 & 1000 & 0.05 & 1.3327 & 0.0327 & 0.0026 & 1.3273 & 0.0827 & 0.1145 & 0.0889 \\
5 & 2000 & 0.025 & 1.3101 & 0.0101 & 0.0019 & 1.3081 & 0.0597 & 0.0778 & 0.0605 \\
\midrule
\multicolumn{10}{l}{$\hat{\sigma}_n(\rho)$} \\
40 & 250 & 0.2 & 1.0442 & 0.0442 & 0.0089 & 1.0032 & 0.2802 & 0.2108 & 0.2835 \\
10 & 1000 & 0.05 & 1.0019 & 0.0019 & 0.0044 & 0.9852 & 0.1376 & 0.1431 & 0.1375 \\
5 & 2000 & 0.025 & 1.0077 & 0.0077 & 0.0035 & 0.9952 & 0.1122 & 0.1282 & 0.1124 \\
\bottomrule
\end{tabular}
\end{table}

\begin{table}[htbp]
\centering
\footnotesize
\setlength{\tabcolsep}{2.5pt}
\caption{LAD estimates for $dX_t=(-\theta_1\,X_t^{\langle 0.8 \rangle} -\theta_2 X_t)\,dt+dZ_t$ with $T=50$ and weight $V(x)=1/(1+x^2)$, based on $L=1000$ replications. True values: $\theta_1=3$, $\theta_2=1$. The column s.e.\ is the Monte Carlo standard error of the bias; $z$ is standardized with $\hat{\alpha}_n(\rho)$, $\rho=0.5$, and the estimated asymptotic covariance.}
\label{tab:lad-nonlinear2-alpha-weighted-T50}
\begin{tabular}{rrrrrrrrrrrr}
\toprule
$m$ & $n$ & $h$ & Mean & Bias & s.e. & Median & SD & IQR & RMSE & Mean($z$) & SD($z$) \\
\midrule
\multicolumn{12}{l}{$\theta_1$} \\
40 & 250 & 0.2 & 1.2398 & $-$1.7602 & 0.0389 & 1.3091 & 1.2293 & 1.5639 & 2.1466 & $-$1.0722 & 0.8032 \\
10 & 1000 & 0.05 & 2.5224 & $-$0.4776 & 0.0309 & 2.5514 & 0.9762 & 1.2810 & 1.0863 & $-$0.4923 & 0.9682 \\
5 & 2000 & 0.025 & 2.8002 & $-$0.1998 & 0.0260 & 2.8036 & 0.8207 & 1.0787 & 0.8442 & $-$0.2571 & 1.0124 \\
\midrule
\multicolumn{12}{l}{$\theta_2$} \\
40 & 250 & 0.2 & 1.6298 & 0.6298 & 0.0378 & 1.5461 & 1.1961 & 1.4498 & 1.3513 & 0.4383 & 0.8252 \\
10 & 1000 & 0.05 & 1.1864 & 0.1864 & 0.0293 & 1.1453 & 0.9259 & 1.1099 & 0.9440 & 0.2090 & 1.0113 \\
5 & 2000 & 0.025 & 1.0724 & 0.0724 & 0.0243 & 1.0777 & 0.7685 & 0.9590 & 0.7716 & 0.1023 & 1.0522 \\
\bottomrule
\end{tabular}
\end{table}


\begin{figure}[htbp]
  \centering
  \begin{minipage}{0.48\textwidth}
    \centering
    \includegraphics[width=\linewidth]{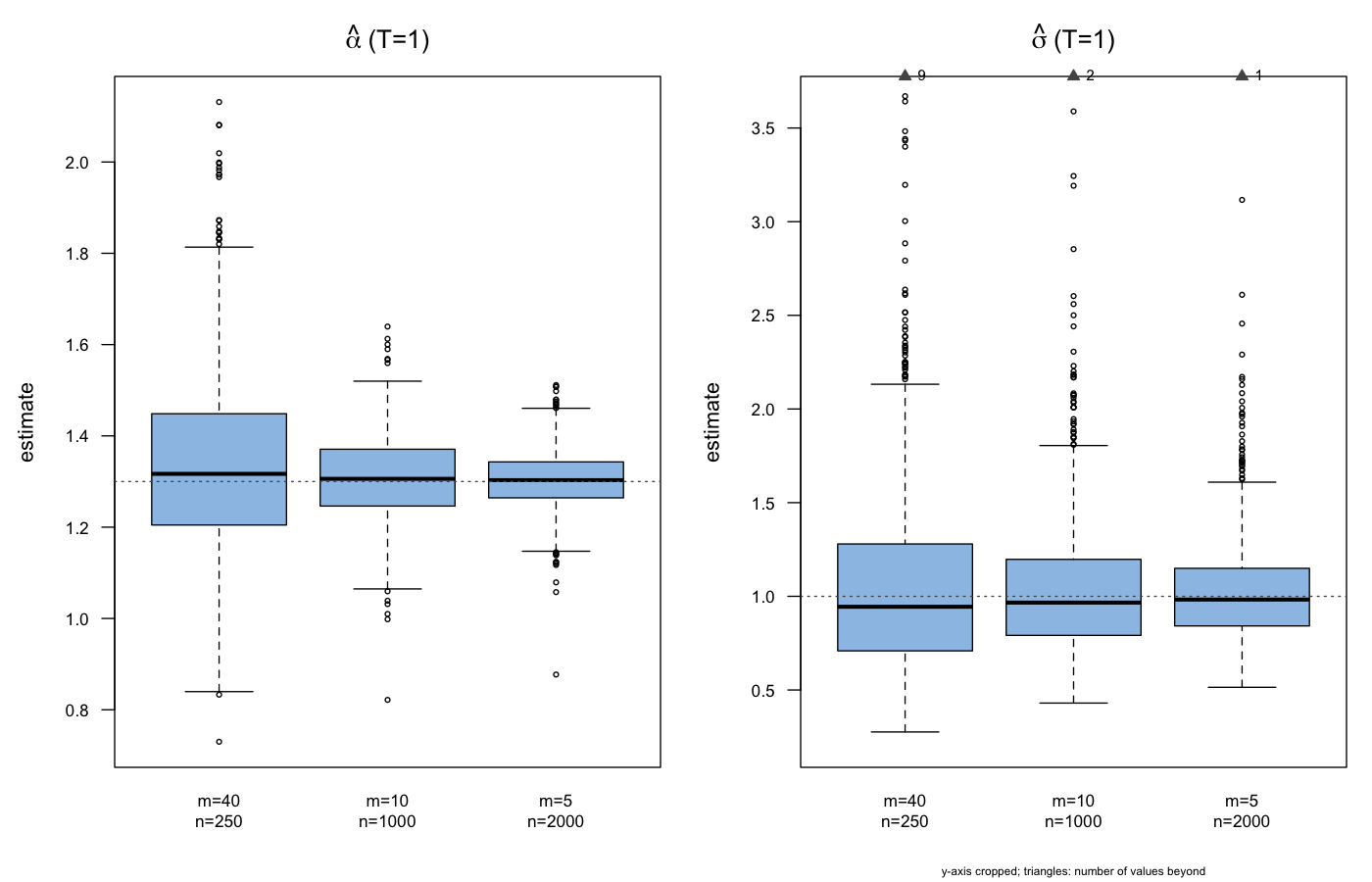}
  \end{minipage}
  \hfill
  \begin{minipage}{0.48\textwidth}
    \centering
    \includegraphics[width=\linewidth]{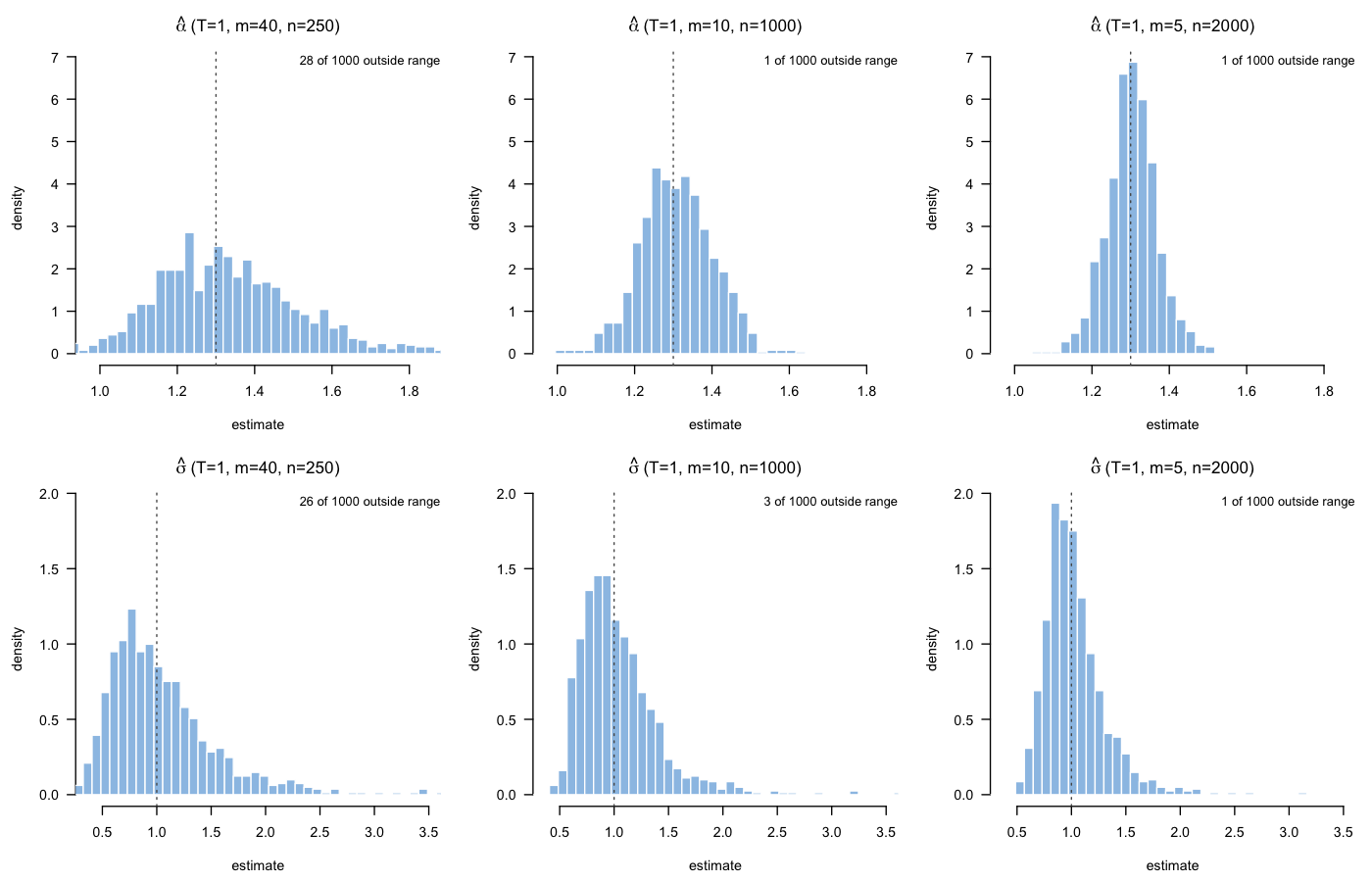}
  \end{minipage}
  \caption{The box plots and histograms of $\hat{\alpha}_n(\rho)$ and $\hat{\sigma}_n(\rho)$ for $T=1$, corresponding to Table \ref{tab:alpha-sigma-nonlinear2-alpha-weighted-T1}.}
  \label{fig:nl-1}
\end{figure}

\begin{figure}[htbp]
  \centering
  \begin{minipage}{0.48\textwidth}
    \centering
    \includegraphics[width=\linewidth]{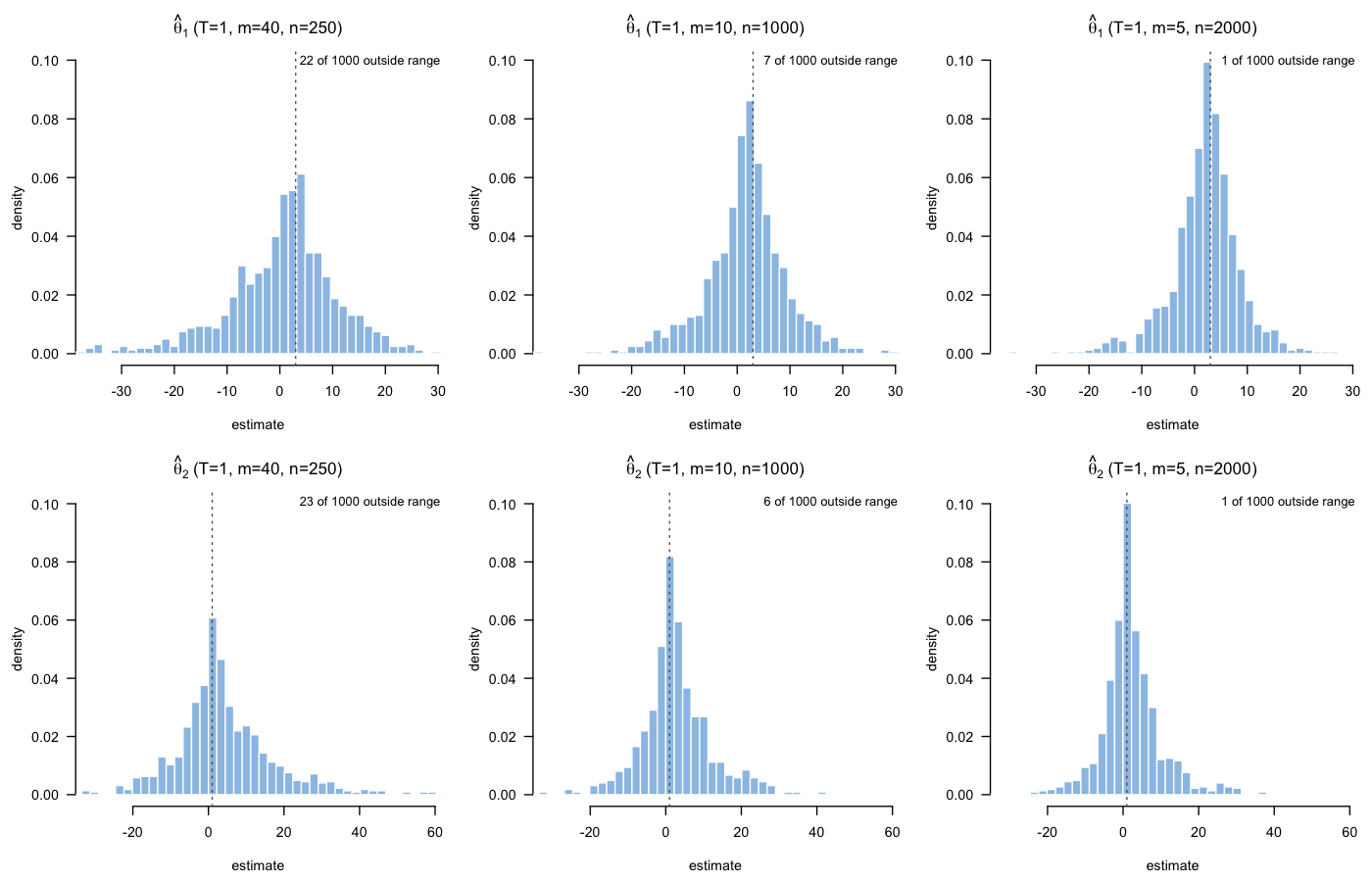}
  \end{minipage}
  \hfill
  \begin{minipage}{0.48\textwidth}
    \centering
    \includegraphics[width=\linewidth]{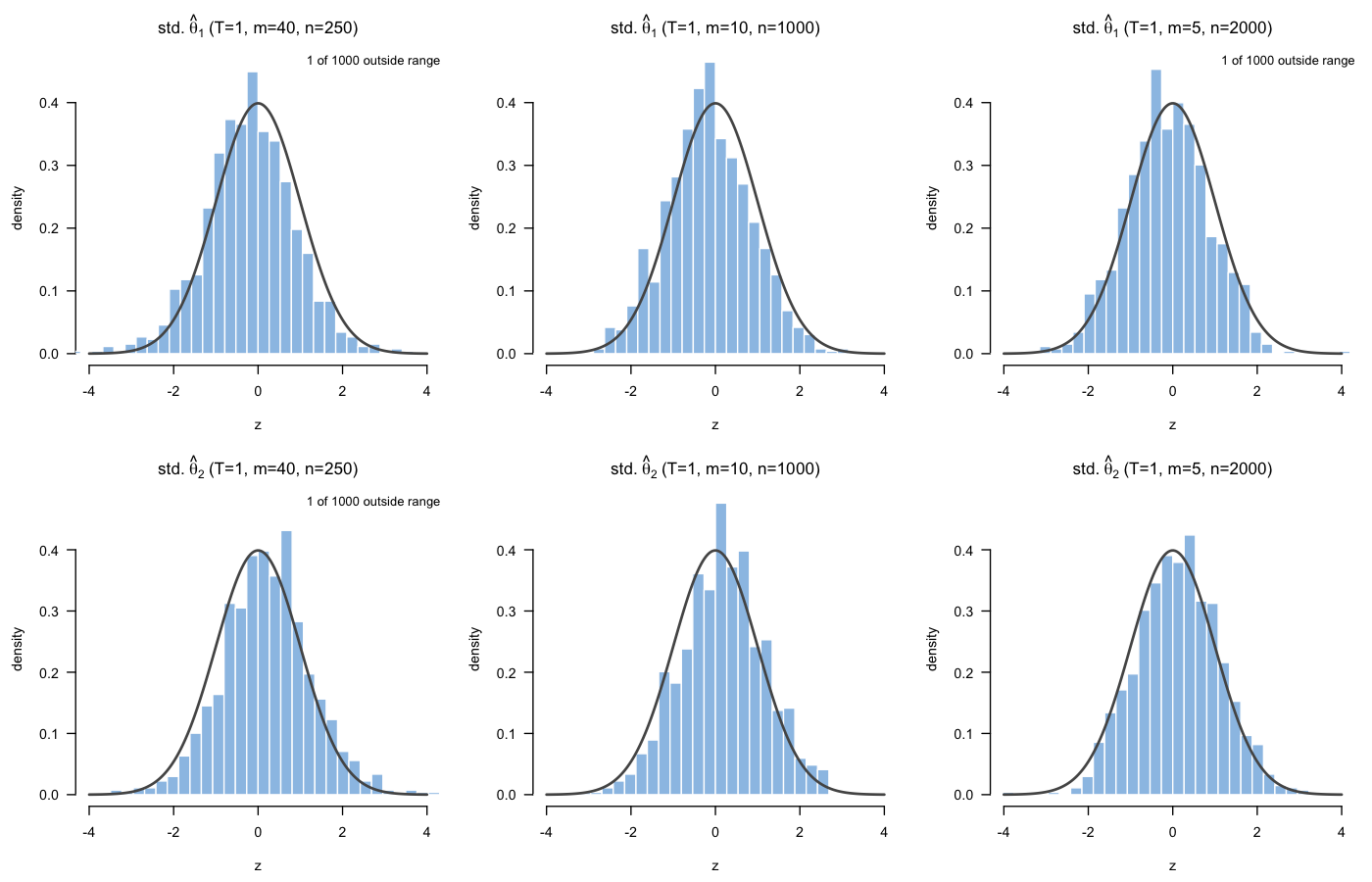}
  \end{minipage}
  \caption{The histograms of the unstandardized and standardized (Studentized) LAD estimators $\tes$ for $T=1$, corresponding to Table \ref{tab:lad-nonlinear2-alpha-weighted-T1}; the solid lines in the right panel show the target standard normal density.}
  \label{fig:nl-2}
\end{figure}

\begin{figure}[htbp]
  \centering
  \includegraphics[width=\textwidth]{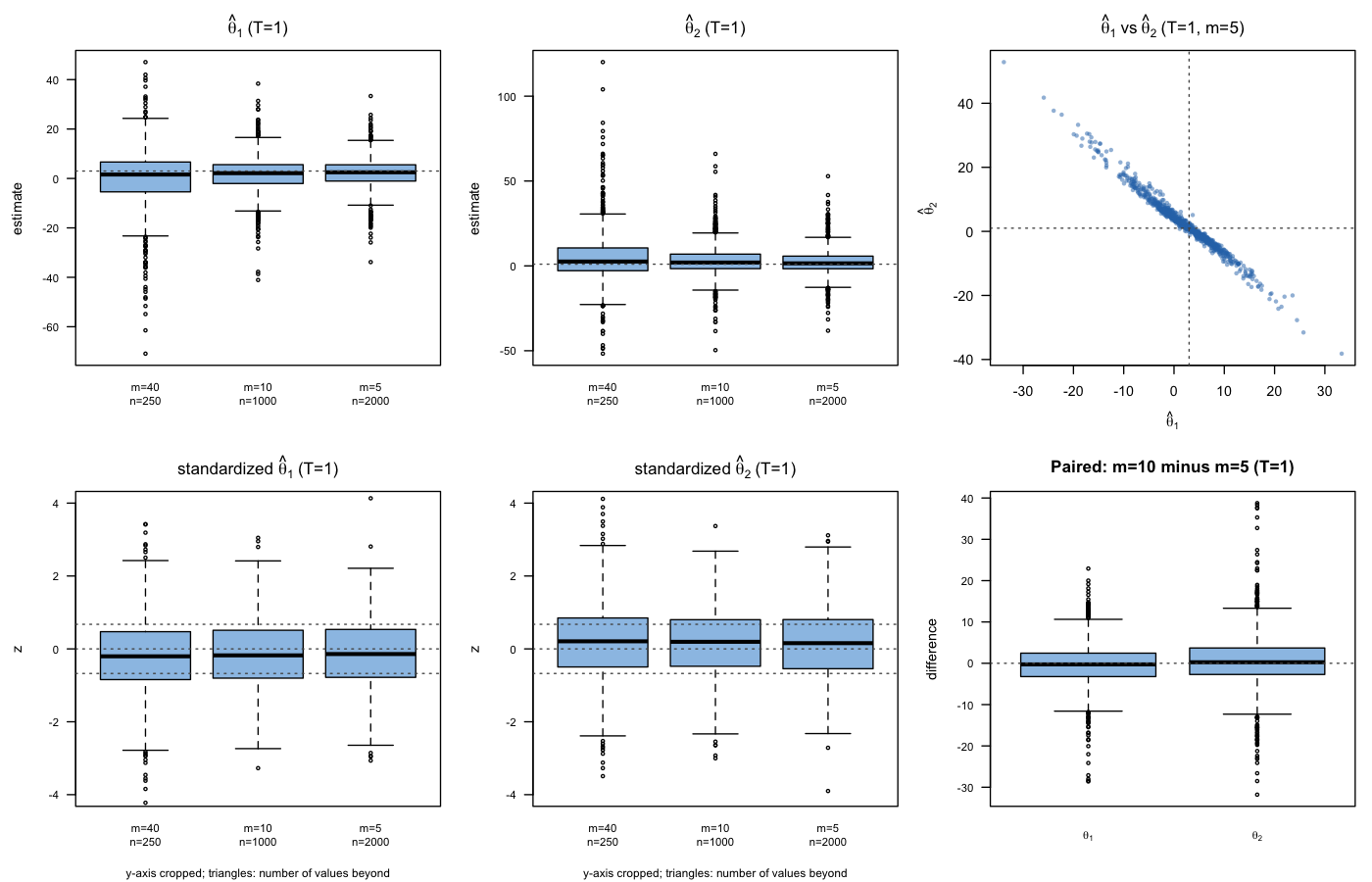}
  \caption{The box plots and histograms of the LAD estimators $\tes$ for $T=1$ (left and middle colomns), corresponding to Table \ref{tab:lad-nonlinear2-alpha-weighted-T1}.
  The upper right panel shows the scatter plot of $\hat{\theta}_{1,n}$ and $\hat{\theta}_{2,n}$.
  The bottom right panel shows the boxplots of the differences of each estimate ($\hat{\theta}_{1,n}$ and $\hat{\theta}_{2,n}$) for $m=10$ and $5$, showing their biases and variations.}
  \label{fig:nl-3}
\end{figure}

\begin{figure}[htbp]
  \centering
  \begin{minipage}{0.48\textwidth}
    \centering
    \includegraphics[width=\linewidth]{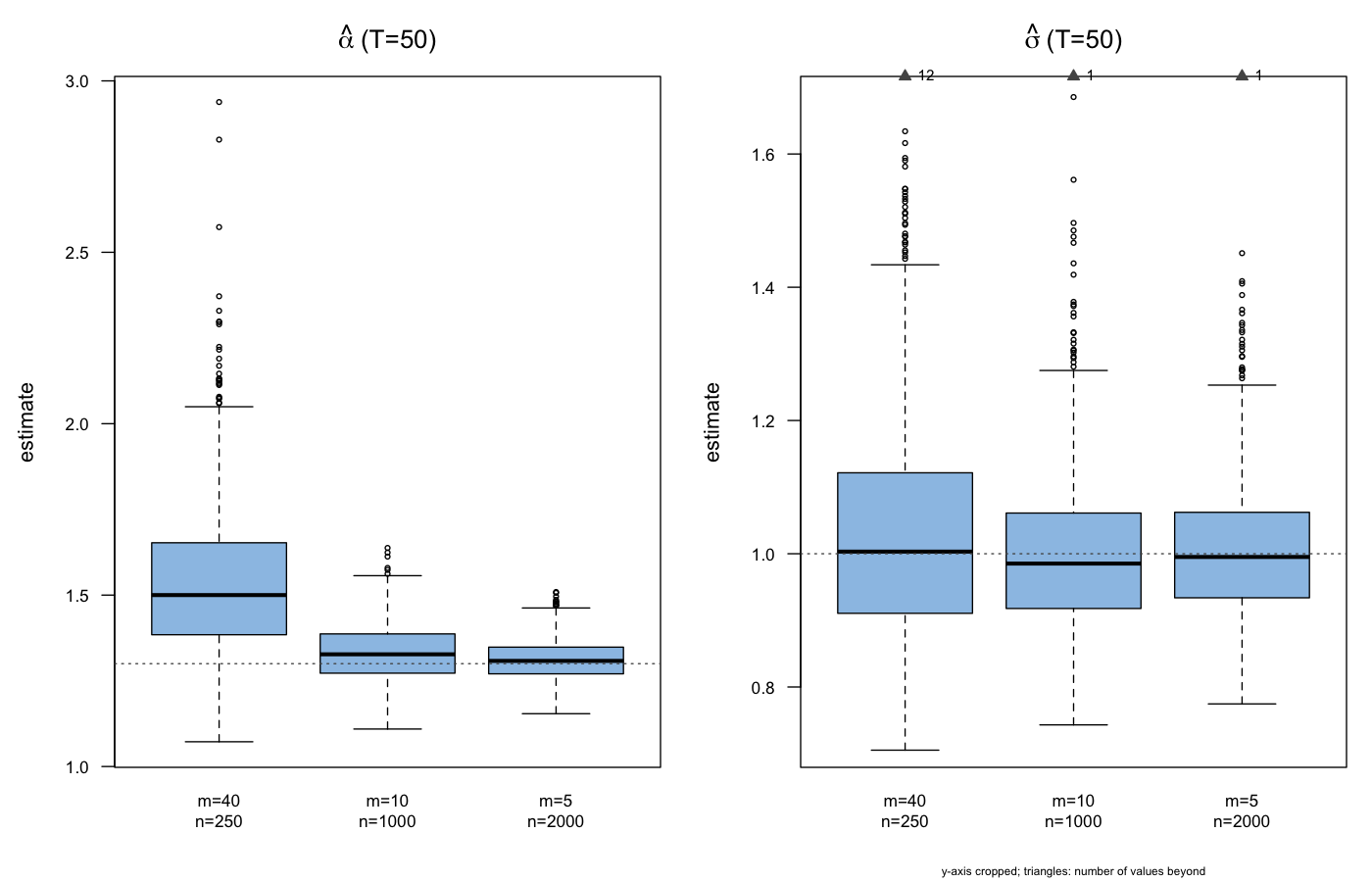}
  \end{minipage}
  \hfill
  \begin{minipage}{0.48\textwidth}
    \centering
    \includegraphics[width=\linewidth]{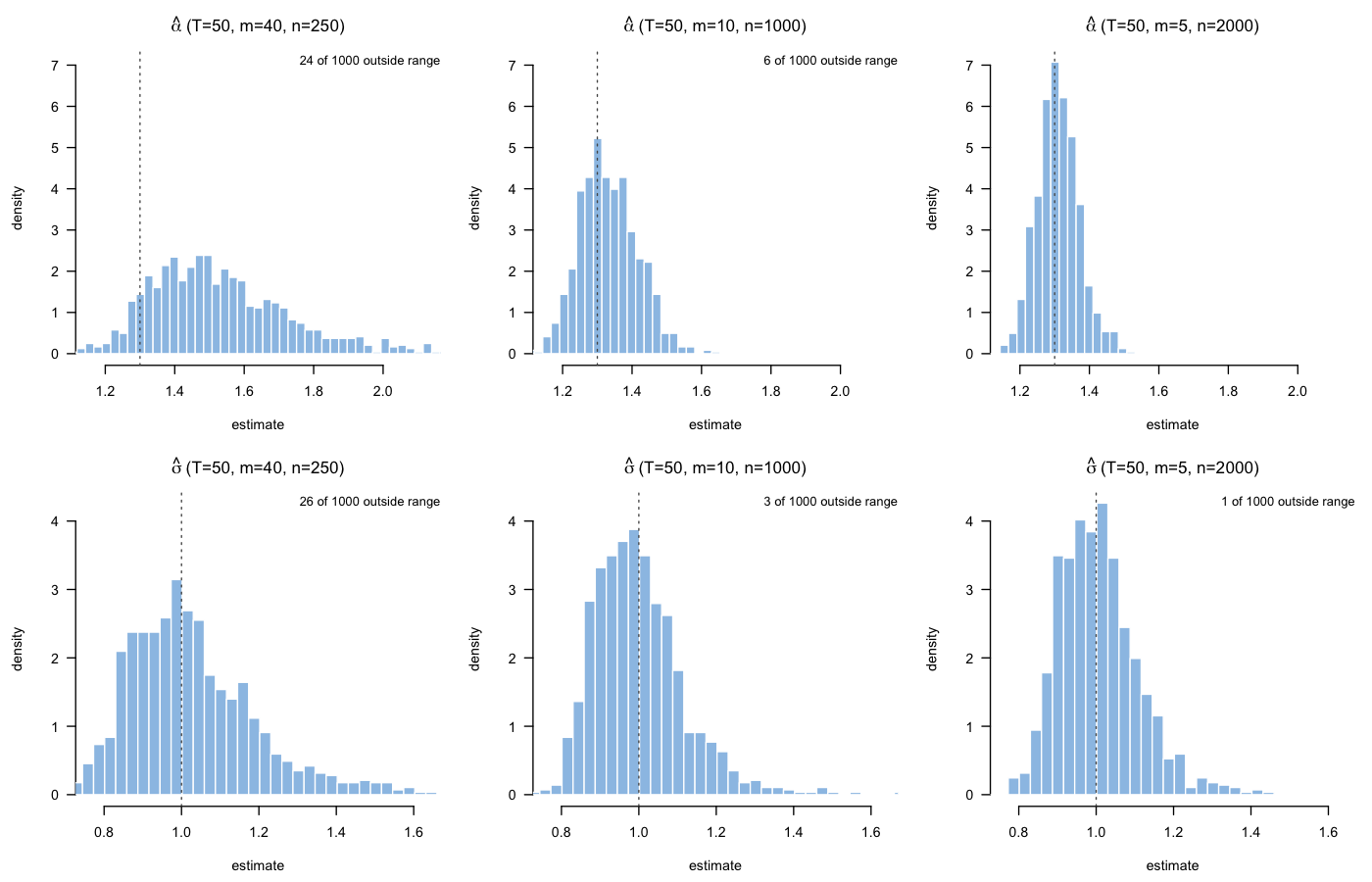}
  \end{minipage}
  \caption{The box plots and histograms of $\hat{\alpha}_n(\rho)$ and $\hat{\sigma}_n(\rho)$ for $T=50$, corresponding to Table \ref{tab:alpha-sigma-nonlinear2-alpha-weighted-T50}.}
  \label{fig:nl-4}
\end{figure}

\begin{figure}[htbp]
  \centering
  \begin{minipage}{0.48\textwidth}
    \centering
    \includegraphics[width=\linewidth]{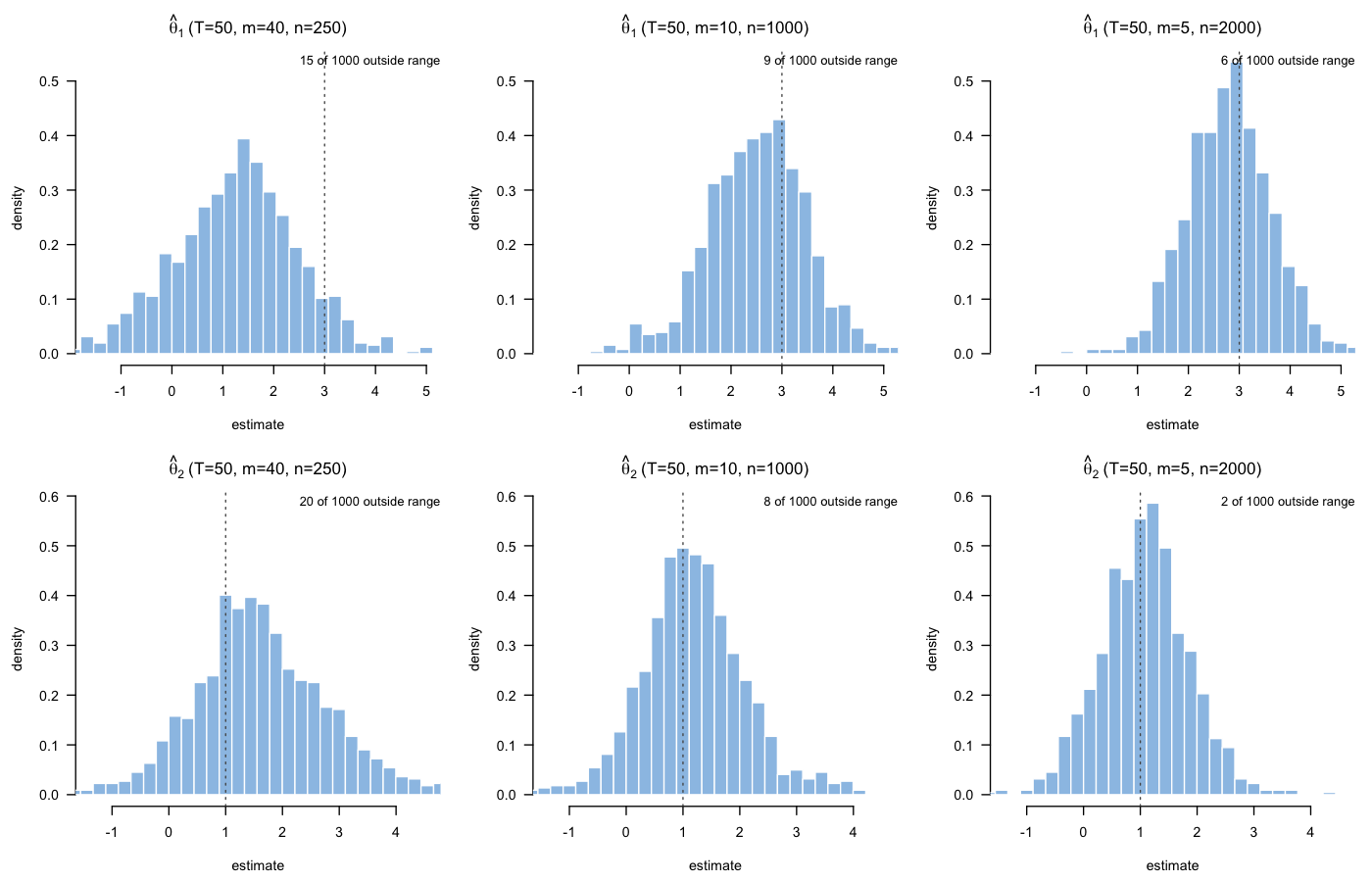}
  \end{minipage}
  \hfill
  \begin{minipage}{0.48\textwidth}
    \centering
    \includegraphics[width=\linewidth]{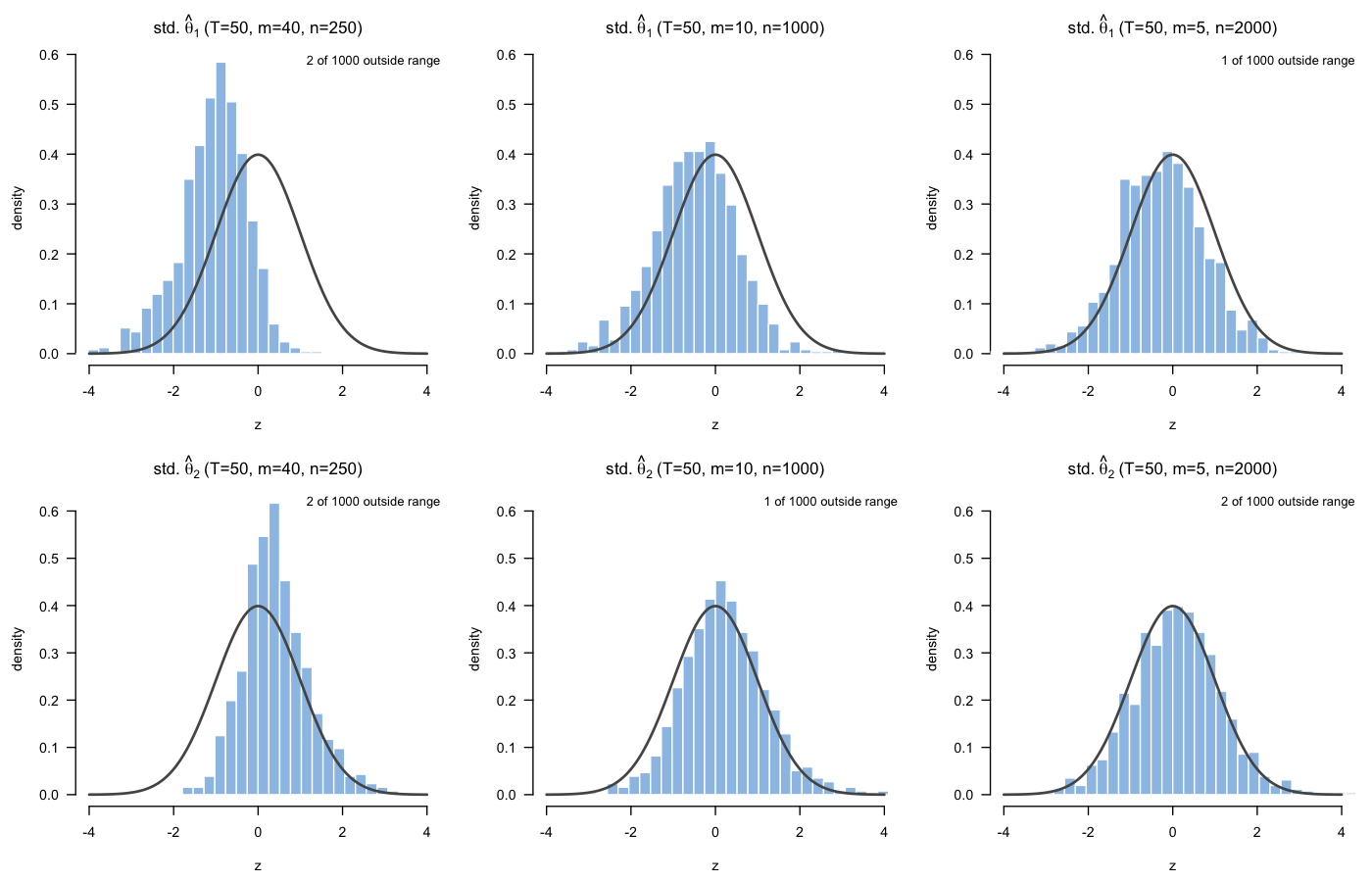}
  \end{minipage}
  \caption{The histograms of the unstandardized and standardized (Studentized) LAD estimators $\tes$ for $T=50$, corresponding to Table \ref{tab:lad-nonlinear2-alpha-weighted-T50}; the solid lines in the right panel show the target standard normal density.}
  \label{fig:nl-5}
\end{figure}

\begin{figure}[htbp]
  \centering
  \includegraphics[width=\textwidth]{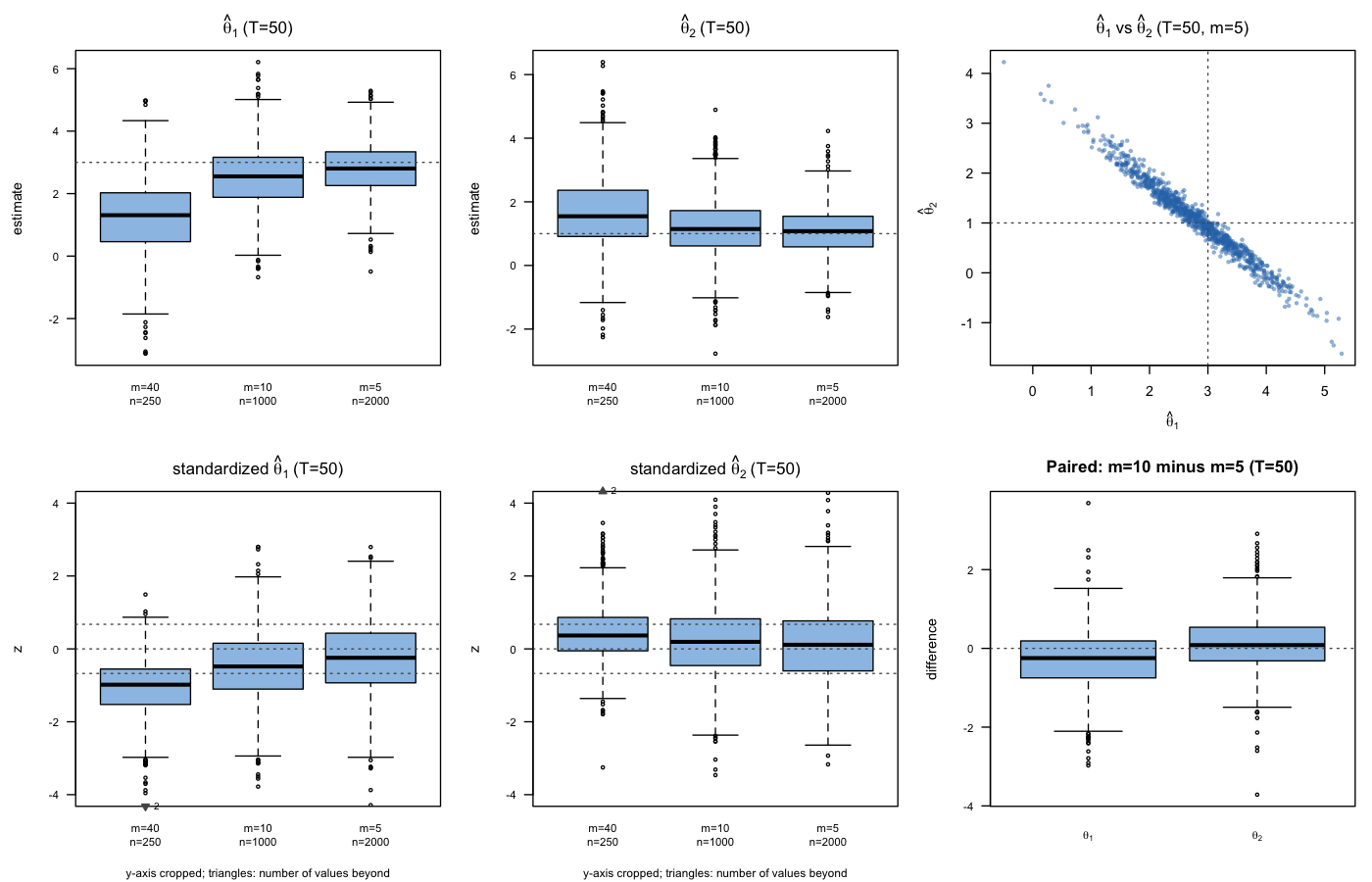}
  \caption{The box plots and histograms of the LAD estimators $\tes$ for $T=50$ (left and middle colomns), corresponding to Table \ref{tab:lad-nonlinear2-alpha-weighted-T50}.
  The upper right panel shows the scatter plot of $\hat{\theta}_{1,n}$ and $\hat{\theta}_{2,n}$.
  The bottom right panel shows the boxplots of the differences of each estimate ($\hat{\theta}_{1,n}$ and $\hat{\theta}_{2,n}$) for $m=10$ and $5$, showing their biases and variations.}
  \label{fig:nl-6}
\end{figure}


\color{black}

\bigskip

\noindent
\textbf{Acknowledgements.} This work was partially supported by JSPS KAKENHI Grant Numbers 23K22410 and 26K02870, and by JST CREST Grant Number JPMJCR2115, Japan (HM).
Claude (Anthropic) was utilized solely as an assistant tool for coding the numerical simulations. The design, execution, and validity of the numerical experiments remain entirely the responsibility of the authors.

\bigskip



\end{document}